\documentclass[a4paper,10pt]{article}
\usepackage{amssymb,amsmath,amsthm}
\usepackage{amsfonts}
\usepackage[english]{babel}
\usepackage{hyperref}
\usepackage{color}
\usepackage{xcolor}

\numberwithin{equation}{section}

\newcommand{\Om}{\Omega}
\newcommand{\la}{\langle}
\newcommand{\ra}{\rangle}

\newcommand{\Op}{\mathrm{Op}\,}
\newcommand{\Lm}{\Lambda}

\renewcommand{\t}{\tau}

\newtheorem{theorem}{Theorem}[section]
\newtheorem{proposition}[theorem]{Proposition}
\newtheorem{lemma}[theorem]{Lemma}

\newtheorem{remark}[theorem]{Remark}
\newtheorem{remarks}[theorem]{Remark}
\newtheorem{definition}[theorem]{Definition}
\newcommand{\be}{\begin{equation}}
\newcommand{\ee}{\end{equation}}

\newcommand{\om}{\omega}

\newcommand{\e}{\varepsilon}

\newcommand{\R}{\mathbb R}
\newcommand{\C}{\mathbb C}

\newcommand{\Z}{\mathbb Z}

\newcommand{\N}{\mathbb N}
\newcommand{\T}{\mathbb T}

\renewcommand{\a }{\alpha }
\renewcommand{\b }{\beta }

\newcommand{\ii }{{\rm i} }
\renewcommand{\d }{\delta }

\newcommand{\g }{\gamma}

\newcommand{\vphi}{\varphi }

\newcommand{\mN}{\mathcal{N}}

\newcommand{\mG}{\mathcal G}
\newcommand{\mK}{\mathcal K}
\newcommand{\mU}{\mathcal U}

\newcommand{\ph}{\varphi}

\newcommand{\mH}{\mathcal{H}}

\newcommand{\mZ}{\mathcal{Z}}
\newcommand{\mL}{\mathcal{L}}

\newcommand{\mA}{\mathcal{A}}
\newcommand{\mR}{\mathcal{R}}
\newcommand{\mF}{\mathcal{F}}

\newcommand{\mE}{\mathcal{E}}
\newcommand{\mS}{\mathcal{S}}
\newcommand{\mQ}{\mathcal{Q}}
\newcommand{\mP}{\mathcal{P}}

\newcommand{\pa}{\partial}
\newcommand{\ompaph}{\om \cdot \partial_\ph}

\newcommand{\odd}{\text{odd}}
\newcommand{\even}{\text{even}}

\def\ba{\begin{aligned}}
\def\ea{\end{aligned}}
\def\beginm{\begin{multline}}
\def\endm{\end{multline}}

\newcommand{\mB}{\mathcal{B}}

\newcommand{\mC}{\mathcal{C}}

\DeclareMathOperator{\diag}{diag}

\DeclareMathOperator{\curl}{curl}
\newcommand{\grad}{\nabla}

\let\div\undefined
\DeclareMathOperator{\div}{div}
\newcommand{\lm}{\lambda}
\DeclareMathOperator{\Mat}{Mat}
\newcommand{\Id}{\mathrm{Id}}
\newcommand{\op}{\mathrm{op}}
\newcommand{\HS}{\mathrm{HS}}

\begin{document}

\title{{\bf Quasi-periodic incompressible Euler flows in 3D}}

\date{}

\author{Pietro Baldi, Riccardo Montalto}

\maketitle

\noindent
{\bf Abstract.}
We prove the existence of time-quasi-periodic solutions 
of the incompressible Euler equation 
on the three-dimensional torus $\T^3$, 
with a small time-quasi-periodic 
external force. 
The solutions are perturbations of constant (Diophantine) vector fields, 
and they are constructed by means of normal forms and KAM techniques 
for reversible quasilinear PDEs.

\smallskip 





\tableofcontents

\section{Introduction}\label{introduction}

We consider the Euler equation for an incompressible fluid on the three-dimensional torus $\T^3$, $\T := \R / 2 \pi \Z$,
\begin{equation}\label{Eulero1}
\begin{cases}
\partial_t U + U \cdot \nabla U + \nabla p = \e f(\omega t, x) \\
\div U = 0
\end{cases} 
\end{equation}
where $\e \in (0, 1)$ is a small parameter, $\omega \in \R^\nu$ is a Diophantine $\nu$-dimensional vector, 
the external force $f$ belongs to ${\cal C}^q(\T^\nu \times \T^3, \R^3)$ 
for some integer $q > 0$ large enough, 
$U = (U_1, U_2, U_3) : \R \times \T^3 \to \R^3$ is the velocity field, 
and $p : \R \times \T^3 \to \R$ is the pressure. 
We look for time-quasi-periodic solutions of \eqref{Eulero1}, 
oscillating with time frequency $\omega$. 
This leads to solve the 
equation 
\begin{equation}\label{Eulero2}
\begin{cases}
\omega \cdot \partial_\vphi U + U \cdot \nabla U + \nabla p = \e f(\vphi, x) \\
\div U = 0
\end{cases}
\end{equation}
where the unknown velocity field $U : \T^\nu \times \T^3 \to \R^3$ 
and the unknown pressure $p : \T^\nu \times \T^3 \to \R$ 
are functions of $(\ph,x) \in \T^\nu \times \T^3$. 
We look for solutions which are small perturbation of a constant vector field $\zeta \in \R^3$, namely we look for solutions of the form 
$$
U (\vphi, x ) = \zeta +  u (\vphi, x) \quad \text{with} \quad \div u = 0\,.
$$
Plugging this ansatz into the equation, one is led to solve 
\begin{equation}\label{Eulero3}
\begin{cases}
\omega \cdot \partial_\vphi u + \zeta \cdot \nabla u + u \cdot \nabla u + \nabla p 
= \e f(\vphi, x) \\
\div u  = 0.
\end{cases}
\end{equation}
We assume that the forcing term $f(\vphi, x)$ is an odd function of the pair $(\ph,x)$, 
namely 
\begin{equation}\label{ipotesi forzante}
f(\vphi, x) = - f(- \vphi, - x) \quad \ 
\forall (\vphi, x) \in \T^\nu \times \T^3\,. 
\end{equation} 
We look for solutions $(u, p)$ of \eqref{Eulero3} that are even functions of the pair $(\vphi, x)$, namely 
\begin{equation}\label{soluzioni formulazione originale euler}
u(\vphi, x) = u(- \vphi, - x), \quad p(\vphi, x) = p(- \vphi, - x) 
\quad \ \forall (\vphi, x) \in \T^\nu \times \T^3\,. 
\end{equation}  
For any real $s \geq 0$, we consider the Sobolev spaces of real scalar 
and vector-valued functions of $(\ph,x)$
\begin{align} 
H^s = H^s(\T^{\nu + 3}, \R^n) 
& := \Big\{ u(\vphi, x) 
= \sum_{(\ell, j) \in \Z^{\nu + 3}} \widehat u(\ell, j) e^{\ii (\ell \cdot \vphi + j \cdot x)} 
: \| u\|_s := \Big(  \sum_{(\ell, j) \in \Z^{\nu + 3}}  
\langle \ell, j \rangle^{2 s} |\widehat u(\ell, j)|^2\Big)^{\frac12} < \infty \Big\} 
\notag \\
H^s_0 &:= \Big\{ u \in H^s : \int_{\T^{3}} u(\vphi, x)\, d x= 0 \Big\}\,, 
\qquad 
\langle \ell, j \rangle := \max \{ 1, |\ell|, |j| \}. 
\label{def sobolev}
\end{align}
We fix any bounded open set $\Omega \subseteq \R^\nu \times \R^3$ 
to which the parameters $(\omega, \zeta)$ belong.
The main result of the paper is the following theorem. 

\begin{theorem}\label{main theorem 2}
There exist $q = q(\nu) > 0$, $s = s(\nu) > 0$, 
such that for every forcing term 
$f \in {\cal C}^q(\T^\nu \times \T^3, \R^3)$ satisfying \eqref{ipotesi forzante} 
there exist $\e_0 = \e_0(f, \nu) \in (0, 1)$, 
$C = C(f,\nu) > 0$
such that for every $\e \in (0, \e_0)$ the following holds. 
There exists a Borel set $\Omega_\e \subset \Omega$ 
of asymptotically full Lebesgue measure, 
i.e. $\lim_{\e \to 0} |\Omega \setminus \Omega_\e| = 0$, 
such that for every $(\omega, \zeta) \in \Omega_\e$ 
there exist $u = u(\cdot \, ; \omega, \zeta) \in H^s(\T^{\nu + 3}, \R^3)$
and $p = p(\cdot \, ; \omega, \zeta) \in H^s(\T^{\nu + 3}, \R)$, 
even functions of the pair $(\vphi, x)$, 
that solve equation \eqref{Eulero3}, 
with $\| u \|_s , \| p \|_s \leq C \e^b$ for some $b \in (0,1)$.
\end{theorem}

Theorem \ref{main theorem 2} is deduced from Theorem \ref{main theorem 1} below, 
dealing with the vorticity formulation of the problem, which we now introduce. 

As is well-known, if we take the divergence of the first equation in \eqref{Eulero3}, 
we can determine the pressure $p$ 
in terms of the unknown $u$ and the forcing term $f$, namely we have 
$$
\Delta p + \div ( u \cdot \nabla u ) = \e {\rm div} f(\omega t, x),
$$
whence 
\begin{equation}\label{definizione pressione}
p = \Delta^{- 1} \big[ \e {\rm div} f(\omega t, x) - \div ( u \cdot \nabla u ) \big]\,.
\end{equation}

If we consider the average 
in the space variable $x \in \T^3$ of equation \eqref{Eulero3},
we get the equation
\[
\omega \cdot \partial_\vphi u_0(\vphi) = f_0(\vphi), \quad \text{where} \ \ 
u_0(\vphi) := \frac{1}{(2 \pi)^3} \int_{\T^3} u(\vphi, x)\, d x, \quad 
f_0(\vphi) := \frac{1}{(2 \pi)^3} \int_{\T^3} f(\vphi, x)\, d x\,. 
\]
This equation can be solved by assuming 
that the frequency vector $\om$ is Diophantine, i.e. 
$$
|\omega \cdot \ell| \geq \frac{\gamma}{|\ell |^\tau} \quad \ 
\forall \ell \in \Z^\nu \setminus \{ 0 \},
$$ 
because, by \eqref{ipotesi forzante}, $f$ has zero average in $(\vphi, x)$. 
Hence, without loss of generality, we can assume that $f$ has zero average in space, namely 
\begin{equation}\label{ipotesi media nulla f}
\int_{\T^3} f(\vphi, x)\, dx = 0\,. 
\end{equation}

We define the vorticity
\begin{equation}\label{def vorticita}
v := \nabla \times u 
:= \curl u 
:= \begin{pmatrix} 
\partial_{x_2} u_3 - \partial_{x_3} u_2 \\
\partial_{x_3} u_1 - \partial_{x_1} u_3 \\ 
\partial_{x_1} u_2 - \partial_{x_2} u_1 \end{pmatrix}.
\end{equation}
By taking the curl 
of equation \eqref{Eulero3} 
(recall that $\nabla \times \nabla \Phi = 0$ for any smooth scalar function $\Phi$), 
as is well known, 
one obtains the equation for the vorticity $v(\vphi, x)$
\[
\omega \cdot \partial_\vphi v + \zeta \cdot \nabla v 
+ u \cdot \nabla v - v \cdot \nabla u = \e F(\vphi, x), 
\quad \ F := \nabla \times f.
\]
We construct an odd solution $v(\vphi, x)$ of the latter equation and we shall prove that there are even functions $\big(u(\vphi, x),  p(\vphi,x) \big)$ which solves the equation \eqref{Eulero3}. Since $\div u = 0$, one has 
\[
\curl v = \curl (\curl u) = - \Delta u;
\]
therefore, if the space average of $u$ is zero, then 
\[
u = (-\Delta)^{-1} \curl v,
\]
and one has the vorticity equations
\begin{equation}\label{equazione vorticita}
\begin{cases}
\omega \cdot \partial_\vphi v + \zeta \cdot \nabla v 
+ u \cdot \nabla v - v \cdot \nabla u = \e F(\vphi, x), \quad F := \nabla \times f, \\
u = \nabla \times \big[(- \Delta)^{- 1} v \big],
\end{cases}
\end{equation}
where $(-\Delta)^{-1}$ is the Fourier multiplier of symbol 
$|\xi|^{-2}$ for $\xi \in \Z^3$, $\xi \neq 0$, 
and zero for $\xi = 0$, namely 
\begin{equation} \label{def Delta inv}
u(x) = \sum_{\xi \in \Z^3} \hat u(\xi) \, e^{\ii \xi \cdot x}
\quad \Rightarrow \quad 
(- \Delta)^{-1} u(x) 
= \sum_{\xi \in \Z^3 \setminus \{ 0 \}} \frac{1}{|\xi|^2} \hat u(\xi) \,  \, e^{\ii \xi \cdot x}.
\end{equation}
The divergence of $v$ and the space average of $v$ are both zero, 
because $v$ is a curl. 
We will prove that  
if $v$ solves \eqref{equazione vorticita} then 
$u := \nabla \times \big(\Delta^{- 1} v \big)$ solves \eqref{Eulero2}.

We define the spaces of even/odd functions of the pair $(\ph,x)$
\begin{equation}\label{funzioni pari dispari}
\begin{aligned}
X & := \big\{  h \in L^2(\T^{\nu + 3}, \R^3) : h(\vphi, x) = h(- \vphi, - x) \big\}, 
\\ 
Y & := \big\{ h \in L^2(\T^{\nu + 3}, \R^3) : h(\vphi, x) = - h(- \vphi, - x) \big\}
\end{aligned}
\end{equation}
and 
\begin{equation}\label{definizione proiettore media spazio tempo}
\Pi_0 h := \frac{1}{(2 \pi)^{ 3}} \int_{\T^{ 3}} h(\vphi, x)\, d x, 
\qquad 
\Pi_0^\bot := {\rm Id} - \Pi_0\,. 
\end{equation}
We will prove that it is enough to look for smooth solutions $v(\vphi, x) \in Y$ with zero average that solve the projected equation 
\begin{equation}\label{equazione vorticita media nulla}
\begin{cases}
\Pi_0^\bot \big( \omega \cdot \partial_\vphi v + \zeta \cdot \nabla v 
+ u \cdot \nabla v - v \cdot \nabla u \big) = \e F(\vphi, x), 
\quad F := \nabla \times f, \\
u = \nabla \times \big[(- \Delta)^{- 1} v \big].
\end{cases}
\end{equation}



It is convenient to replace $(-\Delta)^{-1}$ with an extension of it 
that is invertible also on functions with nonzero space average; 
with $\Pi_0$ defined by \eqref{definizione proiettore media spazio tempo}, 
we then define
\begin{equation} \label{def pi0 Lm Lm inv}
\Lm := \Pi_0 - \Delta, 
\quad \ \  
\Lm^{-1} := \Pi_0 + (-\Delta)^{-1},
\end{equation}
so that $\Lm, \Lm^{-1}$ are Fourier multipliers of symbols, respectively, 
$|\xi|^2$, $|\xi|^{-2}$ for $\xi \in \Z^3 \setminus \{ 0 \}$, 
and $1$ for $\xi = 0$. 
Thus $\Lm \Lm^{-1} u = u$ for all periodic functions $u$ 
(of course, introducing smooth cutoff functions, 
these symbols can be extended to all $\R^3$ 
by preserving the property $\Lm \Lm^{-1} = {\rm Id}$).

%


We introduce the rescaling $v = \tilde\e \tilde v$, 
$\tilde \e := \e^{1/2}$, 
and write \eqref{equazione vorticita media nulla} 
(with $(-\Delta)^{-1}$ replaced by $\Lm^{-1}$)
in terms of $\tilde\e, \tilde v$, namely 
\begin{equation}\label{2803.1}
\Pi_0^\bot ( \omega \cdot \partial_\vphi \tilde v 
+ \zeta \cdot \nabla \tilde v 
+ \tilde \e \tilde u \cdot \nabla \tilde v 
- \tilde \e \tilde v \cdot \nabla \tilde u ) 
= \tilde \e F(\vphi, x), 
\quad \tilde u = \curl (\Lm^{-1} \tilde v).
\end{equation}
We will show the existence of solutions of \eqref{2803.1} 
with a Nash-Moser approach, 
by finding zeros of the nonlinear operator 
(after dropping all the tilde)
$$
{\cal F} : H^{s + 1}_0(\T^{\nu + 3}, \R^3) \cap Y \to H^s_0(\T^{\nu + 3}, \R^3) \cap X
$$
defined by 
\begin{equation}\label{equazione cal F vorticita}
\begin{aligned}
{\cal F}(v) & := 
\omega \cdot \partial_\vphi v + \zeta \cdot \nabla v + \e \Pi_0^\bot 
\big[ {\cal U}(v) \cdot \nabla v - v \cdot \nabla {\cal U}(v) - F(\vphi, x) \big], 
\qquad  
{\cal U}(v) 
:= \curl (\Lm^{- 1} v).
\end{aligned}
\end{equation}
We 
consider parameters $(\omega, \zeta)$ in a bounded open set 
$\Omega \subseteq \R^\nu \times \R^3$; 
we will use such parameters along the proof 
in order to impose appropriate non resonance conditions. 
Now we state precisely the main results of the paper. 

\begin{theorem} \label{main theorem 1}
There exist $q = q(\nu) > 0$, $s = s(\nu) > 0$, such that for every forcing term $f \in {\cal C}^q(\T^\nu \times \T^3, \R^3)$ satisfying \eqref{ipotesi forzante} there exists $\e_0 = \e_0(f, \nu) \in (0, 1)$ such that for every $\e \in (0, \e_0)$ the following holds. There exists a Borel set $\Omega_\e \subset \Omega$ of asymptotically full Lebesgue measure, i.e. $\lim_{\e \to 0} |\Omega \setminus \Omega_\e| = 0$ such that for any $(\omega, \zeta) \in \Omega_\e$ there exists $v(\cdot; \omega, \zeta) \in H^s_0(\T^{\nu + 3}, \R^3)$, $v= {\rm odd}(\vphi, x)$ such that ${\cal F}(v) = 0$. Moreover $\sup_{(\omega, \zeta) \in \Omega_\e} \| v(\cdot; \omega, \zeta) \|_s \leq C \e^a$ for some constants $C = C(\nu) > 0$ and $a = a(\nu) \in (0, 1)$.  
\end{theorem}

From Theorem \ref{main theorem 1} we will deduce 
(see Section \ref{sezione teoremi principali})
the existence of $u(\vphi, x)$, $p(\vphi, x)$ 
solving \eqref{Eulero3}.

\bigskip

\noindent
\emph{Related literature: Euler equations.}
The celebrated Euler equation,  
one of the most important 
mathematical models from fluid dynamics,
has been extensively studied in the last century;  
we refer, for example, to the book \cite{Majda-Bertozzi} 
and the survey \cite{Constantin}
for a rich introduction to the subject 
with a detailed overview of its vast literature 
and its interesting open problems.


Starting from the Seventies, the Cauchy problem for the incompressible Euler equation 
has been analyzed by many authors. 
Without even trying to be exaustive, we mention just few results 
(see e.g.\ \cite{Majda-Bertozzi} for much more references).
The local-in-time existence of smooth solutions for the incompressible Euler equation 
has been established by Kato in \cite{Kato1}, \cite{Kato2}. 
Concerning the \emph{global} existence of smooth solutions, 
the situation is very different in the two-dimensional (2D) case 
and in the three dimensional (3D) one. 
Indeed, the Beale-Kato-Majda criterion \cite{BKM} says that 
if a smooth solution of the Euler equation on a time interval $[0, T_*)$
cannot be continued to $T_*$, 
then its vorticity $v(t,x)$ satisfies 
$\int_0^{T_*} \| v(t, \cdot) \|_{L^\infty_x}\, d t = \infty$; 
viceversa, a priori estimates on that vorticity integral 
implies the existence of global smooth solutions. 
In the 2D case, the vorticity is conserved along the fluid particle trajectories; 
this implies the identity $\| v(t, \cdot ) \|_{L^\infty_x} = \| v(0, \cdot) \|_{L^\infty_x}$ 
and, therefore, the existence of a unique smooth global solution
(by Beale-Kato-Majda criterion).
We also mention that for the 2D Euler equation 
Yudovich \cite{Yu} proved the existence of global weak solutions 
when the initial vorticity is in the space $L^\infty_x \cap L^1_x$, 
while ill-posedness for strong solutions of borderline Sobolev regularity 
has been recently proved by Bourgain and Li \cite{Bourgain Li 2015 Invent};
we also refer to \cite{Bourgain Li 2015 Invent} for other references 
on ill-posedness in both 2D and 3D. 
We mention the very recent work of Elgindi \cite{Elgindi} 
about singularity formation in finite time 
for solutions $C^{1,\a}_{t,x}$ with spatial decay.  

In 3D, 
the vorticity equation contains an additional term, called ``stretching term''. 
Whether smooth solutions of the 3D Euler (and also Navier-Stokes) equation blow up in finite time or exist globally in time is one of the most famous 
open problems in the theory of PDEs. 
Global {\it weak} solutions in 3D 
have been investigated starting form the Ninetiees. 
Schnirelman \cite{Shn} constructed global and distributional solutions 
in $L^2_{t,x}$; 
De Lellis and Sz\'ekelyhidi in \cite{Camillo1} 
constructed global weak solutions in $L^\infty_{t,x}$ 
with compact support;  
then these results have been followed by several important improvements, 
see e.g.\ \cite{Buckmaster ecc 2019 CPAM} and \cite{Isett}. 


Regarding time quasi-periodic solutions of the Euler equation, 
to the best of our knowledge the only result in literature 
is the 2D one of Crouseilles and Faou \cite{Faou}, 
where the scalar vorticity formulation allows the construction 
of quasi-periodic solutions as localized travelling profiles of special form
with compact support, avoiding the small divisors problem 
one usually encounters in KAM results.
On the other hand, KAM techniques have been used by 
Khesin, Kuksin and Peralta-Salas in \cite{Kuksin ecc 2014 Adv Math}, 
\cite{Kuksin ecc 2019} to obtain non-mixing results 
for the 3D Euler equation around steady flows.

\medskip

In this paper, using normal form and KAM (Kolmogorov-Arnold-Moser) techniques, 
we construct smooth (strong) solutions of the forced 3D Euler equations 
that are quasi-periodic in time; as a consequence, in particular, 
they are global in time. This is one of the few KAM results 
for {\it quasilinear} PDEs in {\it higher space-dimension}. 

\medskip

\noindent
\emph{Related literature: KAM for quasilinear PDEs.}
The existence of time-periodic and quasi-periodic solutions of PDEs 
(which is often referred to as ``KAM for PDEs'' theory) 
started in the late 1980s with the pioneering papers of Kuksin \cite{K87}, Wayne \cite{Wayne} and Craig-Wayne \cite{CW}; 
we refer to the recent review \cite{Berti-BUMI-2016} of Berti 
for a general presentation of the theory, of its protagonists and its state-of-the-art. 
Many PDEs arising from fluid dynamics are fully nonlinear or quasi-linear equations, 
namely equations where the nonlinear part contains as many derivatives 
as the linear part. 
The key idea in order to deal with these kind of PDEs has been introduced 
by Iooss, Plotnikov and Toland \cite{IPT} 
in the problem of finding periodic solutions for the water waves equation. 
Using a Nash-Moser iteration to overcome the small divisors problem, 
their strategy to solve the linearized equation at any approximate solution 
relies on a \emph{normal form} procedure based on \emph{pseudo-differential calculus}. 
This approach to quasilinear and fully nonlinear PDEs with small divisors 
has been further developed in \cite{Baldi Benjamin-Ono}, 
\cite{Alazard Baldi} for periodic solutions, 
and it has been successfully combined with a {\it KAM reducibility} procedure 
to develop a general method for 1D problems
for the construction of \emph{quasi-periodic} solutions 
of quasilinear and fully nonlinear PDEs 
and for the analysis of the dynamics of linear PDEs 
with time-quasi-periodic unbounded potentials, 
see \cite{BBM-Airy}, \cite{BBM-auto}, 
\cite{Berti-Montalto}, \cite{BBHM}, \cite{Bam17}, \cite{Bam18}; 
see \cite{Berti-BUMI-2016} for a more complete list of references. 
We recall that, in this context, 
a linear operator is said to be {\it reducible} 
if there exists a change of variables, bounded on Sobolev spaces,
that conjugates it to a diagonal (or {\it block-diagonal}) operator. 

The extension of KAM reducibility results to higher space dimension $d > 1$
is a difficult matter. 
The first reducibility result in 
higher dimension has been obtained by Eliasson and Kuksin in \cite{EK} 
for the linear Schr\"odinger equation with a bounded analytic potential. 
Their proof is strongly based on the fact that the eigenvalues of the Laplacian on $\T^d$ 
are separated and hence some kind 
of {\it second order Melnikov conditions} 
(lower bounds on differences of the eigenvalues) can be imposed {\it block-wise}. 

By extending the Craig-Wayne method \cite{CW}, 
Bourgain \cite{B} and then Berti-Bolle \cite{BB1}, \cite{BB2}, 
Berti-Corsi-Procesi \cite{BCP} 
proved the existence of invariant tori for nonlinear wave (NLW) 
and Schr\"odinger (NLS) equations with bounded perturbations 
in higher space dimension. 
These results are based on a technique called {\it multiscale analysis}, 
which allows to solve the linearized equations arising in the Nash-Moser iterative procedure 
by imposing suitable lower bounds,  
called {\it first order Melnikov conditions}, 
on its eigenvalues.

Extending KAM theory to PDEs with unbounded perturbations in higher space dimension
is one of the 
open problems in the field.
A natural strategy is to try to extend the normal form methods 
based on pseudo-differential calculus developed in 1D 
in \cite{BBM-Airy}, \cite{BBM-auto}, \cite{Berti-Montalto}, \cite{BBHM}, 
\cite{Bam17}, \cite{Bam18}. 
Up to now, this has been achieved only in few examples, 
namely the Kirchhoff equation \cite{Mon}, \cite{CorsiMontalto}, 
the non-resonant transport equation \cite{FGMP}, \cite{BLM} 
and the quantum harmonic oscillator on $\R^d$ 
and Zoll manifolds \cite{BGMR2}, \cite{BGMR1}.  

\medskip
\subsection{Description of the strategy}
In the present paper, to overcome the small divisors problem, 
we construct invariant tori of the Euler equation by means of a Nash-Moser iteration. Therefore, the core of the paper is the analysis of the linearized operators (see  \eqref{operatore linearizzato}) arising in the Nash-Moser scheme, performed in Sections \ref{sezione riduzione ordine alto}-\ref{sezione riducibilita a blocchi}. 
The strategy is somehow related to the one developed in \cite{BLM}, since the linearized Euler equation is a linear transport-like equation with a small, quasi-periodic in time perturbation of order one. On the other hand the procedure developed in \cite{BLM} does not apply. The main reason is that in this case we deal with a {\it vector} transport operator of the form  
\begin{equation}\label{linearizzato introduzione}
h = (h_1, h_2, h_3) \mapsto {\cal L} h := \omega \cdot \partial_\vphi h + \zeta \cdot \nabla h + \e a(\vphi, x) \cdot \nabla h + \e {\cal R}(\vphi) h
\end{equation}
where ${\cal R}(\vphi) = \big({\rm Op}(r_{i j}(\vphi, x, \xi)) \big)_{i, j = 1,2,3}$ is a $3 \times 3$ {\it matrix-valued} pseudo-differential operator of order $0$,
whereas in \cite{BLM} the transport operator to normalize is {\it scalar}.  
In order to invert the operator ${\cal L}$ in \eqref{linearizzato introduzione}, we construct a normal form procedure which reduces the operator \eqref{linearizzato introduzione} to a $3 \times 3$ block-diagonal operator of the form 
$$
\omega \cdot \partial_\vphi h + \zeta \cdot \nabla h + {\rm Op}(Q(\xi)) h
$$
where the $3 \times 3$ matrix symbol $Q(\xi) \in {\rm Mat}_{3 \times 3}$ 
satisfies $\| Q(\xi) \|_{\HS} \lesssim \e \langle \xi \rangle^{- 1}$ 
where the norm $\| \cdot \|_{\HS}$ is the standard Hilbert-Schmidt norm of the matrices. 
Note that we obtain only a $3 \times 3$ block-diagonalization. 
This is due to the fact that the unperturbed operator 
(which is \eqref{linearizzato introduzione} for $\e = 0$) 
has eigenvalues of multiplicity $3$, namely 
for every $j \in \Z^d \setminus \{ 0 \}$
the functions 
\begin{equation} \label{3 eigenfunctions}
\begin{pmatrix} e^{\ii j \cdot x} \\ 0 \\ 0 \end{pmatrix},
\quad 
\begin{pmatrix} 0 \\ e^{\ii j \cdot x} \\ 0 \end{pmatrix},
\quad 
\begin{pmatrix} 0 \\ 0 \\ e^{\ii j \cdot x} \end{pmatrix}
\end{equation}
are orthogonal eigenfunctions in $L^2(\T^3, \C^3)$ 
corresponding to the eigenvalue $\ii \zeta \cdot j$.  

The fact that we deal with $3 \times 3$ matrix-valued pseudo-differential operators 
is actually the main technical difficulty of the paper. 
The point is that, if we take two matrix-valued pseudo-differential operators ${\cal A} = {\rm Op}(A(\vphi, x, \xi)), {\cal B} = {\rm Op}(B(\vphi, x, \xi))$ 
of order $m, m'$ respectively, it is not true (unlike for scalar pseudo-differential operators) 
that the commutator $[{\cal A}, {\cal B}]$ gains one derivative, 
namely it is not of order $m + m' - 1$. 
Indeed the principal symbol of the commutator is given by the commutator of two $3 \times 3$ matrices $[A(\vphi,x, \xi), B(\vphi, x, \xi)]$, which, in general, is \emph{not} zero. This difficulty appears at the normal form step which allows to eliminate the zeroth order term. 
In fact, at the highest order term (which is $\omega \cdot \partial_\vphi 
+ (\zeta + \e a(\vphi, x)) \cdot \nabla$)  
the linearized operator ${\cal L}$ acts in a diagonal way with respect to 
the three components $(h_1, h_2, h_3)$, 
whereas at the zeroth order term the dynamics on these components is strongly coupled.  
This implies that, in order to reduce to constant coefficients the zeroth order term of the linearized operator, we need to solve a {\it variable coefficients homological equation}, 
see the equation \eqref{equazione omologica grado 0 introduzione} below. 

In the reduction of the zeroth order term we use the \emph{reversible structure} 
of the Euler equation: working with functions with even/odd parity in the pair $(\ph,x)$ 
eliminates some average terms that would be an obstruction to the reduction procedure. 
By reversibility, the reduction to constant coefficients of the zero order term 
corresponds to its complete cancellation. 

Once the zeroth order term has been removed, 
in order to reduce to constant coefficients also the lower order terms (starting from the order $- 1$), we use the fact that the perturbation to normalize is at least one-smoothing. 
The homological equations arising along the procedure have constant coefficients 
(see \eqref{eq omologica ordini bassi intro}) and the solutions are of the same order as the remainders we want to normalize: this implies that they {\it gain derivatives}. 
This gain of regularity replaces the gain of derivatives 
that, in the scalar case, is given by the gain of one derivative of commutators. 

Now we describe in more details all the steps of our reduction procedure. 

\begin{itemize}
\item{\bf Reduction of the highest order term.} 
As already said, at the highest order term the operator ${\cal L}$ acts in a diagonal way 
on the three components $(h_1, h_2, h_3)$. 
Hence, in order to reduce to constant coefficients the highest order term, 
it is enough to diagonalize the transport 
operator 
\begin{equation}\label{operatore trasporto introduzione}
{\cal T} := \omega \cdot \partial_\vphi + ( \zeta + \e a(\vphi, x) ) \cdot \nabla\,. 
\end{equation}
This is the content of Proposition \ref{proposizione trasporto}, 
whose proof follows 
\cite{FGMP}. 
The only difference is that here, since the vector field $a(\vphi, x)$ has zero space average 
and zero divergence, we conjugate ${\cal T}$ to the operator $\omega \cdot \partial_\vphi + \zeta \cdot \nabla$, provided the vector $(\omega, \zeta) \in \R^{\nu + 3}$ is Diophantine 
(see \eqref{def cantor set trasporto}), 
whereas in \cite{FGMP} the operator \eqref{operatore trasporto introduzione} is conjugated 
to a constant coefficients operator of the form 
$\omega \cdot \partial_\vphi + \mathtt m(\omega, \zeta) \cdot \nabla$ 
with constant vector field $\mathtt m(\omega, \zeta) = \zeta + O(\e)$. 

Then in Lemma \ref{lemma coniugio cal L (0)} we prove 
that the remainder $\e {\cal R}$ (see \eqref{linearizzato introduzione}) is conjugated, 
by means of the reversibility preserving invertible map ${\cal A}$ constructed in Proposition \ref{proposizione trasporto}, 
to another reversible operator of order zero 
which has the form ${\rm Op}(R_0^{(1)}) + {\rm Op}(R_{- 1}^{(1)})$ 
where 
$R_i^{(1)}$ is a matrix-valued symbol of order $i$, $i=0,-1$.
We prove that the zeroth order term $R_0^{(1)}$ 
satisfies the symmetry condition $R_0^{(1)}(\vphi, x, \xi) = R_0^{(1)}(\vphi, x, - \xi)$. 
This condition, together with the reversibility 
(which, 
for the symbol, becomes $R_0^{(1)}(\vphi, x, \xi) = - R_0^{(1)}(- \vphi, - x, - \xi)$), 
allows to perform the normal form step at the zeroth order term. 

\item{\bf Reduction of the zeroth order term.} 
In order to eliminate the zeroth order term ${\rm Op}(R_0^{(1)})$ from the operator ${\cal L}^{(1)}$ defined in \eqref{cal L (1)}, we conjugate such an operator by means of the transformation ${\cal B} = {\rm Id} + {\rm Op}(M(\vphi,x, \xi))$ where the zeroth order symbol 
$M(\vphi, x, \xi)$ has to satisfy the {\it variable coefficients homological equation} 
\begin{equation}\label{equazione omologica grado 0 introduzione}
\big( \omega \cdot \partial_\vphi + \zeta \cdot \nabla + R_0^{(1)}(\vphi, x, \xi) \big) 
M(\vphi, x, \xi) + R_0^{(1)}(\vphi,x, \xi) = 0\,. 
\end{equation}
This equation is solved in Section \ref{sez eq omologica grado 0 coeff variabili}. 
The main point is to transform the operator 
$\omega \cdot \partial_\vphi + \zeta \cdot \nabla + R_0^{(1)}(\vphi, x, \xi)$ (acting on the space of symbols ${\cal S}^0_{s, 0}$ cf. Definition \ref{definizione norma pseudo diff})
into the operator $\omega \cdot \partial_\vphi + \zeta \cdot \nabla$,
for all Diophantine vectors $(\omega, \zeta) \in DC(\gamma, \tau)$. 
This is made by means of the iterative scheme 
of Lemma \ref{prop riducibilita trasporto vettoriale semilin}, 
using the property that 
the symmetry conditions \eqref{simmetrie Vn iterazione} are preserved along the iteration. 
This means that the space-time average of the remainders is always zero 
and therefore there are no corrections to the {\it normal form operator} $\omega \cdot \partial_\vphi + \zeta \cdot \nabla$. 
We finally get the operator ${\cal L}^{(2)}$ in \eqref{def cal L2} which is a one-smoothing perturbation of the constant coefficients operator $\omega \cdot \partial_\vphi + \zeta \cdot \nabla$. 

The fact that no correction to the constant vector field 
$\omega \cdot \partial_\vphi + \zeta \cdot \nabla$ 
comes from terms of order one and zero in the linearized operator
is due to two different reasons: 
as observed above, the first order term 
gives no corrections because the coefficient $a(\ph,x)$ 
in \eqref{operatore trasporto introduzione}
has zero space average and zero divergence, 
while the zeroth order term gives no corrections 
because Euler equation is reversible and we are working 
in the corresponding invariant subspace 
(namely, where the vorticity is odd in the pair $(\ph,x)$).

\item{\bf Reduction of the lower order terms.} 
In Proposition \ref{proposizione regolarizzazione ordini bassi} we construct 
a reversibility preserving transformation that conjugates the operator ${\cal L}^{(2)}$ in \eqref{def cal L2} to the operator ${\cal L}^{(3)}$ in \eqref{def cal L (3)}, which is 
a regularizing perturbation of arbitrary negative order
of the constant coefficients operator 
$$
\omega \cdot \partial_\vphi + \zeta \cdot \nabla + {\cal Q}.
$$
Here ${\cal Q}$ is a $3 \times 3$ {\it block-diagonal operator} of order $- 1$ of the form 
$$
{\cal Q} h( x) = \sum_{\xi \in \Z^3} Q(\xi) \widehat h(\xi) e^{\ii \xi \cdot x}, 
\quad h \in L^2(\T^3, \R^3)
$$
with 
$$
Q(\xi) \in \Mat_{3 \times 3}(\C), \quad \ 
\| Q(\xi) \|_{\HS} \lesssim \e \langle \xi \rangle^{- 1}.
$$
This is proved iteratively in Lemma \ref{lemma iterativo ordini bassi}. 
In that Lemma, 
the homological equation we solve at each step 
is a {\it constant coefficients equation} of the form 
\begin{equation}\label{eq omologica ordini bassi intro}
(\omega \cdot \partial_\vphi + \zeta \cdot \nabla) M(\vphi, x, \xi) = R(\vphi, x, \xi) - \langle R \rangle_{\vphi, x}(\xi)
\end{equation}
where $\langle R \rangle_{\vphi, x}(\xi)$ is the $(\vphi, x)$-average of the symbol $R$, see \eqref{M n + 1 R n + 1 (2)}. 

Since $R$ is a symbol of order $- n$ with $n \geq 1$, 
for $(\omega, \zeta)$ Diophantine, 
equation \eqref{eq omologica ordini bassi intro} 
has a solution $M$ which is a symbol of order $- n$, 
namely 
the same order as the one we want to normalize. 
At the $n$-th step, since the remainder that we normalize is of order $-(n + 1)$, then also the solution of the equation \eqref{M n + 1 R n + 1 (2)} is of order $-(n + 1)$. This allows to show that the new error term is of order $-(n + 2)$. 

\item{\bf Reducibility.} 
To complete the reduction to constant coefficients of the linearized operator, 
the next step is the reducibility scheme of Section \ref{sezione riducibilita a blocchi}, 
in which we conjugate iteratively the operator ${\cal L}_0$ in \eqref{coniugazione-pre-riducibilita} to a $3 \times 3$, time independent block diagonal operator of the form $\omega \cdot \partial_\vphi + \zeta \cdot \nabla + {\cal Q}_\infty$. 
Here ${\cal Q}_\infty$ is a $3 \times 3$ block diagonal operator ${\rm diag}_{j \in \Z^3 \setminus \{  0 \}} ({\cal Q}_\infty)_j^j$ where the $3 \times 3$ matrices $({\cal Q}_\infty)_j^j$ satisfy $\|({\cal Q}_\infty)_j^j \|_{\HS} \lesssim \e |j|^{- 1}$ for all $j \in \Z^3 \setminus \{ 0 \}$, see Lemma \ref{lemma blocchi finali}. Along the iterative KAM procedure, we need to solve the homological equation \eqref{equazione omologica KAM}. 
In order to solve it (see Lemma \ref{Lemma eq omologica riducibilita KAM}), 
for any $(\ell, j, j') \in \Z^\nu \times (\Z^3 \setminus \{ 0 \}) \times (\Z^3 \setminus \{ 0 \})$, $(\ell, j,j') \neq (0, j,j)$, 
we have to invert the linear operator  
\begin{equation}\label{op eq omologica in intro}
L(\ell,j,j') : {\rm Mat}_{3 \times 3} \to {\rm Mat}_{3 \times 3}, \quad \ 
M \mapsto \ii (\om \cdot \ell + \zeta \cdot (j-j')) M 
+ \mQ_j^j M - M \mQ_{j'}^{j'}
\end{equation}
where ${\cal Q}_j^j$ are time independent $3 \times 3$ matrices. 
Then we impose 
second order Melnikov non-resonance conditions with loss of derivatives 
both in time and in space, 
involving the invertibility of such kinds of operators and suitable estimates for their inverses, see \eqref{insiemi di cantor rid}. 
Note that, since the Hamiltonian structure of the Euler equations 
is not the standard constant one (see \cite{Olver}), 
the $3 \times 3$ blocks ${\cal Q}_j^j$ are, in general, not self-adjoint. 
Hence, to verify that the set of parameters satisfying the 
required non-resonance conditions has a large Lebesgue measure, 
we need to control a sufficiently large number of derivatives 
with respect to the parameter $(\omega, \zeta)$. 
In particular we 
prove that the nineth derivative of the determinant of the 
$9 \times 9$ matrix representing $L(\ell, j, j')$ in \eqref{op eq omologica in intro} is big, 
in order to show that the {\it resonant sets} have small Lebesgue measure (see Lemmata \ref{stima misura risonanti sec melnikov}, \ref{lemma astratto misura risonante}). Since the $3 \times 3$ blocks could be not self-adjoint, we cannot deduce that our solutions are linearly stable. 
\end{itemize}

As a conclusion of this introduction, we remark that the quadratic nonlinearity 
of the Euler equation is already in normal form
(in the sense of homogeneity order, namely as Poincar\'e-Dulac normal form)
because the \emph{dispersion relation} $\lm(\xi) = \ii \zeta \cdot \xi$ 
of the unperturbed operator $\zeta \cdot \grad$ is exactly \emph{linear}. 
This means that the normal form approach, 
which in KAM theory is usually a very efficient way 
of extracting the first contribution to the frequency-amplitude relation 
from the nonlinearity of a PDE, 
for the Euler equation gives no improvement 
with respect to the equation itself. 
This makes it especially difficult to construct
of quasi-periodic solutions for the autonomous (i.e.\ without the external forcing $f$)
Euler equation. 

\medskip

{\sc Acknowledgements.} The authors warmly thank Alberto Maspero and Michela Procesi for many useful discussions and comments. Riccardo Montalto is supported by INDAM-GNFM. 
Pietro Baldi is supported by INdAM-GNAMPA Project 2019.

\section{Norms and linear operators}\label{sez generale norme e operatori}
In this section we collect some general definitions and known results concerning norms, pseudo-differential operators and matrix representation of operators which are used in the whole paper. Subsection \ref{sez coniugi per eulero} deals with some conjugacy properties of ${\rm curl}$, $(- \Delta)^{- 1}$ with changes of variables, required by the analysis of the Euler equation. 

\medskip

\noindent
{\bf Notations.} In the whole paper, the notation $ A \lesssim_{s, m, k_0, \a} B $ means
that $A \leq C(s, m, k_0, \alpha) B$ for some constant $C(s, m, k_0, \alpha) > 0$ depending on 
the Sobolev index $ s $, the constants $ \a, m $ and the index $k_0$ which is the maximal number of derivatives with respect to the parameters $(\omega, \zeta)$ that we need to control along our proof. We always omit to write the dependence on $\nu$ which is the number of frequencies and $\tau$, which is the constant appearing in the non-resonance conditions (see for instance \eqref{def cantor set trasporto}, \eqref{insiemi di cantor rid}). Starting from Section \ref{sez generale L}, we omit to write the dependence on $k_0$ since it is fixed as $k_0 := 11$ in \eqref{definizione finale tau}. Hence, we write $ \lesssim_{s, m, \a}$ instead of $ \lesssim_{s, m, k_0, \a}$. We often write $u = {\rm even}(\vphi, x)$ if $u \in X$ and $u = {\rm odd}(\vphi, x)$ if $u \in Y$ (recall the Definition \eqref{funzioni pari dispari}). 
\subsection{Function spaces and pseudo differential operators}
\label{subsec:function spaces}

We denote by $| \cdot |$ the Euclidean norm of vectors
and by $| \cdot | = \| \cdot \|_{\HS}$ the Euclidean (``Hilbert-Schmidt'') norm 
of matrices:  
if $v \in \C^n$ has components $v_j$,  
and $M \in \Mat_{n \times m}(\C)$ has entries $M_{j,k}$, 
then 
\begin{equation} \label{def Euclidean norm}
|v|^2 := \sum_{j=1}^n |v_j|^2, \quad \ 
\|M\|_{\HS}^2 := \sum_{\begin{subarray}{c} 1 \leq j \leq n \\ 1 \leq k \leq m \end{subarray}} 
|M_{j,k}|^2.
\end{equation}
Let $a : \T^\nu \times \T^3 \to E$, $a = a(\ph,x)$, 
be a function taking values in the space of scalars, or vectors, or matrices, 
namely $E = \C^n$ or $E = \Mat_{n \times m}(\C)$. 
Then, for $s \in \R$, its Sobolev norm $\| a \|_s$ is defined as 
\begin{equation} \label{def Sobolev norm generale}
\| a \|_s^2 := \sum_{(\ell, j) \in \Z^\nu \times \Z^3} 
\langle \ell, j \rangle^{2s} | \widehat a(\ell,j) |^2 ,
\quad \ 
\langle \ell, j \rangle := \max \{ 1, |\ell|, |j| \},
\end{equation}
where $\widehat a(\ell,j) \in E$ (which are scalars, or vectors, or matrices) 
are the Fourier coefficients of $a(\ph,x)$, namely
\[
\hat a(\ell,j) := \frac{1}{(2\pi)^{\nu+3}} \int_{\T^{\nu+3}} 
a(\ph,x) e^{- \ii (\ell \cdot \ph + j \cdot x)} \, d\ph dx,
\]
and $|\hat a(\ell,j)|$ is their norm defined in \eqref{def Euclidean norm}.
We denote 
\[
H^s 
:= H^s_{\ph,x} 
:= H^s(\T^{\nu} \times \T^3) 
:= H^s(\T^{\nu} \times \T^3, E) 
:= \{ u : \T^{\nu} \times \T^3 \to E, \ \| u \|_s < \infty \},
\]
for $E = \C^n$ or $E = \Mat_{n \times m}(\C)$; 
we write, in short, $H^s$ both for vectors and for matrices.

In the paper we use Sobolev norms for 
(real or complex, scalar- or vector- or matrix-valued) functions $u( \ph, x; \om, \zeta)$, 
$(\ph,x) \in \T^\nu \times \T^3$, depending on parameters $(\om,\zeta) \in \R^{\nu+3}$ 
in a Lipschitz way together with their derivatives. 
We use the compact notation $\lambda := (\omega,\zeta)$ to collect 
the frequency $\om$ and the depth $\zeta$ into one parameter vector. 

Recall 
the standard multi-index notation: 
for $ k = ( k_1, \ldots , k_n) \in \N^n$,
we denote $|k| := k_1 + \ldots + k_n$ 
and $k! := k_1!  \cdots k_n!$; 
for $ \lambda = (\lambda_1, \ldots, \lambda_n) \in \R^n$, 
we denote  the derivative $ \pa_\lambda^k := \pa_{\lambda_1}^{k_1} \ldots \pa_{\lambda_n}^{k_n} $ 
and the monomial $\lambda^k := \lambda_1^{k_1}  \cdots \lambda_n^{k_n} $. 
We fix 
\begin{equation}\label{definizione s0}
s_0 > (\nu+3)  + k_0 + 2
\end{equation}
once and for all, 
and define the weighted Sobolev norms in the following way. 

\begin{definition} 
{\bf (Weighted Sobolev norms)} 
\label{def:Lip F uniform} 
Let $k_0 \geq 1$ be an integer,  $\g \in (0,1]$, 
and $s \geq s_0$. 
Given a function $u : \R^{\nu + 3} \to H^s(\T^\nu \times \T^3)$, 
$\lm \mapsto u(\lm) = u(\ph,x; \lm)$ 
that admits $k_0$ derivatives with respect to $\lm$, 
we define its weighted Sobolev norm 
$$
\| u \|_{s}^{k_0, \gamma} := 
\max_{\begin{subarray}{c}
\alpha \in \N^{\nu + 3} \\
 |\alpha| \leq k_0
 \end{subarray}} 
\, \sup_{\lm \in \R^{\nu+3}} \, 
 \gamma^{|\alpha|}\| \partial_\lambda^\alpha u(\lm) \|_{s - |\alpha|} \,.
$$
For $u$ independent of $(\ph,x)$, we simply denote by 
$| u |^{k_0,\g}$ the same norm. 
\end{definition}

For any $N>0$, we define the smoothing operators (Fourier truncation)
\begin{equation}\label{def:smoothings}
(\Pi_N u)(\ph,x) := \sum_{\la \ell,j \ra \leq N} \hat u(\ell, j) e^{\ii (\ell\cdot\ph + j \cdot x)}, \qquad
\Pi^\perp_N := {\rm Id} - \Pi_N.
\end{equation}

\begin{lemma} {\bf (Smoothing)} \label{lemma:smoothing}
The smoothing operators $\Pi_N, \Pi_N^\perp$ satisfy 
the smoothing estimates
\begin{align}
\| \Pi_N u \|_{s}^{k_0, \gamma} 
& \leq N^a \| u \|_{s-a}^{k_0, \gamma}\, , \quad 0 \leq a \leq s, 
\label{p2-proi} \\
\| \Pi_N^\bot u \|_{s}^{k_0, \gamma} 
& \leq N^{-a} \| u \|_{s + a}^{k_0 , \gamma}\, , \quad  a \geq 0.
\label{p3-proi}
\end{align}
\end{lemma}

%
%
%

\begin{lemma}{\bf (Product and composition)}
\label{lemma:LS norms}
$(i)$ For all $ s \geq s_0$, 
\begin{align}
\| uv \|_{s}^{k_0, \gamma}
& \leq C(s, k_0) \| u \|_{s}^{k_0, \gamma} \| v \|_{s_0}^{k_0 , \gamma} 
+ C(s_0, k_0) \| u \|_{s_0}^{k_0 , \gamma} \| v \|_{s}^{k_0 , \gamma}\,. 
\label{p1-pr}
\end{align}
$(ii)$ Let $ \| \alpha \|_{s_0}^{k_0, \gamma} \leq \d (s_0, k_0) $ small enough. 
Then the composition operator 
\begin{equation} \label{def cal A}
\mA : u \mapsto \mA u, \quad  
(\mA u)(\ph,x) := u(\ph, x + \a (\ph,x)) \, , 
\end{equation}
satisfies the following tame estimates: for all $ s \geq s_0$, 
\be\label{pr-comp1}
\| {\cal A} u \|_{s }^{k_0 , \gamma} \lesssim_{s, k_0} \| u \|_{s }^{k_0 , \gamma} 
+ \| \alpha \|_{s}^{k_0 , \gamma} \| u \|_{s_0 }^{k_0 , \gamma} \, .
\ee
The function $ \breve \alpha $, 
defined by the inverse diffeomorphism 
$ y = x + \alpha (\vphi, x) $ if and only if $ x = y + \breve \alpha ( \vphi, y ) $,  
satisfies 
\be\label{p1-diffeo-inv}
\| \breve \alpha \|_{s}^{k_0 , \gamma} \lesssim_{s, k}  \| \alpha \|_{s }^{k_0 , \gamma} \, . 
\ee
As a consequence 
\be\label{pr-comp1 inv}
\| {\cal A}^{- 1} u \|_{s }^{k_0 , \gamma} \lesssim_{s, k_0} \| u \|_{s }^{k_0 , \gamma} 
+ \| \alpha \|_{s}^{k_0 , \gamma} \| u \|_{s_0 }^{k_0 , \gamma} \, .
\ee
$(iii)$ Assume that $ \| \alpha \|_{s_0 + 1}^{k_0, \gamma} \leq \d (s_0, k_0) $ small enough. Then 
\begin{equation}\label{stima A  A star - Id astratto}
\| ({\cal A} - {\rm Id}) h \|_s^{k_0, \gamma}\,,\, \| ({\cal A}^* - {\rm Id}) h \|_s^{k_0, \gamma} \lesssim_{s, k_0} \| \alpha \|_{s_0 + 1}^{k_0, \gamma} \| h \|_{s + 1}^{k_0, \gamma} + \| \alpha \|_{s + 1}^{k_0, \gamma} \| h \|_{s_0 + 1}^{k_0, \gamma}\,. 
\end{equation}
Similar estimates hold for the operators ${\cal A}^{- 1}, ({\cal A}^{- 1})^*$. 
\end{lemma}

\begin{proof}
Items $(i), (ii)$ follows as in \cite{Berti-Montalto}, 
taking into account Definition \ref{def:Lip F uniform} 
(here $\partial_\lambda^\alpha u$, $|\alpha| \leq k_0$, 
is estimated in $H^{s - |\alpha|}$, 
whereas in \cite{Berti-Montalto} and \cite{BBHM} they are estimated in $H^s$) 
and the definition of $s_0$ in \eqref{definizione s0}. 
We only prove $(iii)$. A direct calculation shows that
$$
({\cal A} - {\rm Id})h(\vphi, x) = \int_0^1 \nabla h(\vphi, x + t \alpha(\vphi, x)) \cdot \alpha(\vphi, x)\, d t\,.
$$ 
Then, the estimate on ${\cal A} - {\rm Id}$ follows by applying the estimates \eqref{p1-pr}, \eqref{pr-comp1}. The estimate for ${\cal A}^{- 1} - {\rm Id}$ can be proved similarly. 
Now we estimate the operator ${\cal A}^* - {\rm Id}$. A direct calculation shows that 
$$
{\cal A}^* h(\vphi, y) = {\rm det} \big( {\rm Id} + \nabla \breve \alpha(\vphi, y) \big) 
h(\vphi, y + \breve \alpha(\vphi, y)) \,.
$$
Note that for $\| \nabla \breve \alpha \|_{L^\infty}$ small enough 
one has ${\rm det} ( {\rm Id} + \nabla \breve \alpha(\vphi, y) ) \geq \frac12$. 
This is guaranteed by the smallness assumption $\| \alpha \|_{s_0 + 1}^{k_0, \gamma} \leq \delta$, by the estimate \eqref{p1-diffeo-inv} and by Sobolev embeddings. One writes 
\begin{equation}\label{bla bla car 100}
\begin{aligned}
({\cal A}^* - {\rm Id}) h (\vphi, y) 
& = \Big( {\rm det} \big( {\rm Id} + \nabla \breve \alpha(\vphi, y) \big) - 1 \Big)
{\cal A}^{- 1}h(\vphi, y ) + ({\cal A}^{- 1} - {\rm Id}) h(\vphi, y)\,. 
\end{aligned}
\end{equation}
By estimates \eqref{p1-pr}, \eqref{p1-diffeo-inv} 
together with the smallness assumption $\| \alpha \|_{s_0 + 1}^{k_0, \gamma} \leq \delta$ 
one deduces that 
\[
\| {\rm det} ( {\rm Id} + \nabla \breve \alpha(\vphi, y) ) - 1 \|_s^{k_0, \gamma} 
\lesssim_{s, k_0} \| \alpha \|_{s + 1}^{k_0, \gamma}.
\] 
Therefore the claimed estimate for ${\cal A}^* - {\rm Id}$ follows by \eqref{bla bla car 100}, 
\eqref{p1-pr}, \eqref{pr-comp1 inv}. 
The estimate for $({\cal A}^{- 1})^* - {\rm Id}$ can be proved similarly. 
\end{proof}

Let $\mathtt m : \R^{\nu + 3} \to \R^3$, $(\omega, \zeta) \mapsto \mathtt m(\omega, \zeta)$, 
be a $k_0$ times differentiable function satisfying 
$|\mathtt m - \zeta|^{k_0, \gamma} \leq \frac12$, 
and define the set 
\begin{equation}\label{DC tau0 gamma0}
{\cal O}(\gamma, \tau) := \Big\{ (\omega, \zeta) \in \mathtt \R^\nu \times \R^3 : |\omega \cdot \ell + \mathtt m(\omega, \zeta) \cdot j| \geq \frac{\gamma}{\langle \ell , j\rangle^{\tau} } \ \  \forall (\ell, j) \in \Z^{\nu + 3} \setminus  \{ (0, 0) \} \Big\} \,.
\end{equation}
The equation $(\ompaph + \mathtt m \cdot \nabla)v = u$, where $u(\ph,x)$ has zero average with respect to $ (\vphi, x) $, has the periodic solution 
\begin{equation}\label{def:ompaph}
(\om \cdot \pa_\vphi + \mathtt m \cdot \nabla )^{-1} u (\ph,x) 
:= \sum_{(\ell, j) \in \Z^{\nu + 3} \setminus \{(0, 0)\} } 
\frac{ \widehat u(\ell, j) }{\ii (\om \cdot \ell + \mathtt m \cdot j)} 
e^{\ii (\ell \cdot \vphi + j \cdot x )} \,.
\end{equation}
We define its extension to all $(\om, \zeta) \in \R^\nu \times \R^3$ as
\begin{equation} \label{def ompaph-1 ext}
(\ompaph + \mathtt m \cdot \nabla)^{-1}_{ext} u(\ph,x) 
:= \sum_{(\ell, j) \in \Z^{\nu+3}} 
\frac{\chi\big( (\om \cdot \ell + \mathtt m \cdot j) \g^{-1} \langle \ell, j \rangle^{\t} \big) }
{\ii (\om \cdot \ell + \mathtt m \cdot j) }\,
\widehat u(\ell, j) \, e^{\ii (\ell \cdot \ph + j \cdot x)},
\end{equation}
where $\chi \in \mC^\infty(\R,\R)$ is an even and positive cut-off function such that 
\begin{equation}\label{cut off simboli 1}
\chi(\xi) = \begin{cases}
0 & \quad \text{if } \quad |\xi| \leq \frac13 \\
1 & \quad \text{if} \quad  \ |\xi| \geq \frac23\,,
\end{cases} \qquad 
\partial_\xi \chi(\xi) > 0 \quad \forall \xi \in \Big(\frac13, \frac23 \Big) \, . 
\end{equation}
Note that $(\ompaph + \mathtt m \cdot \nabla)^{-1}_{ext} u 
= (\ompaph + \mathtt m \cdot \nabla)^{-1} u$
for all $(\om, \zeta) \in {\cal O}(\gamma, \tau)$.

\begin{lemma} { \bf (Diophantine equation)}
\label{lemma:WD}
One has 
\be \label{2802.2}
\| (\om \cdot \pa_\vphi + \mathtt m \cdot \nabla)^{-1}_{ext} u \|_{s}^{k_0, \gamma}
\lesssim_{k_0} \g^{-1} \| u \|_{s+\tau_0}^{k_0, \gamma}, 
\qquad \tau_0 := k_0 + \t(k_0+1). 
\ee
\end{lemma}

\begin{definition}\label{definizione norma pseudo diff}
{ \bf (Pseudo-differential operators and symbols)}
Let $m \in \R$, $s \geq s_0$, $\beta \in \N$, $n \in \N$. 
We say that an operator $\mA = \mA(\ph)$ is in the class ${\cal OPS}^{m}_{s, \beta}$ 
if there exists a function 
$A : \T^\nu \times \T^3 \times \R^3 \to \Mat_{n \times n}(\C)$, 
$A = A(\ph,x,\xi)$, differentiable $\b$ times in the variable $\xi$,
such that 
$$
\mA u(x) = {\rm Op}( A ) u(x) 
= \sum_{\xi \in \Z^3 } A(\vphi, x, \xi) \hat u(\xi) e^{\ii x \cdot \xi} 
\quad \forall u \in {\cal C}^\infty(\T^3, \C^n),
$$
and 
\begin{equation} \label{def norma pseudo-diff}
| \mA |_{m, s, \beta} := \sup_{|\a| \leq \beta} \sup_{\xi \in \R^3} 
\| \pa_\xi^\a A(\cdot, \xi) \|_{s} \langle \xi \rangle^{- m + |\a|} < \infty\,;
\end{equation}
in that case, we also say that $A(\ph,x,\xi)$ is in the class $\mS^{m}_{s, \beta}$. 
The operator $\mA$ is said to be a \emph{pseudo-differential operator of order} $m$, 
and the function $A$ is its \emph{symbol}. 

If ${\cal A} = {\cal A}(\lambda)$ depends in a $k_0$ times differentiable way 
on the parameters $\lambda = (\omega, \zeta) \in \R^{\nu+3}$, 
we define 
\begin{equation} \label{3009.2}
|{\cal A}|_{m, s, \beta}^{k_0, \gamma} 
:= \sup_{|\a| \leq \beta} \sup_{\xi \in \R^3} 
\| \pa_\xi^\a A(\cdot, \xi) \|_{s}^{k_0,\g} \langle \xi \rangle^{- m + |\a|}
= \max_{|k| \leq k_0} \gamma^{|k|} \sup_{\lm \in \R^{\nu+3}} 
|\partial_\lambda^k {\cal A}(\lambda)|_{m, s - |k|, \b}.
\end{equation}
\end{definition}

The values of $n$ in Definition \ref{definizione norma pseudo diff} 
that we need for our problem are 
$n=1$ (when both $u$ and $A$ are scalar functions)
and $n=3$ (when the function $u$ takes values in $\R^3$ or $\C^3$, 
and the symbol $A(\ph,x,\xi)$ is a $3 \times 3$ matrix).
Recall that, when $A(\ph,x,\xi)$ is a matrix, its Sobolev norm is given by 
\eqref{def Euclidean norm}-\eqref{def Sobolev norm generale}. 

In the rest of this section we assume, without explicitly writing it, 
that all functions and symbols depend in a $k_0$ times differentiable way 
on the parameter $\lm \in \R^{\nu+3}$. 

Given a symbol $A = A(\ph,x,\xi) \in {\cal S}^m_{s, \beta}$, 
we define the averaged symbol $\langle A \rangle_{\vphi, x}$ as 
\begin{equation}\label{simbolo mediato}
\langle A \rangle_{\vphi, x} (\xi) := \frac{1}{(2 \pi)^{\nu + 3}}\int_{\T^{\nu + 3}} A(\vphi, x, \xi)\, d \vphi\, d x\,. 
\end{equation}
One easily verifies that, for all $s \geq 0$,
\begin{equation}\label{proprieta simbolo mediato}
\langle A \rangle_{\vphi, x} \in {\cal S}^{m}_{s, \beta} 
\quad \text{and} \quad 
|{\rm Op}(\langle A \rangle_{\vphi, x})|_{m,s, \beta}^{k_0, \gamma} 
= |{\rm Op}(\langle A \rangle_{\vphi, x})|_{m, s_0 , \beta}^{k_0, \gamma} 
\leq |{\rm Op}(A)|_{m, s_0, \beta}^{k_0, \gamma}\,.
\end{equation}
Moreover if a symbol $A = A(\ph,x)$ is independent of $\xi$, 
then the corresponding operator $\Op(A) \in {\cal OPS}^0_{s, \beta}$ is a multiplication operator, 
and  
\begin{equation}\label{pseudo norm moltiplicazione}
{\rm Op}(A) : h(\vphi, x) \mapsto A(\vphi, x) h(\vphi, x), 
\qquad 
|{\rm Op}(A)|_{0, s, \beta}^{k_0, \gamma} = \| A \|_s^{k_0, \gamma}\,. 
\end{equation}
By the definition of the norm in \eqref{def norma pseudo-diff}, 
and using the interpolation estimate \eqref{p1-pr}, 
it follows that if ${\cal A} = {\rm Op}(A) \in {\cal OPS}^m_{s, \beta}$, ${\cal B} = {\rm Op}(B) \in {\cal OPS}^{m'}_{s, \beta}$, $s \geq s_0$, then ${\rm Op}(AB) \in {\cal OPS}^{m + m'}_{s, \beta}$ and 
\begin{equation} \label{stima prodotto simboli}
|{\rm Op}(AB)|_{m + m', s, \beta}^{k_0, \gamma} \lesssim_{s, k_0, \beta} |{\cal A}|_{m, s, \beta}^{k_0, \gamma}|{\cal B}|_{m', s_0, \beta}^{k_0, \gamma} + |{\cal A}|_{m, s_0, \beta}^{k_0, \gamma}|{\cal B}|_{m', s, \beta}^{k_0, \gamma}\,. 
\end{equation} 
Iterating estimate \eqref{stima prodotto simboli} one has that if ${\rm Op}(A) \in {\cal OPS}^{0}_{s, \beta}$, $s \geq s_0$, $\beta \in \N$, then for any $n \geq 1$, ${\rm Op}(A^n) \in {\cal OPS}^0_{s, \beta}$ and 
\begin{equation}\label{stima Op A n}
|{\rm Op}(A^n) |_{0, s, \beta}^{k_0, \gamma} \leq \Big(C(s, \beta) |{\rm Op}(A)|_{0, s_0, \beta} \Big)^{n - 1} |{\rm Op}(A)|_{0, s, \beta}
\end{equation}
for some constant $C(s, \beta) > 0$.

\begin{lemma}\label{azione pseudo}
Let $s \geq s_0$ and ${\cal A}(\lambda) \in {\cal OPS}^0_{s, 0}$, 
$u(\lambda) \in H^s(\T^{\nu + 3}, \R^3)$ for all $\lambda \in \R^{\nu+3}$.
Then 
$$
\| {\cal A} u \|_s^{k_0, \gamma} \lesssim_{s, k_0} |{\cal A}|_{0, s, 0}^{k_0, \gamma} \| u \|_{s_0}^{k_0, \gamma} + |{\cal A}|_{0, s_0, 0}^{k_0, \gamma} \| u \|_{s}^{k_0, \gamma}\,. 
$$
\end{lemma}

\begin{lemma} \label{lemma stime Ck parametri} 
{\bf (Composition of pseudo-differential operators)} 
Let $s \geq s_0$, $m, m' \in \R$, $\b \in \N$. 

$(i)$ 
Let ${\cal A} = {\rm Op}(A) \in {\cal OPS}^{m}_{s, \beta}$, 
${\cal B} = {\rm Op}(B) \in {\cal OPS}^{m'}_{s + |m| + \beta, \beta}$. 
Then the composition ${\cal A} {\cal B}$ belongs to ${\cal OPS}^{m + m'}_{s, \b}$,  
and 
\begin{equation} \label{estimate composition parameters}
| {\cal A} {\cal B} |_{m + m', s, \beta}^{k_0, \gamma} 
\lesssim_{s, m, k_0, \beta} |{\cal A} |_{m, s, \beta}^{k_0, \gamma} 
| {\cal B} |_{m', s_0 + |m| + \beta, \beta}^{k_0, \gamma}  
+  | {\cal A}  |_{m, s_0, \beta}^{k_0, \gamma}  
| {\cal B} |_{m', s  + |m| + \beta, \beta}^{k_0, \gamma} \, . 
\end{equation}

$(ii)$ 
Let ${\cal A} = {\rm Op}(A) \in{\cal OPS}^{m}_{s, 1}$, 
${\cal B} = {\rm Op}(B) \in {\cal OPS}^{m'}_{s + |m| + 2, 0}$. 
Then
$$
{\cal A}{\cal B} = {\rm Op} \big( A(\vphi, x, \xi) B(\vphi, x, \xi) \big) + {\cal R}_{AB},
\quad \ {\cal R}_{AB} \in {\cal OPS}^{m + m' - 1}_{s, 0},
$$
where the remainder ${\cal R}_{AB}$ 
satisfies 
$$
|{\cal R}_{AB}|_{m + m' - 1, s, 0}^{k_0, \gamma} \lesssim_{s, m, k_0} |{\cal A}|_{m, s, 1}^{k_0, \gamma} |{\cal B}|_{m', s_0 + |m| + 2, 0}^{k_0, \gamma} +  |{\cal A}|_{m, s_0, 1}^{k_0, \gamma} |{\cal B}|_{m', s + |m| + 2, 0}^{k_0, \gamma}\,. 
$$ 

\end{lemma}

\begin{proof}
See Lemma 2.13 in \cite{Berti-Montalto}
\end{proof}
For functions $u, v : \T^3 \to \R^n$, $n=1$ or $n=3$,
we consider the $L^2(\T^3,\R^n)$ scalar product 
\begin{equation} \label{def scalar prod L2 T3 R3}
\langle u,v \rangle_{L^2(\T^3, \R^n)} 
:= \Pi_0 (u \cdot v)
\end{equation}
where 
$\Pi_0$ is the space average defined in \eqref{definizione proiettore media spazio tempo}
and ``\,$\cdot$\,'' is the standard scalar product in $\R^n$.
The adjoint of an operator mapping $H^s(\T^\nu \times \T^3,\R^n)$ 
into itself is considered with respect to the scalar product \eqref{def scalar prod L2 T3 R3}. 

\begin{lemma}[\bf Adjoint]\label{lemma aggiunto} 
Let $m \in \R$, $s \geq s_0$ and ${\cal A} = {\rm Op}(A) \in {\cal OPS}^m_{s + s_0 + |m|, 0}$. 
Then the adjoint operator ${\rm Op}(A)^*  \in {\cal OPS}^m_{s, 0}$ and $|{\rm Op}(A)^*|_{m, s, 0}^{k_0, \gamma} \lesssim_m |{\rm Op}(A)|_{m, s + s_0 + |m|, 0}^{k_0, \gamma}$. 
\end{lemma}

\begin{proof}
For a matrix symbol $A = (A_{i k})_{i, k = 1,2,3}$, one computes 
${\rm Op}(A)^* = ( {\rm Op}(A_{k i})^* )_{i, k = 1,2,3}$. 
Then we argue as in Lemma in 2.16 \cite{Berti-Montalto} to each operator ${\rm Op}(A_{k i})^*$. 
\end{proof}

\begin{lemma}[\bf Neumann series]\label{Lemma Neumann}
Let $s \geq s_0$ and $\Psi \in {\cal OPS}^{- m}_{s, 0}$, $m \geq 0$. 
There exists $\delta = \delta(s_0, k_0) \in (0, 1)$ small enough 
such that, if $|\Psi|_{- m , s_0, 0}^{k_0, \gamma} \leq \delta$, 
then $\Phi = {\rm Id} + \Psi$ is invertible 
and $| \Phi^{-1} - {\rm Id} |_{- m, s, 0}^{k_0, \gamma} 
\lesssim_{s,   k_0} | \Psi|_{- m, s, 0}^{k_0, \gamma}$. 
\end{lemma}

\begin{proof}
See Lemma 2.17 in \cite{Berti-Montalto}. 
\end{proof}

%

\subsection{Some conjugations with changes of variables}\label{sez coniugi per eulero}

In the next lemma we exploit some properties of some pseudo-differential operators conjugated by a change of variables. We will always assume the hypotheses of Lemma \ref{lemma:LS norms}. 

\begin{lemma}\label{lemma coniugazione cambio variabile moltiplicazione}
Let $S > s_0$, $\alpha(\cdot; \lambda) \in H^S$, 
$M (\cdot ; \lambda) = (a_{ij}(\cdot; \lambda))_{i, j =1,2,3} \in H^S$ 
and $\| \alpha \|_{s_0}^{k_0, \gamma} \leq \delta$ for some $\delta \in (0, 1)$ small enough. 
Then ${\cal A}^{- 1} {\rm Op}(M) {\cal A}$ is the multiplication operator by the $3 \times 3$ matrix $\widetilde M(\vphi, y) = M(\vphi, y + \breve \alpha(\vphi, y))$.  Moreover for any $s_0 \leq s \leq S$, $\beta \in \N$
$$
|{\rm Op}(\widetilde M)|_{0, s, \beta} = \| \widetilde M \|_s^{k_0, \gamma} \lesssim_{s, k_0} \| M \|_s^{k_0, \gamma} + \| \alpha \|_s^{k_0, \gamma} \| M \|_{s_0}^{k_0, \gamma}\,. 
$$
\end{lemma}

\begin{proof}
The lemma is a straightforward consequence of Lemma \ref{lemma:LS norms}-$(ii)$ and the estimate \eqref{pseudo norm moltiplicazione}. 
\end{proof}

\begin{lemma}\label{lemma coniugazione curl nabla diffeo}
Let $S > s_0$, $\alpha(\cdot; \lambda) \in H^{S + 1}$, $a(\cdot; \lambda) \in H^S$ and $\| \alpha \|_{s_0 + 1}^{k_0, \gamma} \leq \delta$ for some $\delta \in (0, 1)$ small enough. Then for any $s_0 \leq s \leq S$, $\beta \in \N$, ${\cal A}^{- 1} a \cdot \nabla {\cal A} = {\rm Op}(M_\nabla) \in {\cal OPS}_{s, \beta}^1$ and ${\cal A}^{- 1} {\rm curl} {\cal A} = {\rm Op}(M_{\rm curl}) \in {\cal OPS}_{s, \beta}^1$. The following estimates hold for any $s_0 \leq s \leq S$: 
$$
\begin{aligned}
& |{\cal A}^{- 1} a \cdot \nabla {\cal A} |_{1, s, \beta}^{k_0, \gamma} \lesssim_{s, k_0, \beta} \| a \|_s^{k_0, \gamma} + \| \alpha \|_{s + 1}^{k_0, \gamma} \| a \|_{s_0}^{k_0, \gamma}\,, \\
& |{\cal A}^{- 1} {\rm curl} {\cal A} |_{1, s, \beta}^{k_0, \gamma} \lesssim_{s, k_0, \beta} 1 + \| \alpha \|_{s + 1}^{k_0, \gamma}\,. 
\end{aligned} 
 $$
 Furthermore, the symbols $M_\nabla, M_{\rm curl}$ satisfy the symmetry conditions
 $$
 M_\nabla (\vphi, x, \xi) = - M_\nabla(\vphi, x, - \xi), \quad M_{\rm curl}(\vphi, x, \xi) = - M_{\rm curl}(\vphi, x, - \xi)\,.
 $$
\end{lemma}

\begin{proof}
By recalling the definition of ${\rm curl}$ given in \eqref{def vorticita}, it suffices only to analyze for any $i = 1,2,3$ the operator ${\cal M} := {\cal A}^{- 1} \partial_{x_i} {\cal A}$. A direct calculation shows that ${\cal M} = \partial_{x_i} + {\cal A}^{- 1}[\partial_{x_i} \alpha] \cdot \nabla$. 
This immediately implies that ${\cal M} = {\rm Op}(M) \in {\cal OPS}^1_{s, k}$ for any $k \in \N$ and $M(\vphi, x, \xi) = - M(\vphi, x, - \xi)$. Moreover by applying the estimates \eqref{pr-comp1},  \eqref{p1-diffeo-inv}, \eqref{pseudo norm moltiplicazione} and the trivial fact that $|\partial_{x_j}|_{1,s, k} \lesssim 1$ for any $s, k$ one gets the estimate $|{\cal M}|_{1, s, \beta}^{k_0, \gamma} \lesssim_{s, \beta, k_0} 1 + \| \alpha \|_{s + 1}^{k_0, \gamma}$. 
\end{proof}

In the following we analyze the conjugation of the operator $\Lambda^{- 1}$ by means of a change of variables, where we recall that $\Lambda := \Pi_0 - \Delta$. Note that the action of the operator $\Lambda$ on a function $u \in H^2(\T^3, \R)$ is given by 
$$
\Lambda u(x) = \widehat u(0) 
+ \sum_{\xi \in \Z^3 \setminus \{ 0 \}} |\xi|^2 \widehat u(\xi) e^{\ii x \cdot \xi}\,. 
$$
We identify the operator $\Lambda$ with ${\rm Op}(\lambda(\xi))$ where 
\begin{equation}\label{simbolo Lambda}
\begin{aligned}
& \lambda \in {\cal C}^\infty(\R^3, \R), \quad \lambda (\xi ) = \lambda(- \xi)\,, \quad \inf_{\xi \in \R^3} \lambda(\xi) > 0\,, \\
& \lambda(\xi) = |\xi|^2 \quad \text{if} \quad |\xi| \geq 1 \quad \text{and} \quad \lambda(0) = 1\,. 
\end{aligned}
\end{equation}
The inverse of $\Lambda$ is computed explicitly 
on periodic functions as 
$$
\Lambda^{- 1} u(x) = \widehat u(0) + \sum_{\xi \in \Z^3 \setminus \{ 0 \}} \frac{1}{|\xi|^2} \widehat u(\xi) e^{\ii x \cdot \xi}\,. 
$$
Then for any $s, \beta\geq 0$ 
\begin{equation}\label{stima Lambda inv}
|\Lambda|_{2, s, \beta}\,,\,|\Lambda^{- 1}|_{- 2, s, \beta} \lesssim_{\beta} 1 \,. 
\end{equation}
We also identify the projector $\Pi_0$ with ${\rm Op}(\chi_0(\xi))$ where $\chi_0 \in {\cal C}^\infty(\R^3, \R)$  satisfies 
\begin{equation}\label{simbolo Pi 0}
\begin{aligned}
& \chi_0 \in {\cal C}^\infty(\R^3, \R), \quad   \chi_0(\xi)= \chi_0 (- \xi)\,, \quad 0 \leq \chi_0 \leq 1, \quad  \chi_0(0) = 1\,, \\
& {\rm supp}(\chi_0) \subset \Big\{ \xi \in \R^3 : |\xi| \leq \frac12 \Big\}\,. 
\end{aligned}
\end{equation}
Hence for any $m, \alpha, s \geq 0$
\begin{equation}\label{stima Pi 0}
|\Pi_0|_{- m, s, \alpha} \lesssim_{m, \alpha} 1\,. 
\end{equation}

\begin{lemma} \label{coniugazione laplaciano cambio variabile}
$S > s_0$, $\alpha(\cdot; \lambda) \in H^{S + 2}$ 
and $\| \alpha \|_{s_0 + 2}^{k_0, \gamma} \leq \delta$ 
for some $\delta = \delta( k_0, \nu) \in (0, 1)$ small enough. 
Then ${\cal P}_\Lambda : = {\cal A}^{- 1} \Lambda {\cal A} 
= \Lambda + {\cal P}_2 + {\cal P}_1$ 
where, for $s_0 \leq s \leq S$, 
${\cal P}_2 = {\rm Op}(P_2) \in {\cal OPS}_{s, 1}^{ 2}$, 
${\cal P}_1 \in {\cal OPS}_{s, 0}^{1}$ 
satisfy the estimates  
$$
|{\cal P}_2|_{2, s, 1}^{k_0, \gamma}\,,\, |{\cal P}_1|_{1 , s, 1}^{k_0, \gamma} \lesssim_{s, k_0}  \| \alpha \|_{s + \mu}^{k_0, \gamma}\,.
$$
Furthermore, the symbol $P_2$ satisfies the symmetry condition 
$P_2(\vphi, x, \xi) = P_2(\vphi, x, - \xi)$.
\end{lemma}

\begin{proof}
The assumption of the lemma allows to apply Lemma \ref{lemma:LS norms} on the change of variables. 
A direct calculation shows that 
$- {\cal A}^{-1} \Delta {\cal A}$ is an elliptic operator of the form 
\begin{equation}\label{bla bla car 0}
\begin{aligned}
- {\cal A}^{-1} \Delta {\cal A} 
& = - \Delta - \sum_{i, j = 1}^3 a_{i j}(\vphi, x) \partial_{x_i x_j} 
+ \sum_{i = 1}^3 b_i(\vphi, x) \partial_{x_i}
\end{aligned}
\end{equation}
with 
\begin{equation}\label{stime a ij b i}
\| a_{i j} \|_s^{k_0, \gamma}\,,\, \| b_i \|_s^{k_0, \gamma} \lesssim_{s, k_0} \| \alpha \|_{s + 2}^{k_0, \gamma}\,, \quad i, j = 1,2,3\,. 
\end{equation}
Recalling the definition of $\Pi_0$ given in \eqref{definizione proiettore media spazio tempo} 
and using that ${\cal A}, {\cal A}^{- 1}$ are of the form \eqref{def cal A}, 
one computes 
\begin{equation}\label{bla bla car 1000}
\Pi_0 {\cal A} = \Pi_0 + \Pi_0 ({\cal A} - {\rm Id})
\end{equation}
and 
\[
\Pi_0 (({\cal A} - {\rm Id}) h) = \Pi_0 (q_0 h), 
\qquad 
q_0(\ph,x) := (\mA - \mathrm{Id})^*[1]
= \det (\mathrm{I} + D \breve \a(\ph,x)) - 1.
\]
Then using \eqref{pseudo norm moltiplicazione}, 
Lemma \ref{lemma stime Ck parametri}-$(iii)$ 
and the trivial facts that $|\Pi_0|_{0, s, 0} \leq 1$, one obtains that $\Pi_0 ({\cal A} - {\rm Id}) \in {\cal OPS}^0_{s, 1}$, 
with
\begin{equation}\label{bla bla car 1001}
|\Pi_0 ({\cal A} - {\rm Id}) |_{0, s, 1}^{k_0, \gamma} 
\lesssim_{s, k_0} \| \alpha \|_{s + 1}^{k_0, \gamma}.
\end{equation} 
Since $\mA^{-1} \Pi_0 = \Pi_0$, 
by \eqref{bla bla car 0}, \eqref{bla bla car 1000}, 
one gets that 
\begin{equation}\label{kiwi 0}
\begin{aligned}
& {\cal A}^{- 1} \Lambda {\cal A} = \Lambda + {\cal P}_2 + {\cal P}_1, \\
& {\cal P}_2 := - \sum_{i, j =1 }^3 a_{i j}(\vphi,x) \partial_{x_i x_j}\,, \quad {\cal P}_1 := \sum_{i = 1}^3 b_i(\vphi, x) \partial_{x_i} + \Pi_0 ({\cal A} - {\rm Id})\,. 
\end{aligned}
\end{equation}
By \eqref{pseudo norm moltiplicazione}, Lemma \ref{lemma stime Ck parametri}, using that $|\partial_{x_i x_j}|_{2, s, 1}, |\partial_{x_i }|_{1, s, 1} \lesssim 1$ and the estimates \eqref{stime a ij b i}, \eqref{bla bla car 1001} one gets the claimed bounds on ${\cal P}_1$ and ${\cal P}_2$. 
Moreover by the formula \eqref{kiwi 0}, one deduces that the symbol of the operator ${\cal P}_2$ 
is $P_2(\vphi, x, \xi) = \sum_{i, j = 1}^3 a_{i j}(\vphi, x) \xi_i \xi_j$, 
which is 
even with respect to the variable $\xi$. 
\end{proof}

\begin{lemma}\label{coniugazione inverso laplaciano cambio variabile}
Let $s > s_0$, 
$\alpha(\cdot; \lambda) \in H^{s + \mu}$, 
and $\| \alpha \|_{s_0 + \mu}^{k_0, \gamma} \leq \delta$ 
for some $\delta = \delta(s, k_0, \nu) \in (0, 1)$ small enough 
and some $\mu = \mu(k_0, \nu) > 0$ large enough. 
Then ${\cal P}_\Lambda := {\cal A}^{-1} \Lambda {\cal A}$ 
defined in Lemma \ref{coniugazione laplaciano cambio variabile}
is invertible and its inverse is of the form 
\[
{\cal P}_\Lambda^{- 1} 
= {\cal A}^{- 1} \Lambda^{- 1} {\cal A} 
= {\cal P}_{- 2} + {\cal P}_{- 3}
\]
where ${\cal P}_{- 2} = {\rm Op}(P_{- 2}) \in {\cal OPS}_{s, 1}^{ - 2}$, 
${\cal P}_{- 3} \in {\cal OPS}_{s, 0}^{- 3}$ 
satisfy the estimates  
$$
|{\cal P}_{ - 2}|_{- 2, s, 1}^{k_0, \gamma}\,,\, |{\cal P}_{- 3}|_{- 3 , s, 0}^{k_0, \gamma} \lesssim_{s, k_0} 1 +  \| \alpha \|_{s + \mu}^{k_0, \gamma}\,.
$$
Furthermore, the symbol $P_{ -2}$ satisfies the symmetry condition $P_{- 2}(\vphi, x, \xi) = P_{- 2}(\vphi, x, - \xi)$.
\end{lemma}

\begin{proof}
By Lemma \ref{coniugazione laplaciano cambio variabile}, 
we write the operator ${\cal P}_\Lambda = \Lm + \mP_2 + \mP_1$ as 
\begin{equation}\label{primo cap P Lambda}
\begin{aligned}
& {\cal P}_\Lambda = \Lambda ( {\rm Id} + {\cal F} ), \quad 
{\cal F} :=   \Lambda^{- 1} {\cal P}_2 + \Lambda^{- 1} {\cal P}_1\,. 
\end{aligned}
\end{equation}
Note that $\Lambda^{- 1} {\cal P}_2$ is of order $0$ 
and $\Lambda^{- 1} {\cal P}_1$ is of order $- 1$. 
Define 
\begin{equation}
{\cal F}_0 := {\rm Op}(F_0(\vphi, x, \xi)), \quad 
F_0(\vphi, x, \xi) := \lambda(\xi)^{- 1} P_2(\vphi, x, \xi), \quad \ 
{\cal F}_{-1} := {\cal F} - {\cal F}_0
\end{equation}
(in fact, $\mF_0 = \mP_2 \Lm^{-1}$).  
By applying Lemma \ref{lemma stime Ck parametri} 
and the estimates of ${\cal P}_1, {\cal P}_2$ 
provided by Lemma \ref{coniugazione laplaciano cambio variabile} 
(recall also that $\Lambda \equiv {\rm Op}(\lambda)$,  see \eqref{simbolo Lambda}) 
with $s$ replaced by $s + 2$, one obtains the bounds
\begin{equation}\label{stime F0 F1 lemma neumann omogeneo}
|{\cal F}_0|_{0, s + 2, 1}^{k_0, \gamma}\,,\, |{\cal F}_{- 1}|_{- 1, s + 2 , 1}^{k_0, \gamma} \lesssim_{s, k_0} \| \alpha \|_{s + \mu}^{k_0, \gamma}
\end{equation}
for some constant $\mu > 0$. 
Moreover, since $\lambda(\xi)$ and the symbol $P_2(\vphi,x, \xi)$ 
are even functions of $\xi$, 
then also $F_0(\vphi, x, \xi)$ is even in $\xi$. 
By Lemma \ref{Lemma Neumann}, 
using \eqref{stime F0 F1 lemma neumann omogeneo} 
and the hypothesis that $\| \alpha \|_{s_0 + \mu}$ is small enough, 
we deduce that ${\rm Id} + {\cal F}$ is invertible, 
and its inverse satisfies the estimate 
\begin{equation} \label{2602.1}
| (\mathrm{Id} - \mF)^{-1} |_{0,s,0}^{k_0,\g} 
\lesssim_{s, k_0} \, 1 + \| \a \|_{s+\mu}^{k_0,\g}.
\end{equation}
By Lemma \ref{lemma stime Ck parametri}$(ii)$ 
and \eqref{stime F0 F1 lemma neumann omogeneo} one has 
\begin{equation} \label{2602.2}
\mF_0 \, \Op \Big( \frac{1}{1+F_0} \Big) 
= \Op \Big( \frac{F_0}{1 + F_0} \Big) + \mR_{F_0}, 
\qquad 
| \mR_{F_0} |_{-1,s,0}^{k_0,\g} 
\lesssim_s \, \| \a \|_{s+\mu}^{k_0,\g} \,.
\end{equation}
Since $\mF = \mF_0 + \mF_{-1}$, one has 
\begin{align*}
(\Id + \mF) \Op \Big( \frac{1}{1+F_0} \Big)
& = \Op \Big( \frac{1}{1+F_0} \Big) 
+ \Op \Big( \frac{F_0}{1+F_0} \Big) + \mR_{F_0} + \mF_{-1} \Op \Big( \frac{1}{1+F_0} \Big) 
\\ & 
= \Id + \mR_{F_0} + \mF_{-1} \mG, 
\qquad \quad 
\mG := \Op \Big( \frac{1}{1+F_0} \Big),
\end{align*}
whence, applying $(\Id + \mF)^{-1}$ from the left, we get
\begin{equation} \label{2602.3}
\mG = (\Id + \mF)^{-1} - \mR_{\mF}, 
\qquad 
\mR_{\mF} := - (\Id + \mF)^{-1} (\mR_{F_0} + \mF_{-1} \mG).
\end{equation}

Since $\frac{1}{1 + F_0} = \sum_{n \geq 0} (- 1)^n F_0^n$, 
by estimates \eqref{stima Op A n}, \eqref{stime F0 F1 lemma neumann omogeneo}, 
using the assumption that $\| \alpha \|_{s_0 + \mu}^{k_0, \gamma} \leq \delta$ is small enough, 
one gets  
\begin{equation}\label{1 + F0 inverse}
| \mG |_{0, s, 1}^{k_0, \gamma} 
\lesssim_{s, k_0}1 + \| \alpha \|_{s + \mu}^{k_0, \gamma}\,. 
\end{equation}
Then, recalling \eqref{primo cap P Lambda} 
and using \eqref{2602.3} to substitute $(\Id + \mF)^{-1}$,  
we get 
\[
{\cal P}_\Lambda^{-1} 
= ( {\rm Id} + {\cal F} )^{- 1} \Lambda^{- 1} 
= (\mG + \mR_{\mF}) \Lm^{-1}
= {\cal P}_{- 2} + {\cal P}_{- 3}
\]
with 
$$
{\cal P}_{-2} 
= \mG \Lm^{-1}
= {\rm Op}(P_{-2}), 
\qquad  
P_{- 2}(\vphi, x, \xi) := \frac{1}{\lambda(\xi)\big(1 + F_0(\vphi, x, \xi)\big)}, 
\qquad 
{\cal P}_{- 3} := {\cal R}_{\cal F} \Lambda^{- 1}\,. 
$$
Since $F_0$ and $\lambda$ are even in $\xi$, 
then $P_{- 2}$ is also even in $\xi$. 

By \eqref{2602.1}, \eqref{2602.2}, \eqref{2602.3}, \eqref{1 + F0 inverse}
we get 
$|{\cal R}_{\cal F}|_{- 1, s, 0}^{k_0, \gamma} 
\lesssim_{s, k_0} \| \alpha \|_{s + \mu}^{k_0, \gamma}$. 
Thus the estimates for ${\cal P}_{- 2}, {\cal P}_{- 3}$ 
follow 
by the composition estimates of Lemma \ref{lemma stime Ck parametri}, and by using \eqref{stima Lambda inv}. 
\end{proof}

\subsection{Matrix representation of linear operators} \label{sezione matrici norme}

Let us consider a linear matrix operator 
${\cal R} = ({\cal R}_{i k})_{i, k = 1,2,3} : L^2(\T^3, \R^3) \to L^2(\T^3, \R^3)$ 
where ${\cal R}_{ik} : L^2(\T^3, \R) \to L^2(\T^3, \R)$ for any $i, k = 1,2,3$. 
Such an operator can be represented as
\begin{equation}\label{matriciale 1}
{\cal R} u (x) := \sum_{j, j' \in \Z^3} {\cal R}_j^{j'}[\widehat u(j')] e^{\ii j \cdot x}, 
\quad \ \text{for} \ u (x) = \sum_{j \in \Z^3} \widehat u(j) e^{\ii j \cdot x}, 
\end{equation}
where, for $j, j' \in \Z^3$, ${\cal R}_j^{j'}$ is the $3 \times 3$ matrix defined by 
\begin{equation}\label{rappresentazione blocchi 3 per 3}
{\cal R}_j^{j'} := \Big( ({\cal R}_{ik})_j^{j'} \Big)_{i, k = 1,2,3}, \quad 
({\cal R}_{ik})_j^{j'} := \frac{1}{(2\pi)^3} \int_{\T^3} 
{\cal R}_{ik}[e^{\ii j' \cdot x}] e^{- \ii j \cdot x}\, d x\,. 
\end{equation}
It is immediate to check that the matrix representation \eqref{matriciale 1} 
of the operator $\mR$ is equivalent to the pseudo-differential representation 
$\mR = \Op(r)$, 
namely $\mR u(x) = \sum_{\xi \in \Z^3} r(x,\xi) \hat u(\xi) e^{\ii \xi \cdot x}$,
where the symbol $r(x,\xi)$, for $\xi = j' \in \Z^3$, 
is the $3 \times 3$ matrix given by 
\begin{equation} \label{equivalent symbol}
r(x,j') = \sum_{j \in \Z^3} R_j^{j'} e^{\ii (j-j') \cdot x}.
\end{equation}
Thus $R_j^{j'}$ is the Fourier coefficient 
$\hat r(j-j',j')$ 
of frequency $j-j'$
of the function $x \mapsto r(x,j')$.

\begin{definition}[Block-diagonal operator] 
\label{def block-diagonal op}
We say that an operator ${\cal R}$ as in 
\eqref{matriciale 1}-\eqref{rappresentazione blocchi 3 per 3}
is a \emph{$3 \times 3$ block-diagonal operator}
if ${\cal R}_j^{j'} = 0$ for all $j, j' \in \Z^3$ with $j \neq j'$.
\end{definition}

By \eqref{equivalent symbol}, 
$\mR$ is a block-diagonal operator 
if and only if the symbol $r(x,j')$ 
does not depend on $x$. 


We also consider smooth $\vphi$-dependent families of linear operators 
$\T^\nu \to {\cal B} (L^2(\T^3, \R^3))$, $\vphi \mapsto {\cal R}(\vphi)$, 
which we write in Fourier series with respect to $\vphi$ as 
$$
{\cal R}(\vphi) = \sum_{\ell \in \Z^\nu} \widehat{\cal R}(\ell) e^{\ii \ell \cdot \vphi}, \quad \widehat{\cal R}(\ell) := \frac{1}{(2 \pi)^\nu} \int_{\T^\nu} {\cal R}(\vphi) e^{- \ii \ell \cdot \vphi}\, d \vphi, \quad \ell \in \Z^\nu\,. 
$$
According to \eqref{rappresentazione blocchi 3 per 3}, for any $\ell \in \Z^\nu$, the linear operator $\widehat{\cal R}(\ell) \in {\cal B} (L^2(\T^3, \R^3))$ is identified 
with the matrix $(\widehat{\cal R}(\ell)_j^{j'})_{j, j' \in \Z^3}$ 
where each entry $\widehat{\cal R}(\ell)_j^{j'}$ belongs to $\Mat_{3\times 3}(\C)$.
A map $\T^\nu \to {\cal B} (L^2(\T^3, \R^3))$, $\vphi \mapsto {\cal R}(\vphi)$ 
can be also regarded as a linear operator 
$L^2(\T^{\nu + 3}, \R^3) \to L^2(\T^{\nu + 3}, \R^3)$ by 
\begin{equation} \label{amatriciana}
{\cal R} u(\vphi, x) := \sum_{\begin{subarray}{c}
\ell, \ell' \in \Z^\nu \\
j, j' \in \Z^3
\end{subarray}} \widehat{\cal R}(\ell - \ell')_j^{j '} \widehat u(\ell', j') e^{\ii (\ell \cdot \vphi + j \cdot x)}, \quad \forall u \in L^2(\T^{\nu + 3}, \R^3)\,. 
\end{equation}
The representation \eqref{amatriciana} of the operator $\mR$ 
is also equivalent to the pseudo-differential representation 
$\mR = \Op(r)$, 
where the symbol $r(\ph,x,\xi)$, for $\xi = j' \in \Z^3$, is
\begin{equation} \label{equivalent symbol con phi}
r(\ph,x,j') = \sum_{j \in \Z^3} R_j^{j'}(\ph) e^{\ii (j-j') \cdot x},
\end{equation}
which is, in fact, \eqref{equivalent symbol} with, in addition, 
the dependence on $\ph$. 
Similarly as above, $R_j^{j'}(\ph)$ is the Fourier coefficient 
$\hat r(\ph, j-j', j')$ of frequency $j-j'$
of the function $x \mapsto r(\ph,x,j')$.
If we expand both $\hat u(\ph,j')$ and $\hat r(\ph,j-j',j')$ 
in Fourier series also in the $\ph$ variable, 
we deduce that $\widehat \mR(\ell-\ell')_j^{j'}$ appearing in \eqref{amatriciana}
is the Fourier coefficient 
of frequency $(\ell-\ell',j-j')$
of the function $(\ph,x) \mapsto r(\ph,x,j')$.


 
%


\begin{definition} \label{block decay norm}
{\bf (Matrix decay norm)}
Let $\mR$ be an operator represented by the matrix in \eqref{amatriciana}. 
For $s \geq 0$, we define its \emph{matrix decay norm} 
\begin{equation} \label{def decay norm}
| \mR |_s := \sup_{j' \in \Z^3} 
\Big( \sum_{(\ell,j) \in \Z^{\nu+3}} 
\langle \ell, j-j' \rangle^{2s} \| \widehat \mR(\ell)_j^{j'}\|_{\HS}^2  \Big)^{\frac12}.
\end{equation}
If the operator $\mR = \mR(\lm)$ is $k_0$ times differentiable in $\R^{\nu+3}$, 
we define for $s \geq s_0$
\begin{equation} \label{def decay norm parametri}
| \mR |_s^{k_0,\g} 
:= \sup_{|k| \leq k_0} \gamma^{|k|} \sup_{\lambda \in \R^{\nu+3}} 
|\partial_\lambda^k {\cal R}(\lambda)|_{s - |k|}.
\end{equation}
\end{definition}

%

\begin{remark} \label{rem:matrix norm = pseudo-diff norm}
The definition of the norm $|\mR|_s$ in \eqref{def decay norm} 
is very similar to the one 
of the norm $|\mR|_{0,s,0}$ in \eqref{def norma pseudo-diff}:
the only difference is that the $\sup$ in \eqref{def decay norm} 
is over $j' \in \Z^3$ 
(which is the natural choice when using matrices 
with row and column indices in $\Z^3$), 
while the $\sup$ in \eqref{def norma pseudo-diff} is over $\xi \in \R^3$ 
(which is the natural choice when derivatives of symbols with respect to $\xi$ 
have to be considered to prove composition formulas 
like in Lemma \ref{lemma stime Ck parametri}-$(ii)$).
In fact, the norms $| \cdot |_s$ and $| \cdot |_{0,s,0}$ 
are equivalent for operators acting on periodic functions. 
%
\end{remark}



%

The norm $| \cdot |_s^{k_0, \gamma}$ is increasing, namely
$| \mR |_s^{k_0, \gamma} \leq | \mR |_{s'}^{k_0, \gamma}$ for all $s \leq s'$. 
Moreover $|{\cal R}|_s^{k_0, \gamma} \leq |{\cal R} \langle D \rangle^m|_s^{k_0, \gamma}$
for all $m \geq 0$. 
We now state some 
standard properties of the decay norms 
that are needed for the reducibiity scheme 
of Section \ref{sezione riducibilita a blocchi}. 

\begin{lemma}\label{proprieta standard norma decay}
$(i)$ Let $s \geq s_0$, 
$|{\cal R}|_s^{k_0, \gamma}\,,\,\| u \|_s^{k_0, \gamma} < \infty$. Then 
$$
\| {\cal R} u \|_s^{k_0, \gamma} \lesssim_{s, k_0} |{\cal R}|_{s_0}^{k_0, \gamma} \| u \|_s^{k_0, \gamma} + |{\cal R}|_s^{k_0, \gamma} \| u \|_{s_0}^{k_0, \gamma}\,. 
$$
$(ii)$ Let $s \geq s_0$, 
$|{\cal R}|_s^{k_0, \gamma}, |{\cal Q}|_s^{k_0, \gamma} < \infty$. Then 
$$
|{\cal R}{\cal Q}|_s^{k_0 ,\gamma} \lesssim_{s, k_0} |{\cal R}|_s^{k_0, \gamma} |{\cal Q}|_{s_0}^{k_0, \gamma} + |{\cal R}|_{s_0}^{k_0, \gamma} |{\cal Q}|_{s}^{k_0, \gamma}\,. 
$$
$(iii)$ Let $s \geq s_0$, $|{\cal R}|_s^{k_0, \gamma} < \infty$. 
Then there exists a constant $C(s, k_0) > 0$ such that, for any integer $n \geq 1$, 
\[
|{\cal R}^n|_s^{k_0, \gamma} \leq C(s, k)^{n - 1} (|{\cal R}|_{s_0}^{k_0, \gamma})^{n - 1} |{\cal R}|_s^{k_0, \gamma}. 
\]
$(iv)$ Let $s \geq s_0$, $|{\cal R}|_s^{k_0, \gamma} < \infty$. 
Then there exists $\delta(s, k_0) \in (0, 1)$ small enough such that, 
if $|{\cal R}|_{s_0}^{k_0, \gamma} \leq \delta(s, k_0)$, 
then the map $\Phi = {\rm Id} + {\cal R}$ is invertible 
and the inverse satisfies the estimate 
\[
|\Phi^{- 1} - {\rm Id}|_s^{k_0, \gamma} \lesssim_{s, k_0} |{\cal R}|_s^{k_0, \gamma}. 
\]
\noindent
$(v)$ Let $s \geq s_0$, $|{\cal R}|_s^{k_0, \gamma} < \infty$ and let ${\cal Z}$ be the $3 \times 3$ block-diagonal operator defined by ${\cal Z} = {\rm diag}_{j \in \Z^3} \widehat{\cal R}_j^j(0)$. 
Then $|{\cal Z}|_s^{k_0, \gamma} \lesssim |{\cal R}|_s^{k_0, \gamma}$. 
As a consequence, 
\[
\| \widehat{\cal R}_j^j(0)\|_{\HS}^{k_0, \gamma} \lesssim |{\cal R}|_{s_0}^{k_0, \gamma}\,.
\] 
\end{lemma}

For $N > 0$, we define the operator $\Pi_N {\cal R}$ by means of its $3 \times 3$ block representation in the following way: 
\begin{equation}\label{def proiettore operatori matrici}
(\widehat{\Pi_N {\cal R}})_{j}^{j'}(\ell) := \begin{cases}
\widehat{\cal R}_j^{j'}(\ell) & \text{if } |\ell|, |j - j'| \leq N, \\
0 & \text{otherwise}. 
\end{cases} \qquad \quad 
\text{Moreover, }  \ \Pi_N^\bot {\cal R} := {\cal R} - \Pi_N {\cal R}\,. 
\end{equation}


\begin{lemma}\label{lemma proiettori decadimento}
For all $s, \a \geq 0$, one has 
$|\Pi_N {\cal R}|_{s + \alpha}^{k_0, \gamma} \leq N^\alpha |{\cal R}|_s^{k_0, \gamma}$ and $|\Pi_N^\bot {\cal R}|_s^{k_0, \gamma} \leq N^{- \alpha} |{\cal R}|_{s + \alpha}^{k_0, \gamma}$. 
\end{lemma}

By \eqref{amatriciana}, \eqref{equivalent symbol con phi}, 
as observed in Remark \ref{rem:matrix norm = pseudo-diff norm}, 
the decay norms \ref{block decay norm} 
and the pseudo-differential norms \ref{definizione norma pseudo diff}
are strictly related; in the next lemma (whose proof is a simple check)
we state a link between these norms.


\begin{lemma}\label{norma pseudo norma dec}
Let $s \geq s_0$, 
${\cal R} \in {\cal OPS}^0_{s, 0}$. 
Then $|{\cal R}|_s^{k_0, \gamma} \lesssim |{\cal R}|_{0, s, 0}^{k_0, \gamma}$.
\end{lemma}

\subsection{Real and reversible operators}\label{Reversible operators}

Remember that, for any function $u(\ph,x)$, 
$u \in X$ means $u = \even(\ph,x)$, 
and $u \in Y$ means $u = \odd(\ph,x)$. 

\begin{definition}
$(i)$ We say that a linear operator $\Phi$ is \emph{reversible} 
if $\Phi : X \to Y$ and $\Phi : Y \to X$. 
We say that $\Phi$ is \emph{reversibility preserving} 
if $\Phi : X \to X$ and $\Phi : Y \to Y$. 

\noindent
$(i)$ We say that an operator $\Phi : L^2(\T^3, \C^3) \to L^2(\T^3, \C^3)$ is real if $\Phi(u) \in L^2(\T^3, \R^3)$ for any $u \in L^2(\T^3, \R^3)$. 
\end{definition}

\begin{lemma}
Let $A = {\rm Op}(a) \in {\cal OPS}^{m}_{s, \alpha}$. Then the following holds: 

\noindent
$(i)$ $A$ is reversible if and only if $a(\vphi, x, \xi) = - a(- \vphi, - x, - \xi)$, 
namely $a = \odd(\ph,x,\xi)$;

\noindent
$(ii)$ $A$ is reversibility preserving if and only if $a(\vphi, x, \xi) = a(- \vphi, - x, - \xi)$, 
namely $a = \even(\ph,x,\xi)$.

\noindent
$(iii)$ $A$ is real if and only if $a(\vphi, x, \xi) = \overline{a(\vphi, x, - \xi)}$. 
\end{lemma}
It can be convenient to reformulate real and reversibility properties of linear operators in terms of matrix representation provided in Section \ref{sezione matrici norme}. 
\begin{lemma}
A linear operator ${\cal R}$ is 
\noindent
$(i)$ real if and only if 
$\widehat{\cal R}_{j}^{j'}(\ell) = \overline{\widehat{\cal R}_{- j}^{- j'}(- \ell)}$ 
for all $\ell \in \Z^\nu$, $j, j' \in \Z^3$;

\noindent
$(ii)$ reversible if and only if 
$\widehat{\cal R}_j^{j'}(\ell) = - \widehat{\cal R}_{- j}^{- j'}(- \ell)$ 
for all $\ell \in \Z^\nu$, $j, j' \in \Z^3$;

\noindent
$(iii)$ reversibility-preserving if and only if 
$\widehat{\cal R}_j^{j'}(\ell) =  \widehat{\cal R}_{- j}^{- j'}(- \ell)$ 
for all $\ell \in \Z^\nu$, $j, j' \in \Z^3$.
\end{lemma}

\section{The linearized operator}\label{sez generale L}
In sections \ref{sez generale L}-\ref{sezione riducibilita a blocchi} 
we assume the following ansatz, 
which will be recursively verified along the Nash-Moser iteration. 
We assume that $v \in {\cal C}^\infty(\T^\nu \times \T^3, \R^3)$, $v = {\rm odd}(\vphi, x)$, 
and 
\begin{equation}\label{ansatz}
\| v \|_{\mu_0}^{k_0, \gamma} \leq 1  
\end{equation}
where
$\mu_0 = \mu_0(\nu, \tau, k_0) > 0 $ is large enough. 

\noindent
As we explained at the beginning of Section \ref{sez generale norme e operatori}, from now on we omit to write the dependence on $k_0$ when we write $\lesssim$, namely we write $\lesssim_{s, m}, \lesssim_s \ldots $ instead of $\lesssim_{s, m, k_0}, \lesssim_{s, k_0} \ldots$.

\noindent
Given a function $f$ depending on $u$, we denote by 
\begin{equation}\label{notazione delta 12}
\Delta_{12} f := f(u_1) - f(u_2)\,. 
\end{equation}
We want to study the linearized operator ${\cal L } := {\cal F}'(v)$, 
where $\mF(v)$ is defined in \eqref{equazione cal F vorticita}. 
Using the identity $\Pi_0^\bot = {\rm Id} - \Pi_0$, 
for all $s \geq s_0$ one has 
\begin{equation}\label{operatore linearizzato}
{\cal L} : H^{s + 1}_0 \to H^s_0 , \quad 
{\cal L} = \Pi_0^\bot \big( \omega \cdot \partial_\vphi  + 
\zeta \cdot \nabla + \e a(\vphi, x) \cdot \nabla + \e {\cal R}(\vphi) \big) \Pi_0^\bot,
\end{equation}
where $a(\ph,x)$ is the vector field
\begin{equation}\label{def a}
a(\vphi, x) := (\mU v) (\vphi, x) = {\rm curl}\,\Lambda^{- 1} v\,, 
\end{equation}
$\mU := \curl \Lambda^{-1}$ is defined 
in \eqref{equazione cal F vorticita}, \eqref{def pi0 Lm Lm inv} (see also \eqref{simbolo Lambda}), 
and ${\cal R}(\vphi)$ is a 
matrix-valued linear pseudo-differential operator of order zero, given by 
\begin{equation} \label{definizione cal R}
\begin{aligned}
& {\cal R} := {\cal R}_0 + {\cal R}_{- 1}\,, \\
& {\cal R }_0 h  := M_{\cal U}(\vphi, x) h  - v \cdot \nabla {\cal U}(h)\,, \quad {\cal R}_{- 1} h :=   M_v(\vphi, x){\cal U}(h)  \,, \\
& M_{\cal U}(\vphi, x) := - D {\cal U}(v)(\vphi, x), \quad M_v(\vphi, x) := D v(\vphi, x)
\end{aligned}
\end{equation}
where $D {\cal U}(v)$ is the jacobian matrix of ${\cal U} (v) = {\rm curl}\Lambda^{- 1} v$ and $D v$ is the jacobian matrix of $v$. Note that ${\cal R}_0$ is an operator of order $0$ and ${\cal R}_{- 1}$ is an operator of order $- 1$. Note that since $v = {\rm odd}(\vphi, x)$, then 
\begin{equation}\label{parita reversibilita linearizzato iniziale}
\begin{aligned}
& M_{\cal U} = {\rm odd}(\vphi, x), \quad M_v = {\rm even}(\vphi, x)\,,  \\
& {\cal R}_0, {\cal R}_{- 1} \ 
\text{ are reversible operators}\,. 
\end{aligned}
\end{equation}
We first analyze the operator ${\cal L}$ 
without the projector $\Pi_0^\bot$, namely we consider the operator 
\begin{equation}\label{def cal L (0)}
{\cal L}^{(0)} := \omega \cdot \partial_\vphi  + 
\zeta \cdot \nabla + \e a(\vphi, x) \cdot \nabla + \e {\cal R}(\vphi).
\end{equation}
By the definitions \eqref{definizione cal R} 
and using the product estimates \eqref{p1-pr}, 
for $\| v \|_{s_0 + 1}^{k_0, \gamma} \leq 1$ 
the following tame estimates hold for any $s \geq s_0$: 
\begin{equation} \label{stime tame operatore cal L0}
\begin{aligned}
& \| \e a(\ph,x) \cdot \grad h + \e \mR(\ph) h \|_s^{k_0, \gamma} 
\lesssim_{s} \e \big(\| h \|_{s + 1}^{k_0, \gamma} 
+ \| v \|_{s + 1}^{k_0, \gamma} \| h \|_{s_0 + 1}^{k_0, \gamma} \big).  
\end{aligned}
\end{equation}

%

\section{Reduction to constant coefficients of the highest order term}
\label{sezione riduzione ordine alto}

First we state a Proposition which allows to reduce to constant coefficients the operator 
\begin{equation}\label{def operatore cal T}
{\cal T} := \omega \cdot \partial_\vphi 
+ \big( \zeta + \e a(\vphi, x) \big) \cdot \nabla 
\end{equation}
where we recall that, by \eqref{def a}, one has that
\begin{equation}\label{proprieta a}
\Pi_0 a = 0 \qquad \text{and} \qquad {\rm div}(a) = 0\,. 
\end{equation}

\begin{proposition}\label{proposizione trasporto}
For any $\gamma \in (0, 1)$, $\frak C_0 > 0$, $S > s_0$, $\tau > 0$ 
there exist $\delta = \delta(S, k_0, \tau, \nu)$, 
$\mu_0 = \mu_0 (k_0, \tau, \nu) > 0$ and 
$\tau_1 = \tau_1(k_0, \tau, \nu) > 0$  
such that if \eqref{ansatz} holds and
\begin{equation}\label{condizione piccolezza rid trasporto}
 \quad N_0^{\tau_1} \e \gamma^{- 1} \leq \delta,
\end{equation} 
then the following holds. 
There exists an invertible diffeomorphism 
$\T^3 \to \T^3$, $x \mapsto x + \alpha(\vphi, x; \omega, \zeta)$ 
with inverse $y \mapsto y + \breve \alpha(\vphi, y; \omega, \zeta)$,
defined for all $(\om,\zeta) \in \R^{\nu+3}$, 
satisfying 
\begin{equation}\label{stima alpha trasporto}
\| \alpha \|_s^{k_0, \gamma}, \| \breve \alpha\|_{s}^{k_0, \gamma} \lesssim_{s} N_0^{\tau_0}\e \gamma^{- 1} \| v \|_{s + \mu}^{k_0, \gamma}, \quad \forall s_0 \leq s \leq S
\end{equation}
(with $\t_0$ defined in \eqref{2802.2})
such that, defining 
\begin{equation}\label{def mappa trasporto}
{\cal A} h (\vphi, x) := h(\vphi, x + \alpha(\vphi, x)) \quad \text{with inverse} \quad {\cal A}^{- 1} h(\vphi, y) = h(\vphi, y + \breve \alpha(\vphi, y)),
\end{equation}
one gets the conjugation 
\begin{equation}\label{coniugazione nel teo trasporto}
{\cal A}^{- 1} {\cal T} {\cal A} = \omega \cdot \partial_\vphi + \zeta \cdot \nabla
\end{equation}
for all $(\omega, \zeta) \in DC(\gamma, \tau)$, where 
\begin{equation}\label{def cantor set trasporto}
DC(\gamma, \tau) := \Big\{ \lambda = (\omega, \zeta)  \in \R^{\nu+3} 
: |\omega \cdot \ell + \zeta \cdot j | \geq \frac{{\frak C}_0  \gamma}{\langle \ell, j \rangle^\tau} 
\quad \forall (\ell, j) \in \Z^{\nu + 3} \setminus \{ (0, 0)\}  \Big\}.
\end{equation}
Furthermore $\alpha, \breve \alpha$ are ${\rm odd}(\vphi, x)$, 
and therefore 
${\cal A}, {\cal A}^{- 1}$ are reversibility preserving maps. 
Such maps satisfy the tame estimates 
\begin{equation}\label{stima tame cambio variabile rid trasporto}
\begin{aligned}
& \| {\cal A}^{\pm 1}  h\|_s^{k_0, \gamma} \lesssim_{s} \| h \|_s^{k_0, \gamma} + \| v \|_{s + \mu}^{k_0, \gamma} \| h \|_{s_0}^{k_0, \gamma}\,, \quad \forall s_0 \leq s\leq S\,,  \\
& \| ({\cal A}^{\pm 1} - {\rm Id})  h\|_s^{k_0, \gamma}\,,\, \| ({\cal A}^* - {\rm Id})  h\|_s^{k_0, \gamma} \lesssim_{s} N_0^{\tau_0}\e \gamma^{- 1}\Big( \| v \|_{s_0 + \mu}^{k_0, \gamma}\| h \|_{s + 1}^{k_0, \gamma} + \| v \|_{s + \mu}^{k_0, \gamma} \| h \|_{s_0 + 1}^{k_0, \gamma} \Big)\,, \quad \forall s_0 \leq s\leq S\,.
\end{aligned}
\end{equation}
Let $s_1 \geq s_0$ and assume that $v_1, v_2$ 
satisfy \eqref{ansatz} with $\mu_0 \geq s_1 + \mu$. 
Then for any $\lambda = (\omega, \zeta) \in DC(\gamma, \tau)$ one has 
\begin{equation}\label{stime delta 12 prop trasporto}
\begin{aligned}
& \| \Delta_{12} \alpha \|_{s_1} \,,\, \| \Delta_{12} \breve \alpha\|_{s_1} \lesssim_{s_1} N_0^{\tau_0}\e \gamma^{- 1} \| v_1 - v_2 \|_{s_1 + \mu}\,, \\
& \| \Delta_{12} {\cal A}^{\pm 1} h \|_{s_1}\,,\, \| \Delta_{12} {\cal A}^* h \|_{s_1} \lesssim_{s_1} N_0^{\tau_0} \e \gamma^{- 1} \| v_1 - v_2 \|_{s_1 + \mu} \| h \|_{s_1 + 1}\,. 
\end{aligned}
\end{equation}
\end{proposition} 

\begin{remark}\label{costante diofantea frak C0}
In this section we only need that the constant ${\frak C}_0$ in \eqref{def cantor set trasporto} is strictly bigger than one (see Lemma \ref{inclusione cantor diofantei}). In Lemma \ref{lemma triviale modi alti}, we shall take the constant $\frak C_0 \geq C(\tau)$ large enough. 
\end{remark}

In order to prove Proposition \ref{proposizione trasporto}, 
we closely follow \cite{FGMP}. First we show the following iterative Lemma. 
We fix the constants
\begin{equation}\label{costanti ridu trasporto}
\begin{aligned}
& N_0 > 0, \quad \chi := \frac32, \quad N_{- 1} := 1, \quad N_n := N_0^{\chi^n}, \quad n \geq 0,  \\
& \frak a := 3(\tau_0 + 1) + 2 \,, \quad \frak b := \frak a + 1, 
\end{aligned}
\end{equation}
and recall that $\tau_0$ is defined in Lemma \ref{lemma:WD}.

\begin{lemma}\label{iterazione riducibilita trasporto}
Let $\gamma \in (0, 1)$, $S > s_0$, $\tau > 0$. Then there exist $\delta = \delta(S, k_0, \tau, \nu) \in (0, 1)$, $N_0 = N_0(S, k_0, \tau, \nu) > 0$, $\mu = \mu (k_0, \tau, \nu) > 0$ such that if \eqref{ansatz}, \eqref{condizione piccolezza rid trasporto} are fullfilled with $\mu_0 \geq s_0 + \mu$, then the following statements hold 
for all $n \geq 0$. 

\medskip

\noindent 
There exists a linear operator 
\begin{equation}\label{def cal Tn transport}
{\cal T}_n := \omega \cdot \partial_\vphi + \mathtt m_n \cdot \nabla + a_n(\vphi, x) \cdot \nabla
\end{equation}
defined for any $\lambda = (\omega, \zeta) \in \R^{\nu + 3}$ such that 
\begin{equation}\label{stima mathtt mn an}
\begin{aligned}
& \| a_n \|_{s}^{k_0, \gamma} 
\leq C_* (s) \e N_{n - 1}^{- \frak a} \| v \|_{s + \mu}^{k_0, \gamma}, 
\quad & & 
\| a_n \|_{s + \frak b}^{k_0, \gamma} 
\leq  C_*(s)N_{n - 1} \e  \| v \|_{s + \mu}^{k_0, \gamma} 
\quad \ \forall s_0 \leq s \leq S\,, 
\\  
& |\mathtt m_n - \zeta |^{k_0, \gamma} 
\lesssim \e
\quad & & 
\text{and, if $n \geq 1$, } \ 
|\mathtt m_n - \mathtt m_{n - 1} |^{k_0, \gamma} 
\lesssim \| a_{n - 1}  \|_{s_0}^{k_0, \gamma}
\end{aligned}
\end{equation}
for some constant $C_* (s) = C_*(s, k_0, \tau) > 0$. 
If $n=0$, define ${\cal O}_0^\gamma := \R^{\nu+3}$; 
if $n\geq 1$, define 
\begin{equation}\label{cal On gamma}
{\cal O}_n^\gamma := \Big\{ (\omega, \zeta) \in {\cal O}_{n - 1}^\gamma : 
|\omega \cdot \ell + \mathtt m_{n - 1} \cdot j| \geq \frac{\gamma}{\langle \ell, j \rangle^\tau} \quad \forall (\ell, j) \in \Z^{\nu + 3} \setminus \{0, 0\}, 
\quad |(\ell, j)| \leq N_{n - 1} \Big\}\,.
\end{equation}
For $n \geq 1$, there exists an invertible diffeomorphism of the torus 
$\T^3 \to \T^3$, $x \mapsto x + \alpha_{n - 1}(\vphi, x)$ 
with inverse $\T^3 \to \T^3$, $y \mapsto y + \widetilde \alpha_{n - 1}(\vphi, y)$ 
such that 
\begin{equation}\label{stime alpha n tilde alpha n}
\begin{aligned}
& \| \alpha_{n - 1}\|_s^{k_0, \gamma}, \| \breve \alpha_{n - 1}\|_s^{k_0, \gamma} \lesssim_{s} N_{n - 1}^{\tau_0} N_{n - 2}^{- \frak a} \e \gamma^{- 1} \| v\|_{s + \mu}^{k_0, \gamma}, \quad \forall s_0 \leq s \leq S\,, \\
&  \| \alpha_{n - 1}\|_{s + \frak b}^{k_0, \gamma}, \| \breve \alpha_{n - 1}\|_{s + \frak b}^{k_0, \gamma} \lesssim_{s} N_{n - 1}^{\tau_0} N_{n - 2} \e \gamma^{- 1} \| v\|_{s + \mu}^{k_0, \gamma},  \quad \forall s_0 \leq s \leq S \,. 
\end{aligned}
\end{equation}
(with $\t_0$ defined in \eqref{2802.2}). 
The operator 
$$
{\cal A}_{n - 1} : h(\vphi, x) \mapsto h(\vphi, x + \alpha_{n - 1}(\vphi, x))
$$
with inverse 
$$
{\cal A}_{n - 1}^{- 1} : h(\vphi, x) \mapsto h(\vphi, y + \breve \alpha_{n - 1}(\vphi, y))
$$
satisfy, for any $\lambda = (\omega, \zeta) \in {\cal O}_n^\gamma$, the conjugation
\begin{equation}\label{coniugazione cal A n - 1 cal T n - 1}
{\cal T}_n = {\cal A}_{n - 1}^{- 1}{\cal T}_{n - 1} {\cal A}_{n - 1}\,. 
\end{equation}
Furthermore, $a_n = {\rm even}(\vphi, x)$, $\alpha_{n - 1}, \breve \alpha_{n - 1} = {\rm odd}(\vphi,x)$, implying that ${\cal T}_n$ is a reversible operator and ${\cal A}_{n - 1}, {\cal A}_{n - 1}^{- 1}$ are reversibility preserving operators. 
%
%
\end{lemma}

\begin{proof}
{\sc Proof of the statement for $n = 0$.} 
The claimed statements for $n = 0$ follows directly by defining 
$a_0 := \e a = \e \,{\rm curl}\,\Lambda^{- 1} v$, see \eqref{def a}.

\medskip

\noindent
{\sc Proof of the induction step.}
Now assume that the claimed properties hold for some $n \geq 0$ and let us prove them at the step $n + 1$. We look for a diffeomorphism of the torus $\T^3 \to \T^3$, $x \mapsto x + \alpha_n(\vphi, x)$, wih inverse given by $y \mapsto y + \breve \alpha_n(\vphi, y)$ such that defining 
$$
{\cal A}_n : h(\vphi, x) \mapsto h(\vphi, x + \alpha_n(\vphi, x)), \quad {\cal A}_n^{- 1} : h(\vphi, x) \mapsto h(\vphi, x + \breve \alpha_n(\vphi, x))
$$
the operator 
${\cal A}_n^{- 1} {\cal T}_n {\cal A}_n$ 
has the desired properties. One computes 
\begin{equation}\label{primo cal L n + 1}
\begin{aligned}
{\cal A}_n^{- 1} {\cal T}_n {\cal A}_n
& = \omega \cdot \partial_\vphi + \mathtt m_n \cdot \nabla 
+ {\cal A}_n^{- 1} \big[ \omega \cdot \partial_\vphi \alpha_n + \mathtt m_n \cdot \nabla \alpha_n + a_n + a_n \cdot \nabla \alpha_n \big] \cdot \nabla  \\
& = \omega \cdot \partial_\vphi + \mathtt m_n \cdot \nabla 
+ {\cal A}_n^{- 1} \big[ \omega \cdot \partial_\vphi \alpha_n + \mathtt m_n \cdot \nabla \alpha_n + \Pi_{N_n}a_n \big] \cdot \nabla  + a_{n + 1} \cdot \nabla
\end{aligned}
\end{equation}
where 
\begin{equation}\label{def a n + 1}
a_{n + 1} := {\cal A}_n^{- 1} f_n, \qquad 
f_n :=  \Pi_{N_n}^\bot a_n +  a_n \cdot \nabla \alpha_n,
\end{equation}
and the projectors $\Pi_{N_n}, \Pi_{N_n}^\bot$ are defined by \eqref{def:smoothings}. 
For any $(\omega, \zeta) \in {\cal O}_{n + 1}^\gamma$, we solve the homological equation 
\begin{equation}\label{eq omologica trasporto}
\omega \cdot \partial_\vphi \alpha_n + \mathtt m_n \cdot \nabla \alpha_n + \Pi_{N_n}a_n = \langle a_n \rangle_{\vphi, x}
\end{equation}
(where $\la a_n \ra_{\ph,x}$ is the average of $a_n$ in time and space),
and, recalling \eqref{def ompaph-1 ext}, 
we extend its solution to the whole parameter space $(\om,\zeta) \in \R^{\nu+3}$
by defining 
\begin{equation}\label{def alpha n}
\alpha_n:= \big( \omega \cdot \partial_\vphi + \mathtt m_n \cdot \nabla \big)_{ext}^{- 1}
\big[ \langle a_n \rangle_{\vphi, x} - \Pi_{N_n}a_n\big]\,.
\end{equation}
We define 
\begin{equation}\label{secondo cal L n + 1}
{\cal T}_{n + 1} 
:= \omega \cdot \partial_\vphi + \mathtt m_{n + 1} \cdot \nabla + a_{n + 1} \cdot \nabla 
\end{equation}
where 
\begin{equation}\label{def mathtt m n + 1}
\mathtt m_{n + 1} := \mathtt m_n + \langle a_n \rangle_{\vphi, x}\,.
\end{equation}
We observe that ${\cal T}_{n+1}$ is defined for all $(\om,\zeta) \in \R^{\nu+3}$, 
and, for $(\om,\zeta) \in {\cal O}^\g_{n+1}$, 
one has $\mA_n^{-1} {\cal T}_n \mA_n = {\cal T}_{n+1}$.
Clearly $\alpha_n (\vphi, x; \omega, \zeta)$ is ${\cal C}^\infty$ in $(\vphi, x)$ 
and $k_0$ times differentiable in $(\omega, \zeta) \in \R^{\nu + 3}$. Furthermore, by Lemma \ref{lemma:WD}, and by the smoothing property \eqref{p2-proi}, for any $s \geq 0$, one has 
\begin{equation}\label{prima stima alpha n an}
\begin{aligned}
& \| \alpha_n \|_s^{k_0, \gamma} \lesssim  \gamma^{- 1}\| \Pi_{N_n} a_n \|_{s + \tau_0}^{k_0, \gamma} \lesssim N_n^{\tau_0} \gamma^{- 1} \| a_n\|_s^{k_0, \gamma},  \\
& \| \nabla \alpha_n\|_s^{k_0, \gamma} \lesssim \gamma^{- 1}\| \Pi_{N_n} a_n \|_{s + \tau_0 + 1}^{k_0, \gamma}  \lesssim N_n^{\tau_0 + 1} \gamma^{- 1} \| a_n\|_s^{k_0, \gamma}\,. 
\end{aligned}
\end{equation}
The latter estimate, together with Lemma \ref{lemma:LS norms} and the induction estimates on \eqref{stima mathtt mn an} on $a_n$, imply that for any $s_0 \leq s \leq S$
\begin{equation}\label{stime alpha n + 1}
\begin{aligned}
& \| \alpha_{n } \|_s^{k_0, \gamma}\,,\,\| \breve \alpha_{n } \|_s^{k_0, \gamma} \lesssim_{s} N_{n }^{\tau_0} N_{n - 1}^{- \frak a} \e \gamma^{- 1} \| v\|_{s + \mu}^{k_0, \gamma} \,, \\
& \| \alpha_{n } \|_{s + \frak b}^{k_0, \gamma} \,,\, \| \breve \alpha_{n } \|_{s + \frak b}^{k_0, \gamma}\lesssim_{s} N_n^{\tau_0} N_{n - 1} \e \gamma^{- 1} \| v\|_{s + \mu}^{k_0, \gamma}
\end{aligned}
\end{equation}
which are the estimates \eqref{stime alpha n tilde alpha n} at the step $n + 1$. 
Note that, using the definition of the constant $\frak a$ in \eqref{costanti ridu trasporto} 
and the ansatz \eqref{ansatz}, 
from \eqref{stime alpha n + 1} with $s=s_0$ one deduces that
\begin{equation}\label{stima alpha n s0}
 \| \alpha_{n } \|_{s_0}^{k_0, \gamma}\,,\,\| \breve \alpha_{n } \|_{s_0}^{k_0, \gamma} \lesssim N_0^{\tau_0} \e \gamma^{- 1}\,.
\end{equation}
 Hence the smallness condition \eqref{condizione piccolezza rid trasporto} (choosing $\tau_1 > \tau_0$), together with Lemma \ref{lemma:LS norms} and the estimate \eqref{prima stima alpha n an} leads to the estimate
\begin{equation}\label{stima cal An pm 1}
\| {\cal A}_n^{\pm 1} h \|_s^{k_0, \gamma} \lesssim_{s} \| h \|_s^{k_0, \gamma} + N_n^{\tau_0} \gamma^{- 1} \| a_n\|_s^{k_0, \gamma} \| h \|_{s_0}^{k_0, \gamma}\,, \quad \forall s_0 \leq s \leq S + \frak b\,.  
\end{equation}
We now estimate the function $a_{n + 1}$ defined in \eqref{def a n + 1}. 
First, we estimate $f_{n }$. 
By \eqref{costanti ridu trasporto}, \eqref{stima mathtt mn an} and using also the ansatz \eqref{ansatz} and the smallness condition \eqref{condizione piccolezza rid trasporto}, 
one has that 
\begin{equation}\label{a n + 1 Nn gamma}
N_n^{\tau_0 + 1} \gamma^{- 1} \| a_n \|_{s_0}^{k_0, \gamma} \leq 1\,.
\end{equation} 
By \eqref{p1-pr}, \eqref{p2-proi}, \eqref{p3-proi}, 
\eqref{prima stima alpha n an}, \eqref{a n + 1 Nn gamma}
one has 
\begin{equation}\label{stima f n + 1 trasporto}
\begin{aligned}
\| f_n \|_{s_0}^{k_0, \gamma} & \lesssim \| a_n \|_{s_0}^{k_0, \gamma}\,, \\
\| f_{n } \|_s^{k_0, \gamma} & \lesssim_{s} N_n^{- \frak b} \| a_n\|_{s + \frak b}^{k_0, \gamma} + N_n^{\tau_0 + 1} \gamma^{- 1} \| a_n\|_s^{k_0, \gamma} \| a_n\|_{s_0}^{k_0, \gamma}\,, 
\quad \forall s_0 < s \leq S, \\
\| f_{n }\|_{s + \frak b}^{k_0, \gamma} 
& \lesssim_{s} \| a_n\|_{s + \frak b}^{k_0, \gamma} 
\big(1 + N_n^{\tau_0 + 1} \g^{-1} \| a_n\|_{s_0}^{k_0, \gamma} \big) 
\lesssim_{s} 
\| a_n\|_{s + \frak b}^{k_0, \gamma}, \quad \forall s_0 \leq s \leq S. 
\end{aligned}
\end{equation}
Hence \eqref{stima cal An pm 1}-\eqref{stima f n + 1 trasporto} 
imply that, for any $s_0 \leq s \leq S$, 
\begin{equation}\label{stime induttive a n + 1}
\begin{aligned}
\| a_{n + 1} \|_s^{k_0, \gamma} & \lesssim_{s} N_n^{- \frak b} \| a_n\|_{s + \frak b}^{k_0, \gamma} + N_n^{\tau_0 + 1} \gamma^{- 1} \| a_n\|_s^{k_0, \gamma} \| a_n\|_{s_0}^{k_0, \gamma}\,, \\
\| a_{n + 1}\|_{s + \frak b}^{k_0, \gamma} & \lesssim_{s} \| a_n \|_{s + \frak b}^{k_0, \gamma} 
\end{aligned}
\end{equation}
and using the definition of the constants $\frak a, \frak b$ in \eqref{costanti ridu trasporto} and the induction estimates on $a_n$ one deduces the estimate \eqref{costanti ridu trasporto} for $a_{n + 1}$. The estimates \eqref{stima mathtt mn an} for $\mathtt m_{n + 1}$ follows by its definition \eqref{def mathtt m n + 1}, by the induction estimate on $a_n$ and by using a telescoping argument. 

Finally, by \eqref{def alpha n}, since $a_n = {\rm even}(\vphi, x)$, then $\alpha_n , \breve \alpha_n = {\rm odd}(\vphi, x)$. This implies that the maps ${\cal A}_n^{\pm 1}$ are reversibility preserving and therefore, by recalling \eqref{def a n + 1}, $a_{n + 1} = {\rm even}(\vphi, x)$ and the claimed statement is proved. 
\end{proof}

We then define 
\begin{equation}\label{def widetilde cal An}
\widetilde{\cal A}_n := {\cal A}_0 \circ {\cal A}_1 \circ \ldots \circ {\cal A}_n, \quad \text{with inverse} \quad \widetilde{\cal A}_n^{- 1} = {\cal A}_n^{- 1} \circ {\cal A}_{n - 1}^{- 1} \circ \ldots \circ {\cal A}_0^{- 1}\,. 
\end{equation}

\begin{lemma}\label{lemma tilde cal An}
Let $S > s_0$, $\gamma \in (0, 1)$. Then there exist $\delta = \delta(S, k_0, \tau, \nu) \in (0, 1)$, $\mu = \mu(k_0, \tau, \nu) > 0$ such that if \eqref{ansatz} holds with $\mu_0 \geq s_0 + \mu$ and if \eqref{condizione piccolezza rid trasporto} holds, then the following properties hold.

\noindent
$(i)$
$$
\widetilde{\cal A}_n h (\vphi, x) = h(\vphi, x + \beta_n(\vphi, x))\,, \quad \widetilde{\cal A}_n^{- 1} h(\vphi, y) = h(\vphi, x + \breve \beta_n(\vphi, x))
$$ 
where, for any $s_0 \leq s \leq S$, 
\begin{equation}\label{stima beta n beta n - 1 trasporto}
\begin{aligned}
\| \beta_0 \|_{s}^{k_0, \gamma}\,,\, \| \breve \beta_0\|_{s}^{k_0, \gamma} 
& \lesssim_{s} N_{0}^{\tau_0} \e \gamma^{- 1}  \| v \|_{s + \mu}^{k_0, \gamma} \,,
\\
\| \beta_n - \beta_{n - 1} \|_s^{k_0, \gamma}\,,\, 
\|\breve \beta_n - \breve \beta_{n - 1} \|_s^{k_0, \gamma} 
& \lesssim_{s} N_{n }^{\tau_0} N_{n - 1}^{- \frak a}\e \gamma^{- 1}  \| v \|_{s + \mu}^{k_0, \gamma}, \quad n \geq 1. 
\end{aligned}
\end{equation}
As a consequence, 
\begin{equation}\label{bound solo beta n}
\| \beta_n \|_s^{k_0, \gamma} \lesssim_{s} N_0^{\tau_0} \e \gamma^{- 1}\| v \|_{s + \mu}^{k_0, \gamma}, \quad \forall s_0 \leq s \leq S\,. 
\end{equation}
Furthermore, $\beta_n, \breve \beta_n = {\rm odd}(\vphi, x)$. 

\noindent
$(ii)$ For any $s_0 \leq s \leq S$, 
the sequence $(\beta_n)_{n \in \N}$ (resp. $ (\breve \beta_n)_{n \in \N}$) 
is a Cauchy sequence with respect to the norm $\| \cdot \|_s^{k_0, \gamma}$ 
and it converges to some limit $\alpha$ (resp. $\breve \alpha$). 
Furthermore $\alpha, \breve \alpha = {\rm odd}(\vphi, x)$ 
and, for any $s_0 \leq s \leq S$, $n \geq 0$, one has
\begin{equation}\label{stima beta n - beta infty}
\begin{aligned}
\| \alpha- \beta_n\|_s^{k_0, \gamma}\,,\, \| \breve \alpha - \breve \beta_n\|_s^{k_0, \gamma} 
& \lesssim_{s} N_{n + 1}^{\tau_0} N_{n }^{1 - \frak a}\e \gamma^{- 1} \| v \|_{s + \mu}^{k_0, \gamma}  \,, \\
\| \alpha \|_s^{k_0, \gamma}, \| \breve \alpha \|_s^{k_0, \gamma} 
& \lesssim_{s} N_0^{\tau_0}\e \gamma^{- 1}  \| v \|_{s + \mu}^{k_0, \gamma} \,.
\end{aligned}
\end{equation}

\noindent
$(iii)$ Define 
\begin{equation}\label{def cal A infty pm 1}
{\cal A} h (\vphi, x) := h(\vphi, x + \alpha(\vphi, x)), \quad \text{with inverse} \quad {\cal A}^{- 1} h(\vphi, y) = h(\vphi, y + \breve \alpha(\vphi, y))\,. 
\end{equation}
Then, for any $s_0 \leq s \leq S$, $\widetilde{\cal A}_n^{\pm 1}$ converges pointwise in $H^s$ to ${\cal A}^{\pm 1}$, namely $\lim_{n \to + \infty} \| {\cal A}^{\pm 1 } h - \widetilde{\cal A}_n^{\pm 1} h \|_s = 0$ for any $h \in H^s$. 
\end{lemma}

\begin{proof}
{\sc Proof of $(i)$.} We prove the Lemma arguing by induction. For $n = 0$, one has that $\widetilde{\cal A}_0^{\pm 1} = {\cal A}_0^{\pm 1}$ and we set $\beta_0 := \alpha_0$, $\breve \beta_0 := \breve \alpha_0$. Then the first estimate \eqref{stima beta n beta n - 1 trasporto} follows by \eqref{stime alpha n tilde alpha n} (applied with $n = 1$). The second statement in \eqref{stima beta n beta n - 1 trasporto} for $n = 0$ is empty. Now assume that the claimed statement holds for some $n \geq 0$ and let us prove it at the step $n + 1$. We prove the claimed statement for $\widetilde{\cal A}_{n + 1}$ since the proof for the map $\widetilde{\cal A}_{n + 1}^{- 1}$ is similar. Using that $\widetilde{\cal A}_{n + 1} := \widetilde{\cal A}_n \circ {\cal A}_{n + 1}$, one computes that 
\begin{equation}\label{definizione induttiva beta n}
\begin{aligned}
\widetilde{\cal A}_{n + 1} h(\vphi, x) = h(\vphi, x + \beta_{n + 1}(\vphi, x))\,,\quad 
\beta_{n + 1} := \beta_{n } + \widetilde{\cal A}_{n}[\alpha_{n + 1}]\,. 
\end{aligned}
\end{equation}
We apply Lemma \ref{iterazione riducibilita trasporto}. 
Since $\alpha_{n + 1} = {\rm odd}(\vphi, x)$ 
and, by the induction hypothesis $\beta_n = {\rm odd}(\vphi, x)$, 
$\widetilde{\cal A}_n$ is reversibility preserving, 
one has $\beta_{n + 1} = {\rm odd}(\vphi, x)$. 
By the induction estimate \eqref{bound solo beta n} for $s = s_0$, using the ansatz \eqref{ansatz}, $\| \beta_n \|_{s_0}^{k_0, \gamma} \lesssim N_0^{\tau_0} \e \gamma^{- 1}$. Then, by the smallness condition \eqref{condizione piccolezza rid trasporto}, 
we can apply Lemma \ref{lemma:LS norms} 
and \eqref{stime alpha n tilde alpha n}, 
\eqref{ansatz}, 
\eqref{condizione piccolezza rid trasporto}, 
\eqref{bound solo beta n}, 
obtaining that, for any $s_0 \leq s \leq S$, 
\begin{equation}\label{stima beta n + 1 beta n}
\begin{aligned}
\| \beta_{n + 1} - \beta_n \|_s^{k_0, \gamma} 
& \lesssim_{s} \| \alpha_{n + 1} \|_s^{k_0, \gamma} 
+ \| \beta_n \|_s^{k_0, \gamma} \| \alpha_{n + 1} \|_{s_0}^{k_0, \gamma}  
\lesssim_{s} N_{n + 1}^{\tau_0} N_n^{- \frak a} \e \gamma^{- 1} \| v \|_{s + \mu}^{k_0, \gamma}, 
\end{aligned}
\end{equation}
which is \eqref{stima beta n beta n - 1 trasporto} at the step $n + 1$. 
The estimate \eqref{bound solo beta n} at the step $n + 1$ follows by using a telescoping argument, since the series $\sum_{n \geq 0} N_n^{\tau_0} N_{n - 1}^{- \frak a} < \infty$.

\medskip

\noindent
{\sc Proof of $(ii)$.} It follows by item $(i)$, using the estimate \eqref{stima beta n beta n - 1 trasporto} and a telescoping argument. 

\noindent
{\sc Proof of $(iii)$.} It follows by item $(ii)$, using the same arguments of the proof of Lemma B6-$(i)$ in \cite{BHM}: 
given $h \in H^s, \e_1 > 0$, 
there exists $N > 0$ (sufficiently large, depending on $\e_1,s,h$)   
such that $h_1 := \Pi_N^\bot h$ satisfies 
$\| \mA h_1 \|_s \leq \e_1 /4$, 
$\| \tilde \mA_n h_1 \|_s \leq \e_1 /4$
uniformly in $n$ (bound \eqref{bound solo beta n} is uniform in $n$).  
On the other hand, $h_0 := \Pi_N h$ satisfies 
\[ 
\begin{aligned}
\| (\mA - \tilde \mA_n) h_0 \|_s 
& \leq \int_0^1 \frac{d}{d\theta} \, 
h_0 \big( \ph, x + \b_n(\ph,x) + \theta [\a(\ph,x) - \b_n(\ph,x)] \big) \, d\theta 
\\ 
& \lesssim_s 
\| \grad h_0 \|_s \| \a - \b_n \|_{s_0} + \| \grad h_0 \|_{s_0} \| \a - \b_n \|_s
\leq \e_1 / 2  
\end{aligned}
\]
for all $n \geq n_0$, for some $n_0$ depending on $\e_1, s, h$. 
\end{proof}

\begin{lemma}\label{inclusione cantor diofantei}
$(i)$ The sequence $(\mathtt m_n)_{n \in \N}$ satisfies the bound 
$$
{\rm sup}_{(\omega, \zeta) \in {\cal O}_n^\gamma} |\mathtt m_n(\omega, \zeta) - \zeta| \lesssim \e N_{n - 1}^{- \frak a} \,.
$$
$(ii)$ The following inclusion holds: $DC(\gamma, \tau) \subseteq \cap_{n \geq 0} {\cal O}_n^\gamma$ (recall the definitions \eqref{def cantor set trasporto}, \eqref{cal On gamma}).
\end{lemma}

\begin{proof}
{\sc Proof of $(i)$.} By recalling the definition \eqref{def widetilde cal An}, using \eqref{def cal Tn transport}, \eqref{coniugazione cal A n - 1 cal T n - 1}, one obtains  
\begin{equation}\label{A tilde n cal Tn}
\omega \cdot \partial_\vphi + \mathtt m_n \cdot \nabla + a_n \cdot \nabla = {\cal T}_n = \widetilde{\cal A}_{n - 1}^{- 1}{\cal T}_0 \widetilde{\cal A}_{n - 1}, \quad \forall (\omega, \zeta) \in {\cal O}_n^\gamma
\end{equation}
and by Lemma \ref{lemma tilde cal An}-$(i)$ one computes explicitely
\begin{equation}\label{A tilde n cal Tn 1}
\widetilde{\cal A}_{n - 1}^{- 1}{\cal T}_0 \widetilde{\cal A}_{n - 1} 
= \omega \cdot \partial_\vphi + \zeta \cdot \nabla 
+ \widetilde {\cal A}_{n - 1}^{- 1} \big( \omega \cdot \partial_\vphi \beta_{n - 1} + \zeta \cdot \nabla \beta_{n - 1} + \e \, a 
+ \e \, a  \cdot \nabla \beta_{n - 1} \big) \cdot \nabla \,.
\end{equation}
Since $\widetilde{\cal A}_{n - 1}$ is a change of variable, 
one has $\widetilde{\cal A}_{n - 1}[c] = c$ for all constant $c \in \R$. 
Hence, by \eqref{A tilde n cal Tn}, \eqref{A tilde n cal Tn 1}, 
one obtains the identity
\begin{equation}\label{uguaglianza mn an betan}
\begin{aligned}
\mathtt m_n + \widetilde{\cal A}_{n - 1} a_n = \zeta +  \omega \cdot \partial_\vphi \beta_{n - 1} + \zeta \cdot \nabla \beta_{n - 1} + \e a + \e a \cdot \nabla \beta_{n - 1}\,. 
\end{aligned}
\end{equation}
Note that 
\begin{equation}\label{media divergenza con beta n - 1}
\begin{aligned}
& \int_{\T^{\nu + 3}} \big( \omega \cdot \partial_\vphi \beta_{n - 1} + \zeta \cdot \nabla \beta_{n - 1} \big)\, d \vphi\, d x = 0\,, \quad   \int_{\T^{\nu + 3}} a(\vphi, x)\, d\vphi \, d x \stackrel{\eqref{proprieta a}}{=} 0\,, \\
& \int_{\T^{\nu + 3}} a \cdot \nabla \beta_{n - 1}\, d \vphi\, d x = - \int_{\T^{\nu + 3}} {\rm div}(a) \beta_{n - 1}\, d \vphi\, d x \stackrel{\eqref{proprieta a}}{=} 0\,.
\end{aligned}
\end{equation}
Taking the space-time average of the equation \eqref{uguaglianza mn an betan}, 
and using \eqref{media divergenza con beta n - 1}, 
we deduce that 
\begin{equation}\label{mathtt mn an}
\mathtt m_n - \zeta = - \langle \widetilde{\cal A}_{n - 1} a_n \rangle_{\vphi, x} 
\quad \forall (\omega, \zeta) \in {\cal O}_n^\gamma\,. 
\end{equation}
The claimed estimate then follows by Lemma \ref{lemma:LS norms}, 
applying \eqref{stima mathtt mn an}, 
\eqref{bound solo beta n}, 
\eqref{ansatz}, 
\eqref{condizione piccolezza rid trasporto}. 

\medskip

\noindent
{\sc Proof of $(ii)$.} We prove the claimed inclusion by induction, i.e. we show that $DC(\gamma, \tau) \subseteq {\cal O}_n^\gamma$ for any $n \geq 0$. 
For $n = 0$, the inclusion holds since ${\cal O}_0^\gamma := \R^{\nu+3}$ 
and $DC(\gamma, \tau) \subseteq \R^{\nu+3}$ (see \eqref{def cantor set trasporto}).  
Now assume that $DC(\gamma, \tau) \subseteq {\cal O}_n^\gamma$ for some $n \geq 0$ 
and let us prove that $DC(\gamma, \tau) \subseteq {\cal O}_{n + 1}^\gamma$. 
Let $(\omega, \zeta) \in DC(\gamma, \tau)$. 
By the induction hypothesis, $(\omega, \zeta)$ belongs to ${\cal O}_n^\gamma$. 
Therefore, by item $(i)$, one has 
$|\mathtt m_n(\omega, \zeta) - \zeta| \lesssim \e N_{n - 1}^{- \frak a}$. 
Hence, for all $(\ell, j) \in \Z^{\nu + 3} \setminus \{(0,0)\}$, 
$|(\ell, j)| \leq N_n$, one has 
$$
\begin{aligned}
|\omega \cdot \ell + \mathtt m_n(\omega, \zeta) \cdot j| & \geq |\omega \cdot \ell + \zeta \cdot j| - |(\mathtt m_n(\omega, \zeta) - \zeta) \cdot j| \\
& \geq \frac{{\frak C}_0 \gamma}{\langle \ell, j \rangle^\tau} - \e C N_n N_{n - 1}^{- \frak a} 
\geq \frac{\gamma}{\langle \ell, j \rangle^\tau}
\end{aligned}
$$
provided $C N_n^{1+\tau} N_{n - 1}^{- \frak a} \e \gamma^{- 1} \leq \frak C_0 - 1$.
This holds for all $n \geq 0$ provided 
\begin{equation} \label{provided trasporto}
C N_0^{1+\tau} \e \g^{-1} \leq \frak C_0 - 1.
\end{equation}
Condition \eqref{provided trasporto} is fullfilled by taking $\frak C_0 \geq 2$, 
using 
\eqref{costanti ridu trasporto} 
and the smallness condition \eqref{condizione piccolezza rid trasporto}. 
Thus, by the definition of ${\cal O}_{n + 1}^\gamma$ (see \eqref{cal On gamma}),
one has that $(\omega, \zeta) \in {\cal O}_{n + 1}^\gamma$, 
and the proof is concluded. 
\end{proof}


\begin{proof}[Proof of Proposition \ref{proposizione trasporto}]
For any $(\om,\zeta) \in DC(\gamma, \tau)$, 
by Lemma \ref{inclusione cantor diofantei}, 
$\mathtt{m}_n \to \zeta$ as $n \to \infty$. 
By \eqref{bound solo beta n}, 
\eqref{stima mathtt mn an}
and Lemma \ref{lemma:LS norms}, 
one has $\| \widetilde \mA_{n-1} a_n \|_{s_0} \lesssim \| a_n \|_{s_0} \to 0$ 
Also, 
\[
\| \pa_\ph \b_{n-1} - \pa_\ph \a \|_{s_0} , \ 
\| \grad \b_{n-1} - \grad \a \|_{s_0} 
\leq \| \b_{n-1} - \a \|_{s_0+1} 
\to 0 \quad \ (n \to \infty).
\]
Hence, passing to the limit in norm $\| \ \|_{s_0}$ 
in the identity \eqref{uguaglianza mn an betan}, we obtain the identity
\begin{equation} \label{1303.1}
\ompaph \a + \zeta \cdot \grad \a + \e a(\ph,x) + \e a(\ph,x) \cdot \grad \a = 0
\end{equation}
in $H^{s_0}(\T^{\nu+3})$, 
and therefore pointwise for all $(\ph,x) \in \T^{\nu+3}$,
for any $(\om,\zeta) \in DC(\g,\t)$.
As a consequence, 
\begin{align*}
{\cal A}^{- 1} {\cal T} {\cal A} 
& = \omega \cdot \partial_\vphi + \zeta \cdot \nabla 
+ \{ \mA^{-1} \big( \ompaph \a + \zeta \cdot \grad \a + \e a + \e a \cdot \grad \a \big) \} 
\cdot \grad 
= \omega \cdot \partial_\vphi + \zeta \cdot \nabla 
\end{align*}
for all $(\omega, \zeta) \in DC(\gamma, \tau)$, 
which is \eqref{coniugazione nel teo trasporto}.
The estimates \eqref{stima alpha trasporto}, \eqref{stima tame cambio variabile rid trasporto} 
follow from the estimates \eqref{stima beta n - beta infty} and by Lemma \ref{lemma:LS norms}. 

It remains only to prove the estimate \eqref{stime delta 12 prop trasporto}. 
Let $a_i := a_i(v_i)$, $i = 1, 2$ satisfy \eqref{proprieta a} 
and assume that, for $s_1 > s_0$, $v_1, v_2$ satisfy \eqref{ansatz} 
with $\mu_0 \geq s_1 + \mu$. 
Let $\a_i$ and $\mA_i$, $i=1,2$, 
be the corresponding function and operator
given by Lemma \ref{lemma tilde cal An}-$(ii),(iii)$. 
Then, by \eqref{1303.1} and \eqref{coniugazione nel teo trasporto}, 
for $i=1,2$ one has 
\begin{equation} \label{1303.2}
{\cal A}_i^{- 1} {\cal T}_i {\cal A}_i = L_0,  \quad \ 
{\cal T}_i (\alpha_i) + a_i = 0 \quad 
\forall (\omega, \zeta) \in DC(\gamma, \tau),
\end{equation}
where $L_0 := \omega \cdot \partial_\vphi + \zeta \cdot \nabla$ 
and ${\cal T}_i := L_0 + \e a_i(\ph,x) \cdot \grad$. 
Hence
$$
{\cal T}_1 (\alpha_1 - \alpha_2) + f = 0, \quad f := (a_1 - a_2) \cdot \nabla \alpha_2 + a_1 - a_2\,. 
$$
By \eqref{1303.2} one has ${\cal T}_1 = \mA_1 L_0 \mA_1^{-1}$, 
and therefore
$$
L_0 {\cal A}_1^{- 1} (\alpha_1 - \alpha_2) + {\cal A}_1^{- 1}(f) = 0\,. 
$$
Since $\langle {\cal A}_1^{- 1}(f) \rangle_{\vphi, x} = - \langle L_0 {\cal A}_1^{- 1} (\alpha_1 - \alpha_2) \rangle_{\vphi, x} = 0$, $(\omega, \zeta) \in DC(\gamma, \tau)$ and $\alpha_1 - \alpha_2, {\cal A}_1^{- 1}(\alpha_1 - \alpha_2) = {\rm odd}(\vphi, x)$ (the operator $L_0$ has only the trivial kernel, restricted to the space of odd functions in $(\vphi, x)$), one has 
$$
\alpha_1 - \alpha_2 = - {\cal A}_1 L_0^{- 1}{\cal A}_1^{- 1}(f), \quad \forall (\omega, \zeta) \in DC(\gamma ,\tau)\,. 
$$
Then, using that $\| v_i \|_{s_1 + \mu} \leq 1$, 
by applying \eqref{stima tame cambio variabile rid trasporto}, \eqref{stima alpha trasporto} 
and the product 
estimate \eqref{p1-pr}, 
recalling also that by \eqref{def a}, $a_i = {\rm curl}\,\Lambda^{- 1} v_i$, $i = 1, 2$, 
one gets the estimate \eqref{stime delta 12 prop trasporto} for $\Delta_{12} \alpha$. 
The corresponding estimate for $\Delta_{12} \breve \alpha$ can be done by using that $\breve \alpha_i = - {\cal A}_i^{- 1}(\alpha_i)$ and using the mean value theorem. 
Finally, the estimate for $\Delta_{12} {\cal A}^{\pm 1}, \Delta_{12} {\cal A}^*$ 
follow by using the estimate for $\Delta_{12} \alpha$ and $\Delta_{12} \breve \alpha$, 
using the mean value theorem and Lemma \ref{lemma:LS norms}. 
\end{proof}

In the next Lemma we exploit the conjugation of the operator ${\cal L}^{(0)}$ defined in \eqref{def cal L (0)} by means of the map ${\cal A}$ constructed in Proposition \ref{proposizione trasporto}. With a slight abuse of notations we denote with the same letter ${\cal A}$ the operator acting on $H^s(\T^\nu \times \T^3, \R)$ and acting on $H^s(\T^\nu \times \T^3, \R^3)$. The action on the spaces $H^s(\T^\nu \times \T^3, \R^3)$ is given by 
$$
{\cal A} h = ({\cal A} h_1, {\cal A}h_2, {\cal A} h_3), \quad \forall h = (h_1, h_2, h_3) \in H^s(\T^\nu \times \T^3, \R^3)\,. 
$$

\begin{lemma}\label{lemma coniugio cal L (0)}
For any $S > s_0$, $\gamma \in (0, 1)$, $\tau > 0$ there exists $\delta = \delta (S, k_0, \tau, \nu) \in (0, 1)$, $\mu = \mu(k_0, \tau, \nu) > 0$ such that if \eqref{ansatz} holds and $N_0^{\tau_1}\e \gamma^{- 1} \leq \delta$, 
for any $(\omega, \zeta ) \in DC(\gamma, \tau)$ 
(see \eqref{def cantor set trasporto}) one has
\begin{equation}\label{cal L (1)}
{\cal L}^{(1)} := {\cal A}^{- 1} {\cal L}^{(0)} {\cal A} = \omega \cdot \partial_\vphi + \zeta \cdot \nabla + {\cal R}_0^{(1)} + {\cal R}_{- 1}^{(1)}
\end{equation}where ${\cal R}_0^{(1)} = {\rm Op}(R_0^{(1)}) \in {\cal OPS}^0_{S, 1}$, ${\cal R}_{- 1}^{(1)} = {\rm Op}(R_{- 1}^{(1)}) \in {\cal OPS}^{- 1}_{S, 0}$ and 
\begin{equation}\label{stime cal R (1)}
\begin{aligned}
& |{\cal R}_0^{(1)}|_{0, s, 1}^{k_0, \gamma}\,, \,|{\cal R}_{- 1}^{(1)}|_{- 1, s, 0}^{k_0, \gamma} \lesssim_{s} \e \| v \|_{s + \mu}^{k_0, \gamma}, \quad \forall s_0 \leq s \leq S\,. 
\end{aligned}
\end{equation}
Moreover ${\cal L}^{(1)}$, ${\cal R}_0^{(1)}$, ${\cal R}_{- 1}^{(1)}$ are real and reversible operators and the symbol of the zero-th order operator ${\cal R}_0^{(1)}$ satisfies the symmetry condition 
\begin{equation}\label{simmetria grado zero nella riduzione}
R_0^{(1)}(\vphi, x, \xi) = R_0^{(1)}(\vphi, x, - \xi)\,. 
\end{equation}
Let $s_1 \geq s_0$ and assume that $v_1 , v_2$ 
satisfy \eqref{ansatz} with $\mu_0 \geq s_1 + \mu$. 
Then, for any $\lambda = (\omega, \zeta) \in DC(\gamma, \tau)$, 
\begin{equation}\label{stime delta 12 cal R (1)}
 |\Delta_{12} {\cal R}_0^{(1)}|_{0, s_1, 1}\,, \,|\Delta_{12} {\cal R}_{- 1}^{(1)}|_{- 1, s_1, 0} \lesssim_{s_1} \e \| v_1 - v_2 \|_{s_1 + \mu}\,. 
\end{equation}
\end{lemma}

\begin{proof}
By Proposition \ref{proposizione trasporto}, using the formula \eqref{definizione cal R}, one has that 
$$
{\cal L}^{(1)} := {\cal A}^{- 1} {\cal L}^{(0)} {\cal A} = \omega \cdot \partial_\vphi + \zeta \cdot \nabla + \e {\cal A}^{- 1} {\cal R}_0 {\cal A} + \e {\cal A}^{- 1} {\cal R}_{- 1} {\cal A}\,. 
$$
We analyze separately the terms ${\cal A}^{- 1} {\cal R}_0 {\cal A}$ and ${\cal A}^{- 1} {\cal R}_{- 1} {\cal A}$. 

\noindent
{\sc Analysis of ${\cal A}^{- 1} {\cal R}_0 {\cal A}$.} 
By the formula \eqref{definizione cal R}, we recall that 
\begin{equation}\label{cal R0 nel lemma cambio variabile}
{\cal R}_0 := M_{\cal U}(\vphi, x)  - v \cdot \nabla {\cal U}
\end{equation}
where $\mU := \curl \Lambda^{-1}$ is defined 
in \eqref{equazione cal F vorticita} 
and we denote by $M_{\cal U}(\vphi, x)$ the multiplication operator 
by the $3 \times 3$ matrix $M_{\cal U}(\vphi, x)$. 
By the definition of $M_{\cal U}$ it is straightforward to verify that 
\begin{equation}\label{stima M cal U}
\| M_{\cal U} \|_s^{k_0, \gamma} \lesssim \| v \|_s^{k_0, \gamma}\,, 
\quad \forall s \geq s_0.
\end{equation}
Hence, by Lemma \ref{lemma coniugazione cambio variabile moltiplicazione}, 
by the estimates 
\eqref{stima M cal U}, 
\eqref{stima alpha trasporto} 
and using \eqref{ansatz} and $N_0^{\tau_0} \e \gamma^{- 1} \leq 1$, 
one gets that ${\cal A}^{- 1} {\rm Op}(M_{\cal U}) {\cal A} = {\rm Op}(\widetilde M_{\cal U}(\vphi, x))$ is a multiplication operator with 
\begin{equation}\label{stima widetilde M cal U}
|{\rm Op}(\widetilde M)|_{0,s, \beta}= \| \widetilde M_{\cal U} \|_s^{k_0, \gamma} \lesssim_{s, \beta}  \| v \|_{s + \mu}^{k_0, \gamma}, \quad \forall s_0 \leq s \leq S\,, \quad \forall \beta \in \N\,.  
\end{equation}
Since $M_{\cal U}\,,\,\breve \alpha = {\rm odd}(\vphi, x)$ then also $\widetilde M_{\cal U} = {\rm odd}(\vphi, x)$.  
We now study the conjugation ${\cal A}^{- 1} v \cdot \nabla {\cal U} {\cal A}$. 
Since ${\cal U} = {\rm curl}\, \Lambda^{- 1}$, we write
$$
{\cal A}^{- 1} v \cdot \nabla {\cal U} {\cal A} 
= ({\cal A}^{- 1} v \cdot \nabla {\cal A})( {\cal A}^{- 1} {\rm curl} {\cal A})
( {\cal A}^{- 1} \Lambda^{- 1} {\cal A} ).
$$
This formula, together with Lemmata \ref{lemma coniugazione curl nabla diffeo}, 
\ref{coniugazione inverso laplaciano cambio variabile}, 
the estimates \eqref{stima alpha trasporto}, 
the ansatz \eqref{ansatz} 
and the bound $N_0^{\tau_0} \e \gamma^{- 1} \leq 1$, 
imply that 
\begin{equation}\label{coniugio A A inv v nabla}
{\cal A}^{- 1} v \cdot \nabla {\cal U} {\cal A} 
= {\rm Op}(M_\nabla) {\rm Op}(M_{\rm curl}) \mP_{-2} 
+ {\rm Op}(M_\nabla) {\rm Op}(M_{\rm curl}) \mP_{-3} 
\end{equation}
with $\mP_{-2} = \Op(P_{-2})$,
\begin{equation}\label{parita nella dim M nabla curl}
\begin{aligned}
& M_\nabla(\vphi, x, \xi) = - M_\nabla(\vphi, x, - \xi), \quad 
M_{\rm curl}(\vphi, x, \xi) = - M_{\rm curl}(\vphi, x, - \xi)\,, \\
& P_{- 2}(\vphi,x, \xi) = P_{- 2}(\vphi, x, - \xi)
\end{aligned}
\end{equation}
and
\begin{equation}\label{stime nella dim M nabla curl}
\begin{aligned}
& M_\nabla, M_{\rm curl} \in {\cal S}^1_{S, 1}, \quad 
\mP_{-2} \in {\cal OPS}^{-2}_{S, 1}, \quad \mP_{-3} \in {\cal OPS}^{-3}_{S, 1}\,, \\
& |{\rm Op}(M_\nabla)|_{1, s, 1}^{k_0, \gamma} \lesssim_s \| v \|_{s + \mu}^{k_0, \gamma}\,, \\
& |{\rm Op}(M_{\rm curl})|_{1, s, 1}^{k_0, \gamma}\,,\, 
|\mP_{-2}|_{- 2, s, 1}^{k_0, \gamma}\,,\, 
|\mP_{-3}|_{-3, s, 0}^{k_0, \gamma} 
\lesssim_{s} 1 + \| v \|_{s + \mu}^{k_0, \gamma}
\end{aligned}
\end{equation}
for any $s_0 \leq s \leq S$. 
Hence, by \eqref{cal R0 nel lemma cambio variabile}, \eqref{stima widetilde M cal U}, \eqref{parita nella dim M nabla curl}, \eqref{stime nella dim M nabla curl}, by applying also \eqref{pseudo norm moltiplicazione}, \eqref{stima prodotto simboli}, Lemma \ref{lemma stime Ck parametri} 
and the ansatz \eqref{ansatz} with $\mu_0 > 0$ large enough, one obtains that 
\begin{equation}\label{espansione finale cal R0 1}
\begin{aligned}
& \mA^{-1} {\cal R}_0 \mA 
= {\rm Op}(R_{0,0}) + {\rm Op}(R_{0, - 1})\,, \quad R_{0, 0} \in {\cal S}^0_{S, 1}, \quad R_{0, - 1} \in {\cal S}^{- 1}_{S, 0}  \\
& R_{0,0}(\vphi, x, \xi) := \widetilde M_{\cal U}(\vphi, x) - M_\nabla (\vphi, x, \xi) M_{\rm curl} (\vphi, x, \xi)P_{- 2}(\vphi, x, \xi)\,, \\
& |{\rm Op}(R_{0,0})|_{0, s, 1}^{k_0, \gamma}\,,\, |{\rm Op}(R_{0, - 1})|_{- 1, s, 0}^{k_0, \gamma} \lesssim_{s}  \| v \|_{s + \mu}^{k_0, \gamma}, \quad \forall s_0 \leq s \leq S\,. 
\end{aligned}
\end{equation}
Moreover \eqref{parita nella dim M nabla curl} implies that 
\begin{equation}\label{parita R 00}
R_{0, 0}(\vphi, x, \xi) = R_{0,0}(\vphi, x, - \xi)
\end{equation} 
and, since ${\cal A}$ is reversibility preserving and ${\cal R}_0$ is reversible, 
then ${\rm Op}(R_{0,0}), {\rm Op}(R_{0, - 1})$ are reversible. 

\medskip

\noindent
{\sc Analysis of ${\cal A}^{- 1} {\cal R}_{- 1} {\cal A}$.} By \eqref{definizione cal R}, one writes 
$$
{\cal A}^{- 1} {\cal R}_{- 1} {\cal A} = ({\cal A}^{- 1} {\rm Op}(M_v) {\cal A}) ({\cal A}^{- 1} {\rm curl}{\cal A})({\cal A}^{- 1} \Lambda^{- 1}{\cal A})\,.
$$
Then using that $\| M_v \|_s^{k_0, \gamma} \lesssim \| v \|_{s + 1}^{k_0, \gamma}$, by applying Lemmata \ref{lemma coniugazione cambio variabile moltiplicazione}, \ref{lemma coniugazione curl nabla diffeo}, \ref{coniugazione inverso laplaciano cambio variabile}, the composition Lemma \ref{lemma stime Ck parametri}, the estimate \eqref{stima alpha trasporto}, 
the ansatz \eqref{ansatz} 
and $N_0^{\tau_0} \e \gamma^{- 1} \leq 1$, 
one gets that 
\begin{equation}\label{A inv R - 1 A}
\begin{aligned}
& {\cal A}^{- 1} {\cal R}_{- 1} {\cal A} \in {\cal OPS}^{- 1}_{S, 0} \,, \quad  |{\cal A}^{- 1} {\cal R}_{- 1} {\cal A}|_{- 1, s, 0}^{k_0, \gamma} \lesssim_{s}  \| v \|_{s + \mu}^{k_0, \gamma}, \quad \forall s_0 \leq s \leq S\,. 
\end{aligned}
\end{equation}
Moreover, since ${\cal R}_{- 1}$ is reversible and ${\cal A}$ is reversibility preserving, then ${\cal A}^{- 1} {\cal R}_{- 1} {\cal A}$ is reversible. 


The claimed properties \eqref{stime cal R (1)}, \eqref{simmetria grado zero nella riduzione}, then follow by setting $${\cal R}_0^{(1)} := \e {\rm Op}(R_{0,0}),\quad {\cal R}_{- 1}^{(1)} :=\e {\rm Op}(R_{0, - 1}) +  \e {\cal A}^{- 1} {\cal R}_{- 1} {\cal A}$$
using \eqref{espansione finale cal R0 1}, \eqref{parita R 00}, \eqref{A inv R - 1 A}. The estimates \eqref{stime delta 12 cal R (1)} can be proved by similar arguments. 
\end{proof}

\section{Elimination of the zero-th order term via a variable coefficients homological equation}

In this section, our aim is to construct a transformation of the form 
${\cal B}= {\rm Id} + {\cal M}$, with ${\cal M} = {\rm Op}(M(\vphi, x, \xi))$ of order $0$, 
in such a way that the transformed operator ${\cal B}^{- 1} {\cal L}^{(1)} {\cal B}$ is a one-smoothing perturbation of the operator $\omega \cdot \partial_\vphi + \zeta \cdot \nabla$. 
This means that we look for ${\cal M}$ that completely eliminates 
the zero-th order term ${\cal R}_0^{(1)}$ from the operator ${\cal L}^{(1)}$ in \eqref{cal L (1)}. The main technical issue here is that we deal with {\it matrix-valued} pseudo-differential operators, therefore the commutator $[{\cal R}_0^{(1)}, {\cal M}]$ 
does not gain derivatives 
(unlike commutators of {\it scalar} pseudo-differential operators do),
so that $[{\cal R}_0^{(1)}, {\cal M}]$ is still an operator of order $0$. 
This implies that, in order to remove the zero-th order term, we have to solve a {\it variable coefficients} homological equation, which is an equation 
in the unknown $M(\ph,x,\xi)$ of the form 
\begin{equation}\label{equazione omologica step 00}
\Big(\omega \cdot \partial_\vphi + \zeta \cdot \nabla + R_0^{(1)}(\vphi, x, \xi) \Big) M(\vphi, x, \xi) + R_0^{(1)}(\vphi, x, \xi) = 0\,. 
\end{equation}

\subsection{The homological equation at the zero-th order term}
\label{sez eq omologica grado 0 coeff variabili}
In order to simplify notations in this section, we set $V (\vphi, x, \xi) := R_0^{(1)}(\vphi, x, \xi)$. 
To deal with the equation \eqref{equazione omologica step 00},
the first step is to diagonalize the linear operator 
$$
{\cal P} := \omega \cdot \partial_\vphi + \zeta \cdot \nabla +  V(\vphi, x, \xi)
$$
acting on the space of matrix symbols ${\cal S}^0_{s + 1, 0}$ for any $s_0 \leq s \leq S$. 
The action of the operator ${\cal P}$ is given by 
\begin{equation}\label{azione op eq homo grado 0}
\begin{aligned}
& {\cal P } : {\cal S}^0_{s + 1, 0} \to {\cal S}^0_{s, 0}\,, \\
& A(\vphi, x, \xi) \mapsto  \omega \cdot \partial_\vphi A(\vphi, x, \xi) + \zeta \cdot \nabla A(\vphi, x, \xi) +  V(\vphi, x, \xi) A(\vphi, x, \xi)\,. 
\end{aligned}
\end{equation}
To develop the reducibility scheme, we use the pseudo differential norm 
$$
|V|_s^{k_0, \gamma} :=  |{\rm Op}(V)|_{0, s, 0}^{k_0, \gamma} 
= \max_{|\beta| \leq k_0}  \gamma^{|\beta|} \sup_{\lambda \in \R^{\nu + 3}} 
\sup_{\xi \in \R^3} \| \partial_\lambda^\beta V(\cdot, \xi; \lambda) \|_s\,,
$$
see Definition \eqref{definizione norma pseudo diff}. 
By Lemma \ref{lemma:smoothing} and by the estimate \eqref{p1-pr}, one easily gets the following properties of the norm $|\cdot |_s^{k_0, \gamma}$. 

\begin{lemma}\label{lemma proprieta norma su OPS 0}
$(i)$ Let $s \geq s_0$ and $V, B \in {\cal S}^0_{s, 0}$. Then 
$$
|VB|_s^{k_0, \gamma} \lesssim_{s} |V|_s^{k_0, \gamma} |B|_{s_0}^{k_0, \gamma} + |V|_{s_0}^{k_0, \gamma} |B|_s^{k_0, \gamma}\,. 
$$
$(ii)$ Let $N > 0$, $s \geq 0$. Then
\begin{align}
| \Pi_N V |_{s}^{k_0, \gamma} 
 \leq N^\alpha | V |_{s-\alpha}^{k_0, \gamma}\, , \quad 0 \leq \a \leq s, \quad  | \Pi_N^\bot V |_{s}^{k_0 , \gamma}  \leq N^{-\alpha} | V |_{s + \alpha}^{k_0 , \gamma}\, , \quad  \a \geq 0.
\label{p3-proi-ops}
\end{align}
\end{lemma}


\begin{proposition}
\label{coniugazione finale trasporto semilineare}
Let $S > s_0$, $\tau > 0$, $\gamma \in (0, 1)$. Then there exists $\mu = \mu(\tau, \nu, k_0) > 0$ and $\delta = \delta (S, k_0, \tau, \nu) \in (0, 1)$ small enough such that if \eqref{ansatz}, \eqref{condizione piccolezza rid trasporto} hold with $\mu_0 \geq s_0 + \mu $ and $\tau_1 = \tau_1(k_0, \tau, \nu) > 0$ large enough, then the following properties holds. 

\smallskip

\noindent
$(i)$ There exists a $k_0$ times differentiable matrix valued symbol $\Phi_\infty(\cdot ; \omega, \zeta) \in {\cal S}^0_{S, 0}$ such that, for any $(\omega, \zeta) \in DC(\gamma, \tau)$ (see \eqref{def cantor set trasporto}), one has   
$$
\Phi_\infty^{- 1} {\cal P} \Phi_\infty = \omega \cdot \partial_\vphi + \zeta \cdot \nabla,
$$
namely 
$\Phi_\infty^{- 1} {\cal P} (\Phi_\infty A) 
= \omega \cdot \partial_\vphi A + \zeta \cdot \nabla A$
for all symbols $A(\ph,x,\xi)$. 

Moreover $\Phi_\infty^{\pm 1} = {\rm even}(\vphi, x, \xi)$, 
$\Phi_\infty(\vphi, x, \cdot)^{\pm 1}$ $= {\rm even}(\xi)$ 
and the following estimates hold: 
$$
| \Phi_\infty^{\pm 1} - {\rm Id} |_s^{k_0, \gamma} \lesssim_{s} N_0^{\tau_0} \e \gamma^{- 1}  \| v \|_{s + \mu}^{k_0, \gamma}\,, \quad \forall s_0 \leq s \leq S\,. 
$$ 

\noindent
$(ii)$ There exists $M \in {\cal S}^{0}_{S, 0}$, 
$M = {\rm even}(\vphi, x, \xi)$, 
that solves the equation 
$$
(\omega \cdot \partial_\vphi + \zeta \cdot \nabla ) M(\vphi, x, \xi) 
+ V(\vphi, x, \xi) M(\vphi, x, \xi) + V(\vphi, x, \xi) = 0,
$$
with 
\[
|M|_s^{k_0, \gamma} \lesssim_{s} \e \gamma^{- 1}  \| v \|_{s + \mu}^{k_0,\gamma},
\quad \ s_0 \leq s \leq S. 
\]

\noindent
$(iii)$ Let $s_1 \geq s_0$ and let $u_1, u_2$ 
satisfy \eqref{ansatz} with $\mu_0 \geq s_1 + \mu$. 
Then, for any $(\omega, \zeta) \in DC(\gamma, \tau)$, one has 
\begin{equation}
\begin{aligned}
|\Delta_{12} \Phi_\infty^{\pm 1}|_{s_1} \lesssim_{s_1} N_0^{\tau_0} \e \gamma^{- 1} \| v_1 - v_2 \|_{s_1 + \mu}\,, \quad \, |\Delta_{12} M|_{s_1} \lesssim_{s_1} \e \gamma^{- 1} \| u_1 - u_2 \|_{s_1 + \mu}\,. 
\end{aligned}
\end{equation}
\end{proposition}

To prove Proposition \ref{coniugazione finale trasporto semilineare},
we fix the constants
\begin{equation}\label{costanti rid trasporto semi-lineare}
N_{- 1} := 1, \quad N_0\,,\,\tau > 0, \quad N_n := N_0^{(3/2)^n}, \quad \tau_0 := k_0+  \t(k_0+1) , \quad \mathtt a := 3 \tau_0 + 1 \,, \quad \mathtt b := \mathtt a + 1
\end{equation}
and prove the following Proposition first.

\begin{proposition} \label{prop riducibilita trasporto vettoriale semilin}
Let $S > s_0$, $\tau > 0$, $\gamma \in (0, 1)$. 
Then there exists $N_0 = N_0(S, k_0, \tau, \nu) > 0$, 
$\mu = \mu(k_0, \tau, \nu) > 0$ large enough, 
$\delta = \delta(S, k_0, \tau, \nu) \in  (0, 1)$ small enough, 
and $C_*(s) > 0$, $s_0 \leq s \leq S$, 
such that 
if \eqref{ansatz}, \eqref{condizione piccolezza rid trasporto} hold 
(for $\tau_1$ possibly larger) with $\mu_0 \geq s_0 + \mu$, 
the following statements hold for all $n \geq 0$. 

\smallskip

\noindent
${\bf (S1)_n}$ 
There exists a linear operator (acting on symbols)
\begin{equation}\label{op cal Ln}
{\cal P}_n := \omega \cdot \partial_\vphi + \zeta \cdot \nabla + V_n(\vphi, x, \xi) : {\cal S}^0_{s + 1, 0} \to {\cal S}^0_{s, 0}, \quad s_0 \leq s \leq S
\end{equation}
with 
\begin{equation}\label{simmetrie Vn iterazione}
V_n = {\rm odd}(\vphi, x, \xi)\,, \quad V_n(\vphi, x, \cdot) = {\rm even}(\xi)
\end{equation} 
and 
\begin{equation}\label{stima Vn}
\begin{aligned}
& | V_n |_{s}^{k_0, \gamma} \leq   C_*(s)N_{n - 1}^{- \mathtt a} \e  \| v \|_{s + \mu}^{k_0, \gamma}\,, \qquad  | V_n |_{s + \mathtt b}^{k_0, \gamma} \leq C_*(s) N_{n - 1} \e  \| v \|_{s + \mu}^{k_0, \gamma}
\end{aligned}
\end{equation}
for some constant $C_*(s) = C_*(s, k_0, \tau) > 0$. If $n \geq 1$, there exists $\Psi_{n - 1} \in {\cal S}^0_{s, 0}$, $\forall s \geq s_0$, $\Psi_{n - 1} = {\rm even}(\vphi, x, \xi)$, $\Psi_{n - 1}(\vphi, x, \cdot) ={\rm even}(\xi)$, such that 
\begin{equation}\label{stima Psin}
| \Psi_{n - 1}|_s^{k_0, \gamma} \lesssim N_{n - 1}^{\tau_0} \gamma^{- 1} | V_{n - 1}|_s^{k_0, \gamma}, \quad  \forall s \geq s_0
\end{equation}
and $\Phi_n := {\rm Id} + \Psi_n$ is invertible and, for any $\lambda = (\omega, \zeta) \in DC(\gamma, \tau)$, satisfies 
\begin{equation}\label{coniugazione a}
{\cal P}_n = \Phi_{n - 1}^{- 1} {\cal P}_{n - 1} \Phi_{n - 1}\,. 
\end{equation}

${\bf (S2)_n}$ Let $s_1 \geq s_0$ and assume that $u_1, u_2$ satisfy 
the ansatz \eqref{ansatz} with $\mu_0 \geq s_1 + \mu$. Then, for any $(\omega, \zeta) \in DC(\gamma, \tau)$, one has 
$$
\begin{aligned}
& |\Delta_{12} V_n|_{s_1} \lesssim_{s_1} N_{n - 1}^{- \mathtt a} \e  \| u_1 - u_2\|_{s_1 + \mu}\,, \quad |\Delta_{12} V_n|_{s_1 + \frak b} \lesssim_{s_1} N_{n - 1} \e  \| u_1 - u_2\|_{s_1 + \mu}  
\end{aligned}
$$
and for $n \geq 1$
$$
\begin{aligned}
& |\Delta_{12} \Psi_{n - 1}|_{s_1} \lesssim_{s_1} N_{n - 1}^{\tau_0} N_{n - 2}^{- \mathtt a} \e \gamma^{- 1} \| v_1 - v_2 \|_{s_1 + \mu}, \\
&   |\Delta_{12} \Psi_{n - 1}|_{s_1 + \mathtt b} \lesssim_{s_1} N_{n - 1}^{\tau_0} N_{n - 2} \e \gamma^{- 1} \| v_1 - v_2 \|_{s_1 + \mu}\,. 
\end{aligned}
$$
\end{proposition}

\begin{proof}
{\sc Proof of ${\bf (S1)_0}$-${\bf (S2)_0}$.} It follows by Lemma \ref{lemma coniugio cal L (0)}, by setting $V_0 = V := R_0^{(1)}$. 

\noindent
{\sc Proof of ${\bf (S1)_{n + 1}}$.} Arguing by induction, at the $n$-th step, we have the operator ${\cal P}_n$ in \eqref{op cal Ln}, whose remainder $V_n$ satisfies the estimate \eqref{stima Vn}. 
Let $\Phi_{n }(\vphi, x, \xi) = {\rm Id} + \Psi_{n }(\vphi, x, \xi)$, 
where the matrix valued symbol $\T^\nu \times \T^3 \times \R^3 \to {\rm Mat}_{3 \times 3}(\C)$, 
$(\vphi, x, \xi) \mapsto \Psi_{n }(\vphi, x, \xi)$ has to be determined. 
One has 
\[
\mP_n (\Phi_n A) 
= \Phi_n (\ompaph A + \zeta \cdot \grad A)
+ (\ompaph \Psi_n + \zeta \cdot \grad \Psi_n) A 
+ (\Pi_{N_n} V_n) A
+ (\Pi_{N_n}^\bot V_n) A
+ V_n \Psi_n A
\]
for all symbols $A = A(\ph,x,\xi)$, namely  
\begin{equation}\label{coniugazione 1}
\begin{aligned}
{\cal P}_n \Phi_{n} & =  \Phi_{n} (\omega \cdot \partial_\vphi + \zeta \cdot \nabla) 
+ ( \omega \cdot \partial_\vphi + \zeta \cdot \nabla ) \Psi_{n} 
+ \Pi_{N_n} V_n  + \Pi_{N_n}^\bot V_n+V_n\Psi_n
\end{aligned}
\end{equation}
where, to simplify notations, we do not write explicit dependence on $(\vphi, x, \xi)$.
The simmetry conditions  \eqref{simmetrie Vn iterazione} imply that $V_n(\cdot, \xi) = {\rm odd}(\vphi, x)$, therefore, for any $(\omega, \zeta) \in {\cal O}_\infty^\gamma$, we can solve the {\it homological equation }
\begin{equation}\label{equazione omologica 1}
(\omega \cdot \partial_\vphi + \zeta \cdot \nabla ) \Psi_{n}(\vphi, x, \xi) 
+ \Pi_{N_n} V_n(\vphi, x, \xi ) = 0
\end{equation}
by defining 
\begin{equation}\label{soluzione equazione omologica 1}
\Psi_{n }(\vphi, x, \xi) := 
- \big( \omega \cdot \partial_\vphi + \zeta \cdot \nabla \big)^{- 1}_{ext} 
\Pi_{N_n} V_n(\vphi, x, \xi), 
\end{equation}
where we recall \eqref{def ompaph-1 ext}. 
Since $V_n (\cdot, \xi)= {\rm odd}(\vphi, x)$ and $V_n(\vphi, x, \cdot) = {\rm even}(\xi)$, it is easy to verify that $\Psi_{n }(\cdot, \xi) = {\rm even}(\vphi, x)$, $\Psi_{n }(\vphi, x, \cdot) = {\rm even}(\xi)$. 
By Lemmata \ref{lemma:WD}, \ref{lemma proprieta norma su OPS 0} one immediately gets the estimates
\begin{equation} \label{stima Psi n + 1}
| \Psi_{n} |_s^{k_0, \gamma} \lesssim N_n^{\tau_0} \gamma^{- 1} | V_n |_s^{k_0, \gamma},  \quad  \forall s \geq s_0
\end{equation}
which is \eqref{stima Psin} at the step $n + 1$. 
The estimate \eqref{stima Psi n + 1}, 
together with 
\eqref{stima Vn} and 
\eqref{ansatz}, imply that for $s = s_0$
\begin{equation}\label{piccolezza Psi n + 1 s0}
| \Psi_{n } |_{s_0}^{k_0, \gamma} \lesssim N_n^{\tau_0} N_{n - 1}^{- \mathtt a} \e \gamma^{- 1}\,. 
\end{equation}
Hence, by taking $\e \gamma^{- 1}$ small enough and recalling \eqref{costanti rid trasporto semi-lineare}, the matrix $\Phi_{n}(\vphi, x, \xi) = {\rm Id} + \Psi_{n}(\vphi, x, \xi)$ 
is invertible by standard Neumann series, and 
\begin{equation}\label{stima Phi n + 1 inv}
| \Phi_{n}^{\pm 1} - {\rm Id} |_s^{k_0, \gamma} 
\lesssim_{s} | \Psi_{n}|_s^{k_0, \gamma}\,, \quad \forall s \geq s_0
\end{equation}
Hence \eqref{coniugazione 1}, \eqref{equazione omologica 1} imply that 
\begin{equation}\label{def cal L n + 1}
\begin{aligned}
{\cal P}_{n + 1} & := \Phi_{n}^{-1} {\cal P}_n \Phi_{n} 
= \omega \cdot \partial_\vphi + \zeta \cdot \nabla + V_{n + 1}(\vphi, x, \xi)\,, \\
V_{n + 1} & := \Pi_{N_n}^\bot V_n + \big( \Phi_{n }^{- 1} - {\rm Id} \big) 
\Pi_{N_n}^\bot V_n+ \Phi_{n}^{- 1} V_n \Psi_{n }\,.
\end{aligned}
\end{equation}
Since $V_n = {\rm odd}(\vphi, x, \xi)$, 
$V_n(\vphi,x, \cdot) = {\rm even}(\xi)$ 
and $\Psi_{n} = {\rm even}(\vphi,x, \xi)$, 
$\Psi_{n }(\vphi, x, \cdot)= {\rm even}(\xi)$, 
it follows that $V_{n+1}$ has the same parities as $V_n$, 
namely $V_{n + 1} = {\rm odd}(\vphi, x, \xi)$, 
$V_{n + 1}(\vphi, x, \cdot) = {\rm even}(\xi)$. 

It only remains to show the estimates \eqref{stima Vn} at the step $n + 1$. 
By Lemma \ref{lemma proprieta norma su OPS 0} 
and the estimates \eqref{stima Psi n + 1}-\eqref{stima Phi n + 1 inv}, 
for any $s_0 \leq s \leq S$ one has
\begin{equation}\label{prima stima induttiva V n + 1}
\begin{aligned}
& | V_{n + 1} |_s^{k_0, \gamma} \leq N_n^{- \mathtt b}  | V_n |_{s + \mathtt b}^{k_0, \gamma} + C(s) N_n^{\tau_0} \gamma^{- 1} | V_n|_{s_0}^{k_0, \gamma} | V_n|_s^{k_0, \gamma} \,, \\
& | V_{n + 1} |_{s + \mathtt b}^{k_0, \gamma} \leq |V_n|_{s + \mathtt b}^{k_0, \gamma}
\big(1 +  C(s) N_n^{\tau_0} \gamma^{- 1} | V_n|_{s_0}^{k_0, \gamma} \big). 
\end{aligned}
\end{equation}
Then by the induction estimates on $V_n$, by the choice of $\mathtt a, \mathtt b$ in \eqref{costanti rid trasporto semi-lineare}, and by taking $\e \gamma^{- 1} \leq \delta(S)$ small enough and $N_0 = N_0(S)> 0$ large enough, one gets the estimate \eqref{stima Vn} at the step $n + 1$. 
The estimates in ${\bf (S2)_{n + 1}}$ can be proved arguing similarly. 
\end{proof}

For $n \geq 1$, let 
\begin{equation}\label{def widetilde Phin}
\widetilde\Phi_n := \Phi_0  \Phi_1  \cdots \Phi_n\,. 
\end{equation}

\begin{lemma}\label{convergenza trasformaz trasporto semi-lin}
$(i)$ For any $s_0 \leq s \leq S$, the sequence $\widetilde \Phi_n^{\pm 1}$ 
converges in the norm $\| \cdot \|_s^{k_0, \gamma}$ 
to some limit $\Phi_\infty^{\pm 1}$ with $\Phi_\infty^{\pm 1} - {\rm Id}= {\rm even}(\vphi, x, \xi)$, $(\Phi_\infty^{\pm 1} - {\rm Id})(\vphi, x, \cdot) = {\rm even}(\xi)$. Moreover 
$$
| \Phi_\infty^{\pm 1} - {\rm Id} |_s^{k_0, \gamma} \lesssim_s  N_0^{\tau_0 }\e \gamma^{- 1} \| v \|_{s + \mu}^{k_0, \gamma}\,, \quad \forall s_0, \leq s \leq S\,. 
$$ 
$(ii)$ Let $s_1 \geq s_0$ and assume that $v_1, v_2$ satisfy \eqref{ansatz} with $\mu_0 \geq s_1 + \mu$. Then for any $\lambda \in DC(\gamma, \tau)$, $|\Delta_{12} \Phi_\infty^{\pm 1}|_{s_1} \lesssim_{s_1} N_0^{\tau_0} \e \gamma^{- 1} \| v_1 - v_2 \|_{s_1 + \mu}$. 
\end{lemma}

\begin{proof}
The proof of 
$(i)$ is standard and it can be done as the one of Corollary 4.1 in \cite{BBM-Airy}. 
The proof of 
$(ii)$ follows by similar arguments, 
using Proposition \ref{prop riducibilita trasporto vettoriale semilin}-${\bf (S2)_n}$. 
\end{proof}

\begin{proof}[{\sc Proof of Proposition \ref{coniugazione finale trasporto semilineare}}] 
Item $(i)$ follows in a straightforward way 
by Proposition \ref{prop riducibilita trasporto vettoriale semilin} 
and Lemma \ref{convergenza trasformaz trasporto semi-lin}. 
Now we prove item $(ii)$. 
By Proposition \ref{coniugazione finale trasporto semilineare}-$(i)$,
one is led to solve the equation
\begin{equation}\label{pappa bambini}
(\omega \cdot \partial_\vphi + \zeta \cdot \nabla) 
(\Phi_\infty(\vphi, x, \xi)^{- 1} M(\vphi, x, \xi)) 
= - \Phi_\infty(\vphi, x, \xi)^{- 1} V(\vphi, x, \xi)\,.
\end{equation}
Since $V = {\rm odd}(\vphi, x, \xi)$, $\Phi_\infty^{\pm 1} = {\rm even}(\vphi, x, \xi)$, $V(\vphi, x, \cdot ), \Phi_\infty(\vphi,x, \cdot)^{\pm 1} = {\rm even}(\xi)$, one deduces that 
\[
\Phi_\infty(\cdot, \xi)^{\pm 1} = \even(\ph,x), 
\quad \ 
V(\cdot, \xi) = \odd(\ph,x). 
\]
Hence $(\Phi_\infty^{- 1} V)(\cdot, \xi) = {\rm odd}(\vphi, x)$, 
and therefore its average in $(\ph,x)$ is zero. 
As a consequence, the equation \eqref{pappa bambini} is solved 
for any $(\omega, \zeta) \in DC(\gamma, \tau)$ 
by setting 
$$
M(\vphi, x, \xi) := - \Phi_\infty(\vphi, x, \xi) 
(\omega \cdot \partial_\vphi + \zeta \cdot \nabla)^{- 1}_{ext}
[\Phi_\infty^{- 1} V](\vphi, x, \xi)\,. 
$$
One easily verifies that 
$M = {\rm even}(\vphi, x, \xi)$. 
The claimed estimates on $M$, $\Phi_\infty^{\pm 1}$ 
follow by Lemmata \ref{lemma:WD}, \ref{lemma proprieta norma su OPS 0}-$(i)$, 
\ref{convergenza trasformaz trasporto semi-lin} 
and by the estimate \eqref{stime cal R (1)} (recall that $V := R_0^{(1)}$).
\end{proof}

\subsection{Elimination of the zero-th order term}

\begin{lemma}\label{riduzione ordine 0}
Let $S > s_0$, $\gamma \in (0, 1)$, $\tau > 0$. Then there exists $\mu = \mu (S, k_0, \tau, \nu) > 0$, $\tau_1 = \tau_1(k_0, \tau, \nu) > 0$ and $\delta = \delta(S, k_0, \tau, \nu) \in (0, 1)$ such that if \eqref{ansatz}, \eqref{condizione piccolezza rid trasporto} are fullfilled with $\mu_0 \geq s_0 + \mu$, then the following properties hold. 
There exist 
a real and reversibility preserving map 
${\cal M} = {\rm Op}(M) \in {\cal OPS}^{0}_{S, 0}$ 
such that the map ${\cal B} := {\rm Id} + {\cal M}$ is invertible, with
\begin{equation}\label{stima cal B ord zero}
|{\cal M}|_{0, s, 0}^{k_0, \gamma}\,,\, |{\cal B}^{- 1} - {\rm Id}|_{0, s, 0}^{k_0, \gamma}\,,\, |{\cal M}^*|_{0, s, 0}^{k_0, \gamma} \lesssim_{s} \e \gamma^{- 1} \| v \|_{s +. \mu}^{k_0, \gamma}, \quad \forall s_0 \leq s \leq S,
\end{equation}
and a real and reversible operator ${\cal L}^{(2)}$ 
of the form 
\begin{equation}\label{def cal L2}
{\cal L}^{(2)} :=  \omega \cdot \partial_\vphi + \zeta \cdot \nabla + {\cal R}^{(2)},
\end{equation}
defined for all 
$(\omega, \zeta) \in \R^{\nu + 3}$, 
with ${\cal R}^{(2)} \in {\cal OPS}^{- 1}_{S, 0}$ and 
\begin{equation}\label{stima cal R (2)}
|{\cal R}^{(2)}|_{- 1, s, 0}^{k_0, \gamma} 
\lesssim_{s} \e  \| v \|_{s + \mu}^{k_0, \gamma} 
\quad \ \forall s_0 \leq s \leq S, 
\end{equation}
such that, for all $(\omega, \zeta) \in DC(\gamma, \tau)$, 
the operator ${\cal L}^{(1)}$ defined in \eqref{cal L (1)}
is conjugated to ${\cal L}^{(2)}$, namely 
\[
{\cal L}^{(2)} = {\cal B}^{- 1} {\cal L}^{(1)} {\cal B}.
\] 


\noindent
Let $s_1 \geq s_0$ 
and assume that $v_1, v_2$ satisfy \eqref{ansatz} with $\mu_0 \geq s_1 + \mu$. 
Then, for all $\lambda = (\omega, \zeta) \in DC(\gamma, \tau)$, 
\begin{equation}\label{stime delta 12 cal B}
\begin{aligned}
& |\Delta_{12} {\cal B}^{\pm 1}|_{0, s_1, 0}\,,\, |\Delta_{12} {\cal B}^{*}|_{0, s_1, 0} \lesssim_{s_1} \e \gamma^{- 1} \| v_1 - v_2 \|_{s_1 + \mu}\,, \\
& |\Delta_{12} {\cal R}^{(2)}|_{- 1, s_1, 0} \lesssim_{s_1} \e \| v_1 - v_2 \|_{s_1 + \mu}\,. 
\end{aligned}
\end{equation}
\end{lemma}

\begin{proof} 
To eliminate the term of order $0$ 
from the operator ${\cal L}^{(1)}$ in \eqref{cal L (1)},
we look for 
a transformation 
\begin{equation}\label{def cal M1}
{\cal B} = {\rm Id} + {\cal M}, \quad 
{\cal M} = {\rm Op}\big( M(\vphi, x, \xi)\big) \in {\cal OPS}^0_{S, 0}, 
\end{equation}
where the symbol $M(\vphi, x, \xi)$ has to be determined. 
By Lemma \ref{lemma stime Ck parametri}, one has 
\begin{equation}\label{pupilla 0}
{\cal R}_0^{(1)}{\cal M} 
= {\rm Op} \big( R_0^{(1)}(\vphi, x, \xi) M(\vphi, x, \xi) \big) 
+ {\cal R}_{R_0^{(1)} M},
\end{equation}
where the remainder ${\cal R}_{R_0^{(1)} M}$ is an operator of order $- 1$. 
Hence 
\begin{equation}\label{coniugazione cal L (1) a}
\begin{aligned}
{\cal L}^{(1)} {\cal B}
& = {\cal B} (\omega \cdot \partial_\vphi + \zeta \cdot \nabla ) 
+ {\rm Op} \big( (\omega \cdot \partial_\vphi + \zeta \cdot \nabla) M + R_0^{(1)} M + R_0^{(1)}  \big) 
+ {\cal R}_{R_0^{(1)} M} + {\cal R}_{- 1}^{(1)}{\cal B} \,. 
\end{aligned}
\end{equation}
Note that 
${\cal R}_{- 1}^{(1)}{\cal B}$ is also an operator of order $- 1$. 
By Proposition \ref{coniugazione finale trasporto semilineare} 
applied with $V := R_0^{(1)}$, 
there exists a symbol $M = M(\ph,x,\xi) \in \mS^0_{S,0}$, $\even(\ph,x,\xi)$, 
that solves the equation \eqref{equazione omologica step 00} 
and satisfies the bounds \eqref{stima cal B ord zero}. 
%
%
By \eqref{coniugazione cal L (1) a}, \eqref{equazione omologica step 00} 
one deduces formula \eqref{def cal L2} with 
\[
{\cal R}^{(2)} := {\cal B}^{- 1} {\cal R}_{R_0^{(1)} M} 
+ {\cal B}^{- 1} {\cal R}_{- 1}^{(1)}{\cal B}.
\] 
The bounds for ${\cal M}^*$ and ${\cal B}^{- 1} - {\rm Id}$ follow by Lemmata \ref{lemma aggiunto}, \ref{Lemma Neumann}. 
The estimate of $|{\cal R}^{(2)}|_{- 1, s, 0}^{k_0, \gamma}$ then follows 
by Lemma \ref{lemma stime Ck parametri} 
and by 
estimates \eqref{stime cal R (1)}, \eqref{stima cal B ord zero}, 
\eqref{ansatz}, 
\eqref{condizione piccolezza rid trasporto}. 
Since $M = {\rm even}(\vphi, x, \xi)$, the maps ${\cal B}^{\pm 1}$ are reversibility preserving maps, hence since ${\cal L}^{(1)}$ is reversible then also ${\cal L}^{(2)}$ is reversible. The estimates \eqref{stime delta 12 cal B} can be proved arguing similarly. 
\end{proof}

\section{Reduction to constant coefficients of the lower order terms}
In this section we reduce the operator ${\cal L}^{(2)}$ in \eqref{def cal L2} to constant coefficients up to an arbitarily regularizing remainder. More precisely, we prove the following Proposition. 

\begin{proposition} \label{proposizione regolarizzazione ordini bassi}
Let $S > s_0$, $M \in \N$, $\gamma \in (0, 1)$, $\tau > 0$. Then there exists $\mu = \mu( M, k_0, \tau, \nu) > 0$, $\delta = \delta(S, M, k_0, \tau, \nu) \in (0, 1)$ such that if \eqref{ansatz}, \eqref{condizione piccolezza rid trasporto} are fullfilled with $\mu_0 \geq s_0 + \mu$, then the following holds. There exists a real and reversibility preserving, invertible map ${\cal V}$ satisfying 
\begin{equation}\label{stima Phi M reg ordini bassi}
|{\cal V}^{\pm 1}- {\rm Id} |_{0, s, 0}^{k_0, \gamma}\,,\, |{\cal V}^* - {\rm Id} |_{0, s, 0}^{k_0, \gamma}  \lesssim_{s, M}\e \gamma^{- 1}   \| v \|_{s + \mu}^{k_0, \gamma}, \quad \forall s_0 \leq s \leq S
\end{equation}
and a linear operator ${\cal L}^{(3)}$ defined by 
\begin{equation}\label{def cal L (3)}
{\cal L}^{(3)} :=  \omega \cdot \partial_\vphi + \zeta \cdot \nabla + {\cal Q} + {\cal R}^{(3)}
\end{equation}
where ${\cal Q} = {\rm Op}(Q(\xi)) \in {\cal OPS}^{- 1}_{S, 0}$ is an operator whose $3 \times 3$ matrix symbol is independent of $(\vphi, x)$, 
${\cal R}^{(3)}$ belongs to ${\cal OPS}^{- M}_{S, 0}$, 
and 
\begin{equation}\label{stima cal Z cal R (3)}
|{\cal Q}|_{- 1, s, 0}^{k_0, \gamma} \, , \, 
|{\cal R}^{(3)}|_{- M, s, 0}^{k_0, \gamma} 
\lesssim_{s, M} \e  \| v \|_{s + \mu}^{k_0, \gamma} 
\quad \ \forall s_0 \leq s \leq S\,.
\end{equation}
The operators ${\cal L}^{(3)}$, ${\cal Q}$, ${\cal R}^{(3)}$ are real and reversible operators defined for all the values of the parameters $(\omega, \zeta) \in \R^{\nu + 3}$. For any $(\omega, \zeta) \in DC(\gamma, \tau)$, one has ${\cal V}^{- 1} {\cal L}^{(2)} {\cal V} = {\cal L}^{(3)}$. 

Let $s_1 \geq s_0$ and let 
$v_1, v_2$ satisfy \eqref{ansatz} with $\mu_0 \geq s_1 + \mu$. 
Then for any $(\omega, \zeta) \in DC(\gamma, \tau)$ one has 
\begin{equation}\label{stime delta 12 cal L (3)}
\begin{aligned}
& |\Delta_{12} {\cal V}^{\pm 1}|_{0, s_1, 0}\,,\, |\Delta_{12} {\cal V}^*|_{0, s_1, 0} \lesssim_{s_1, M} \e \gamma^{- 1} \| v_1 - v_2 \|_{s_1 + \mu}\,, \\
& |\Delta_{12} {\cal Q}|_{- 1, s_1, 0}\,,\, |\Delta_{12} {\cal R}^{(3)}|_{- M, s_1, 0} \lesssim_{s_1, M} \e \| v_1 - v_2\|_{s_1 + \mu}\,. 
\end{aligned}
\end{equation}
\end{proposition}

Proposition \ref{proposizione regolarizzazione ordini bassi} 
follows by the following iterative lemma. 

\begin{lemma}\label{lemma iterativo ordini bassi}
Let $S > s_0$, $M \in \N$, $\gamma \in (0, 1)$, $\tau > 0$. Then there exists $\mu_M = \mu( M, k_0, \tau, \nu) > 0$, $\delta = \delta(S, M, k_0, \tau, \nu) \in (0, 1)$ and $\mu_1 < \mu_2 <\ldots <\mu_M$ such that if \eqref{ansatz}, \eqref{condizione piccolezza rid trasporto} are fullfilled with $\mu_0 \geq s_0 + \mu_M$, then the following holds.  For any $n = 0, \ldots, M - 1$, there exists a linear operator ${\cal L}^{(2)}_n$ of the form 
\begin{equation}\label{def cal L (2) n}
{\cal L}^{(2)}_n := \omega \cdot \partial_\vphi + \zeta \cdot \nabla + {\cal Z}_n + {\cal R}_n^{(2)}
\end{equation}
where ${\cal Z}_n = {\rm Op}(Z_n(\xi))\in {\cal OPS}^{- 1}_{S, 0}$ 
and ${\cal R}_n^{(2)} = {\rm Op}(R_n^{(2)}(\vphi, x, \xi)) \in {\cal OPS}^{- (n + 1) }_{S, 0}$ 
are defined for all the values of the parameters $(\omega, \zeta) \in \R^{\nu + 3}$ and satisfy
\begin{equation}\label{stima induttiva cal Zn cal Rn (2)}
\begin{aligned}
& |{\cal Z}_n|_{- 1, s, 0}^{k_0, \gamma} \lesssim_{s, n} \e  \| v \|_{s_0 + \mu_n}^{k_0, \gamma} \quad \ \forall s_0 \leq s \leq S\,, \\
& |{\cal R}_n^{(2)}|_{- (n + 1), s, 0}^{k_0, \gamma} \lesssim_{s, n} \e \| v \|_{s + \mu_n}^{k_0, \gamma} \quad \ \forall s_0 \leq s \leq S\,. 
\end{aligned}
\end{equation} 
The operators ${\cal L}_n^{(2)}, {\cal Z}_n, {\cal R}_n^{(2)}$ are real and reversible. 
%
For any $n = 1, \ldots, M$ there exists a reversibility preserving, invertible map ${\cal T}_n$ defined for any $(\omega, \zeta) \in \R^{\nu + 3}$, satisfying the estimate
\begin{equation}\label{stima cal Tn reg ordini bassi}
|{\cal T}_n^{\pm 1} - {\rm Id}|_{- n, s, 0}^{k_0, \gamma} \lesssim_{s, n} \e \gamma^{- 1} \| v \|_{s + \mu_n}^{k_0, \gamma} \quad \ \forall s_0 \leq s \leq S
\end{equation}
and 
\begin{equation}\label{coniugazione lemma iterativo ordini bassi}
{\cal L}^{(2)}_n = {\cal T}_n^{- 1}{\cal L}_{n - 1}^{(2)} {\cal T}_n 
\quad \ \forall (\omega, \zeta) \in DC(\gamma, \tau)\,. 
\end{equation}
Let $s_1 \geq s_0$ and 
assume that $v_1, v_2$ satisfy \eqref{ansatz} with $\mu_0 \geq s_1 + \mu_M$. 
Then, for any $\lambda = (\omega, \zeta) \in DC(\gamma, \tau)$,
\begin{equation}\label{Delta 12 Tn Zn Rn (2)}
\begin{aligned}
& |\Delta_{12} {\cal T}_n^{\pm 1} |_{- n, s_1, 0} \lesssim_{s_1, n} \e \gamma^{- 1} \| v_1 - v_2 \|_{s_1 + \mu_n}\,, \\
& |\Delta_{12} {\cal Z}_n|_{- 1, s_1, 0}\,,\,  |\Delta_{12} {\cal R}_n^{(2)}|_{- n, s_1, 0} \lesssim_{s_1, n} \e \| v_1 - v_2 \|_{s_1 + \mu_n}\,. 
\end{aligned}
\end{equation}
\end{lemma}

\begin{proof}
We prove the lemma arguing by induction. For $n = 0$ the claimed statement follows by Lemma \ref{riduzione ordine 0}, by defining ${\cal L}^{(2)}_0 := {\cal L}^{(2)}$, ${\cal Z}_0 = 0$, ${\cal R}^{(2)}_0 := {\cal R}^{(2)}$.

We assume that the claimed statement holds for some $n \in \{0, \ldots, M - 1\}$ 
and we prove it at the step $n + 1$. Let us consider a transformation ${\cal T}_{n + 1} = {\rm Id} + {\cal M}_{n + 1}$ where ${\cal M}_{n + 1} = {\rm Op}(M_{n + 1}(\vphi, x, \xi))$ is an operator of order $- (n + 1)$ which has to be determined. One computes 
\begin{equation}\label{prima formula cal L n (2)}
\begin{aligned}
{\cal L}_n^{(2)} {\cal T}_{n + 1} 
& = {\cal T}_{n + 1} \big( \omega \cdot \partial_\vphi + \zeta \cdot \nabla \big) + {\cal Z}_n 
+ {\rm Op}\Big( (\omega \cdot \partial_\vphi + \zeta \cdot \nabla) M_{n + 1} + R_n^{(2)} \Big) 
+ {\cal Z}_n{\cal M}_{n + 1} + {\cal R}_n^{(2)}{\cal M}_{n + 1}\,. 
\end{aligned}
\end{equation}
We define the symbol $M_{n + 1}(\vphi, x, \xi)$ as 
\begin{equation}\label{def M n + 1 lot}
M_{n + 1} := \big( \omega \cdot \partial_\vphi + \zeta \cdot \nabla \big)_{ext}^{- 1} 
\big[ \langle R_n^{(2)} \rangle_{\vphi, x} -   R_n^{(2)} \big],
\end{equation} 
which is defined for any $(\omega, \zeta) \in \R^{\nu + 3}$ (recall \eqref{def ompaph-1 ext}). 
Clearly ${\cal M}_{n + 1}$ is an operator of the same order as $\mR_n^{(2)}$, namely 
${\cal M}_{n + 1}\in {\cal OPS}^{- (n + 1)}_{S, 0}$, 
and, by Lemma \ref{lemma:WD} and the induction estimate \eqref{stima induttiva cal Zn cal Rn (2)}, 
one verifies that
\begin{equation}\label{stima induttiva M n + 1 nel reg}
|{\cal M}_{n + 1}|_{- (n + 1), s, 0}^{k_0, \gamma} \lesssim_{s, n} \e \gamma^{- 1} \| v \|_{s + \mu_n + \tau_0}^{k_0, \gamma}\,, \quad \forall s_0 \leq s \leq S\,. 
\end{equation}
Furthermore, for any $(\omega, \zeta) \in DC(\gamma, \tau)$, 
$M_{n + 1}(\ph,x,\xi)$ solves the homological equation
\begin{equation}\label{M n + 1 R n + 1 (2)}
(\omega \cdot \partial_\vphi + \zeta \cdot \nabla) M_{n + 1} + R_n^{(2)} = \langle R_n^{(2)} \rangle_{\vphi, x}\,. 
\end{equation}
Using the ansatz \eqref{ansatz} with $\mu_0 \geq s_0 + \mu_n + \tau_0$, one gets $|{\cal M}_{n + 1}|_{- (n + 1), s_0, 0}^{k_0, \gamma} \lesssim \e \gamma^{- 1}$, hence by taking $\e \gamma^{- 1}$ small enough, one can apply Lemma \ref{Lemma Neumann}, obtaining that ${\cal T}_{n + 1} = {\rm Id} + {\cal M}_{n + 1}$ is invertible with inverse ${\cal T}_{n + 1}^{- 1}$ satisfying the estimate 
\begin{equation}\label{stima cal T n + 1 - 1 nel proof}
|{\cal T}_{n + 1}^{- 1} - {\rm Id}|_{- (n + 1), s, 0}^{k_0, \gamma} 
\lesssim_{s,  n} \e \gamma^{- 1} \| v \|_{s + \mu_n + \tau_0}^{k_0, \gamma}, 
\quad \forall s_0 \leq s \leq S\,. 
\end{equation}
Hence the estimate \eqref{stima cal Tn reg ordini bassi} at the step $n + 1$ holds by taking $\mu_{n + 1} \geq \mu_n + \tau_0$. 
By \eqref{prima formula cal L n (2)}, \eqref{M n + 1 R n + 1 (2)} we then get, 
for any $(\omega, \zeta) \in DC(\gamma, \tau)$, 
the conjugation \eqref{coniugazione lemma iterativo ordini bassi} at the step $n + 1$ 
where ${\cal L}^{(2)}_{n + 1}$ has the form \eqref{def cal L (2) n}, with 
\begin{equation}\label{cal Z cal R2 n + 1}
\begin{aligned}
{\cal Z}_{n + 1} & := {\cal Z}_n + {\rm Op} \big(\langle R_n^{(2)} \rangle_{\vphi, x} \big)\,, \\
{\cal R}_{n + 1}^{(2)} & := ({\cal T}_{n + 1}^{- 1} - {\rm Id}){\cal Z}_{n + 1} 
+ {\cal T}_{n + 1}^{- 1} 
\big( {\cal Z}_n{\cal M}_{n + 1} + {\cal R}_n^{(2)}{\cal M}_{n + 1} \big)  \,. 
\end{aligned}
\end{equation}
Since ${\cal M}_{n + 1}, {\cal T}_{n + 1}^{\pm 1}, {\cal R}_n^{(2)}$ are defined 
for all 
$(\omega, \zeta) \in \R^{\nu + 3}$, 
then also ${\cal R}_{n + 1}^{(2)}$ is defined for all $(\omega, \zeta) \in \R^{\nu + 3}$. 
Since ${\cal Z}_n$ is of order $- 1$ and ${\cal R}_n^{(2)}, {\cal M}_{n + 1}$ 
are of order $- (n + 1)$, one gets that 
${\cal Z}_n{\cal M}_{n + 1} + {\cal R}_n^{(2)}{\cal M}_{n + 1}$ 
is of order $- (n + 2)$.  
The estimate \eqref{stima induttiva cal Zn cal Rn (2)}  at the step $n + 1$ 
for the operator ${\cal Z}_{n + 1}$ follows 
by the induction estimate on ${\cal Z}_n, {\cal R}_n^{(2)}$
and by the property \eqref{proprieta simbolo mediato}.  
The estimate for the operator ${\cal R}_{n + 1}^{(2)}$ follows by \eqref{stima induttiva M n + 1 nel reg}, \eqref{stima cal T n + 1 - 1 nel proof}, the estimates for ${\cal Z}_{n + 1}$ and ${\cal Z}_n$, ${\cal R}_n^{(2)}$ and using Lemma \ref{lemma stime Ck parametri}-$(i)$. Furthermore, since ${\cal R}_n^{(2)}$ and ${\cal Z}_n$ are reversible, then by \eqref{M n + 1 R n + 1 (2)} one verifies that ${\cal M}_{n + 1} = {\rm Op}(M_{n + 1})$ is reversibility preserving 
and therefore, by \eqref{cal Z cal R2 n + 1}, 
${\cal Z}_{n + 1}$ and ${\cal R}_{n + 1}^{(2)}$ are reversible operators. 
The estimates \eqref{Delta 12 Tn Zn Rn (2)} can be proved by similar arguments. 
\end{proof}

\begin{proof}[{\sc Proof of Proposition \ref{proposizione regolarizzazione ordini bassi}}] 
We define 
$$
{\cal L}^{(3)} := {\cal L}^{(2)}_{M - 1}, \quad 
{\cal Q} := {\cal Z}_{M - 1}, \quad 
{\cal R}^{(3)} := {\cal R}_{M - 1}^{(2)}, \quad 
{\cal V} := {\cal T}_1 \circ \ldots \circ {\cal T}_{M - 1} \,.
$$
Then, in order to deduce the claimed properties of ${\cal Z}$ and ${\cal R}^{(3)}$, 
it suffices to apply Lemma \ref{lemma iterativo ordini bassi} for $n = M - 1$. The estimate \eqref{stima Phi M reg ordini bassi} then follows by the estimate \eqref{stima cal Tn reg ordini bassi} on ${\cal T}_n$, $n = 1, \ldots, M - 1$, using the composition Lemma \ref{lemma stime Ck parametri}-$(i)$ and Lemma \ref{lemma aggiunto} (to estimate the adjoint operator ${\cal V}^*$). The estimates \eqref{stime delta 12 cal L (3)} follow similarly. 
\end{proof}

\section{Conjugation of the operator ${\cal L}$}
In this section we go back to the operator ${\cal L}$ in \eqref{operatore linearizzato} which contains the projector $\Pi_0^\bot$, see \eqref{definizione proiettore media spazio tempo}. We define
\begin{equation}\label{def cal U}
{\cal E} := {\cal A}{\cal B} {\cal V} \quad \text{and} \quad {\cal E}_\bot := \Pi_0^\bot {\cal E} \Pi_0^\bot
\end{equation}
We also recall that $H^s_0$ is defined in \eqref{def sobolev}.

\begin{lemma}\label{lemma invertibilita cal U bot}
Let $S > s_0$, $M \in \N$, $\tau > 0$, $\gamma\in (0, 1)$. Then there exists $\delta = \delta(S, M, k_0, \tau) \in (0, 1)$ small enough and $\mu_M := \mu(M, k_0, \tau, \nu ) > 0$ such that if \eqref{ansatz}, \eqref{condizione piccolezza rid trasporto} are fullfilled with $\mu_0 \geq s_0 + \mu_M$, the following properties hold. 

\smallskip

\noindent
$(i)$ For any $s_0 \leq s \leq S$, the map ${\cal E}_\bot : H^s_0 \to H^s_0$ is invertible and satisfies the tame estimates
\begin{equation}\label{stima cal U bot}
\| {\cal E}_\bot^{\pm 1} h \|_s^{k_0, \gamma} \lesssim_{s, M} \| h \|_s^{k_0, \gamma} + \| v \|_{s + \mu_M}^{k_0, \gamma} \| h \|_{s_0}^{k_0, \gamma}\,, \quad \forall s_0 \leq s \leq S\,. 
\end{equation}
Furthermore ${\cal E}_\bot$ and ${\cal E}_\bot^{- 1}$ are real and reversibility preserving.

\smallskip

\noindent
$(ii)$ For any $(\omega, \zeta) \in DC(\gamma, \tau)$, the map ${\cal E}_\bot$ conjugates the operator ${\cal L}$ in \eqref{operatore linearizzato} to an operator ${\cal L}_0 : H^{s + 1}_0 \to H^s_0$ for any $s_0 \leq s \leq S$, defined by 
\begin{equation}\label{coniugazione-pre-riducibilita}
{\cal L}_0 := {\cal E}_\bot^{- 1} {\cal L} {\cal E}_\bot 
= \omega \cdot \partial_\vphi + \zeta \cdot \grad + {\cal Q}_0 + {\cal R}_0
\end{equation}
where ${\cal Q}_0$ is a $3 \times 3$ time independent block-diagonal operator 
(see Definition \ref{def block-diagonal op})
defined for any $(\omega, \zeta) \in \R^{\nu + 3}$, 
represented by the 
matrix 
${\cal Q}_0 = {\rm diag}_{j \in \Z^3 \setminus \{ 0 \}} ({\cal Q}_0)_j^j$,  
satisfying 
\begin{equation}\label{stima-blocch-i33-pre-rid}
\sup_{j \in \Z^3 \setminus \{ 0 \}} 
\| ({\cal Q}_0)_j^j \|_{\HS}^{k_0, \gamma} |j| \lesssim_{ M}  \e  \| v \|_{s_0 + \mu_M}^{k_0, \gamma}
\end{equation}
and ${\cal R}_0$ is a linear operator defined 
for any $(\omega, \zeta) \in \R^{\nu + 3}$ satisfying the estimate 
\begin{equation}\label{stima cal R0 pre rid}
|{\cal R}_0 \langle D \rangle^M|_{s}^{k_0, \gamma} \lesssim_{s, M} \e \| v\|_{s + \mu_M}^{k_0, \gamma},  \quad \forall s_0 \leq s \leq S
\end{equation}
where the decay norm $| \cdot |_s^{k_0, \gamma}$ is defined in 
Definition \ref{block decay norm}
and $\langle D \rangle^M$ is the Fourier multiplier of symbol 
$\langle \xi \rangle^M = (1 + |\xi|^2)^{M/2}$. 
Furthermore, ${\cal L}_0, {\cal Q}_0, {\cal R}_0$ are real and reversible. 

\smallskip

\noindent
$(iii)$ Let $s_1 \geq s_0$ 
and assume that $v_1, v_2$ satisfy \eqref{ansatz} with $\mu_0 \geq s_1 + \mu_M$. Then for any $(\omega, \zeta) \in DC(\gamma, \tau)$, one has 
\begin{equation}\label{delta 12 E bot R0 Q0}
\begin{aligned}
& \| \Delta_{12} {\cal E}_\bot^{\pm 1} h \|_{s_1} \lesssim_{s_1, M} \e \gamma^{- 1} \| v_1 - v_2 \|_{s_1 + \mu_M} \| h \|_{s_1 + 1}\,, \\
& |\Delta_{12} {\cal R}_0 \langle D \rangle^M|_{s_1}^{k_0, \gamma} \lesssim_{s_1, M} \e  \| v_1 - v_2 \|_{s_1 + \mu_M}\,,\quad  \sup_{j \in \Z^3 \setminus \{ 0 \}} 
\| (\Delta_{12} {\cal Q}_0)_j^j \|_{\HS} |j| \lesssim_{ M}  \e  \| v_1 - v_2 \|_{s_1 + \mu_M}\,. \\
\end{aligned}
\end{equation}
\end{lemma}

\begin{proof}
To simplify notations, in this proof we write $\| \cdot \|_s$ 
instead of $\| \cdot \|_s^{k_0, \gamma}$. 
Also, given any operator $A$ mapping $H^s(\T^{\nu+3},\R^3)$ into itself, 
we denote by $A[1] = A[1](\ph,x)$
the $3 \times 3$ matrix of entries 
\begin{equation} \label{def A[1]}
(A[1])_{j,k} := A_{j,k}[1] 
= \langle A(e_k) , e_j \rangle_{\R^3}, 
\quad \ j,k = 1,2,3,
\end{equation}
where $e_j$ is the $j$-th vector of the canonical basis of $\R^3$; 
we denote by $\Pi_0 A[1] = (\Pi_0 A[1])(\ph)$ 
the $3 \times 3$ matrix of entries $\Pi_0 (A_{j,k}[1])(\ph)$. 
Note that, if $A = \Op(a(\ph,x,\xi))$, then $A[1] = a(\ph,x,0)$.

\medskip

\noindent
{\sc Proof of $(i)$.} By Proposition \ref{proposizione trasporto}, Lemmata \ref{lemma coniugio cal L (0)}, \ref{riduzione ordine 0}, Proposition \ref{proposizione regolarizzazione ordini bassi}, using also Lemma \ref{azione pseudo}, the operator ${\cal E}$ is invertible, ${\cal E}, {\cal E}^{- 1}$ are real and reversibility preserving and,
for $s_0 \leq s \leq S$,   
\begin{equation}\label{invB}
\begin{aligned}
& \| {\cal E}^{\pm 1} h \|_s 
\lesssim_{s,  M} \| h \|_s + \| v \|_{s + \mu_M} \| h \|_{s_0} \,,
\\
& \| ({\cal E}^{\pm 1} - {\rm Id}) h  \|_s \, , \, 
\| ({\cal E}^*  - {\rm Id}) h  \|_s  
\lesssim_{s,  M} N_0^{\tau_0} \e \gamma^{- 1} 
\big( \| v \|_{s_0 + \mu_M} \| h \|_{s + 1} + \| v \|_{s + \mu_M} \| h \|_{s_0 + 1} \big) ,
\end{aligned}
\end{equation}
for some $\mu_M > 0$. 
For any $u \in H^s(\T^{\nu + 3}, \R^3)$, 
we split $u = \Pi_0^\bot u + \Pi_0 u$, 
where $\Pi_0^\bot u \in H^s_0 = H^s_0(\T^{\nu + 3}, \R^3)$ 
and $\Pi_0 u \in H^s_\vphi := H^s(\T^\nu, \R^3)$ 
(recall that $(\Pi_0 u)(\ph)$ is the space average of $u$; 
it is a function of $\ph$, independent of $x$). 
Hence the operator $ {\cal E} $ is decomposed accordingly into 
$$
{\cal E} = 
\begin{pmatrix} 
\Pi_0 {\cal E} \Pi_0  &  \Pi_0 {\cal E} \Pi_0^\bot \vspace{4pt}  \\
 \Pi_0^\bot {\cal E} \Pi_0 &   \Pi_0^\bot {\cal E} \Pi_0^\bot
\end{pmatrix}. 
$$
{\sc Step 1:} \emph{Invertibility of $\Pi_0 {\cal E} \Pi_0$.} 
We write $\Pi_0 {\cal E} \Pi_0 = \Pi_0 + {\cal R}_\mE$
with ${\cal R}_\mE := \Pi_0 ({\cal E} - {\rm Id}) \Pi_0$. 
For any $f \in H^s_\ph(\T^\nu, \R^3)$, 
namely $f = f(\ph)$ independent of $x$, one has 
\begin{equation} \label{quasi scalar}
(\mE f)(\ph,x) = (\mE[1])(\ph,x) \cdot f(\ph)
\end{equation}
where ``$\, \cdot \,$'' is the matrix product of the matrix $\mE[1]$ by the vector $f$. 
In fact, $\mathcal{B} \mathcal{V}$ is a pseudo-differential operator 
of matrix symbol, say, $P(\ph,x,\xi)$; 
therefore $\mathcal{B} \mathcal{V} f = P(\ph,x,0) f(\ph) 
= (\mathcal{B} \mathcal{V})[1](\ph,x) f(\ph)$;
and then 
$\mA \{ (\mathcal{B} \mathcal{V})[1](\ph,x) f(\ph) \} 
= (\mathcal{B} \mathcal{V})[1](\ph,x + \a(\ph,x)) f(\ph)$
because $\mA$ is a change of the $x$ variable.  
Hence 
$$
{\cal R}_\mE \Pi_0 h 
= \Pi_0 ({\cal E} - {\rm Id}) \Pi_0 h 
= (\Pi_0 ({\cal E} - {\rm Id})[1])(\ph) \cdot 
(\Pi_0 h)(\ph)
$$
because $(\Pi_0 h)(\ph)$ is independent of $x$.
Therefore, by the estimate \eqref{invB}, 
using the ansatz \eqref{ansatz} and the product estimate \eqref{p1-pr}, 
one obtains that 
\begin{equation}\label{stima-Pi0-U-Pi 0-a}
\| {\cal R}_\mE \Pi_0 h \|_s \lesssim_{s,  M} N_0^{\tau_0}\e \gamma^{- 1}\Big( \| v \|_{s_0 + \mu_M}\| \Pi_0 h \|_s + \| v \|_{s + \mu_M} \| \Pi_0 h \|_{s_0} \Big), \quad \forall s_0 \leq s \leq S\,. 
\end{equation}
By Neumann series, using the smallness condition \eqref{condizione piccolezza rid trasporto},
one then gets that 
$\Pi_0 {\cal E} \Pi_0 : H^s_\vphi \to H^s_\vphi$ is invertible and 
\begin{equation}\label{stima-Pi0-U-Pi 0-b}
\| (\Pi_0 {\cal E} \Pi_0)^{- 1} \Pi_0 h \|_s \lesssim_{s,  M}  \| \Pi_0 h \|_s + \| v \|_{s + \mu_M} \| \Pi_0 h \|_{s_0}, \quad \forall s_0 \leq s \leq S\,. 
\end{equation}
%
%

\medskip

\noindent
{\sc Step 2:} {\it The operator 
\begin{equation}\label{def:M}
M := \Pi_0^\bot {\cal E} \Pi_0^\bot 
- (\Pi_0^\bot {\cal E} \Pi_0) (\Pi_0 \mE \Pi_0)^{-1} (\Pi_0 {\cal E} \Pi_0^\bot )
\end{equation}
is invertible, and}
\begin{equation}\label{stM}
\| M^{-1} z \|_s 
\lesssim_{s , M} \| z \|_s + \| v \|_{s + \mu_M} \| z \|_{s_0}, \quad \forall s_0 \leq s \leq S\,,
\quad \forall z \in H^s_0 \, .  
\end{equation}
Let us prove the invertibility of $M$. 
Since $\mE : H^s \to H^s$ is invertible, 
given any $z \in H^s_0$ there exists a unique $h \in H^s$ 
such that $\mE h = z$. 
Since $(\Pi_0 \mE \Pi_0) : H^s_\ph \to H^s_\ph$ is invertible, 
one has ${\cal E} h = z $ if and only if
$$
\begin{pmatrix} 
\Pi_0 {\cal E} \Pi_0  &  \Pi_0 {\cal E} \Pi_0^\bot \vspace{4pt}  \\
 \Pi_0^\bot {\cal E} \Pi_0 &   \Pi_0^\bot {\cal E} \Pi_0^\bot
\end{pmatrix}
\begin{pmatrix}
\Pi_0 h   \\
\Pi_0^\bot h   
\end{pmatrix}
= 
\begin{pmatrix}
0 \\ z 
\end{pmatrix}
\quad \Longleftrightarrow \quad M \Pi_0^\bot h = z \, . 
$$
Using \eqref{invB}, 
\[ 
\| M^{-1} z \|_s = \| \Pi_0^\bot h \|_s 
\leq \| h \|_s = \| {\cal E}^{-1} z \|_s 
\lesssim_{s,  M} \| z \|_s + \| v \|_{s + \mu_M} \| z \|_{s_0} \,,
\] 
and \eqref{stM} is proved. 

\medskip

\noindent
{\sc Step 3:} {\it Estimates of $ \Pi_0^\bot {\cal E} \Pi_0 $, $ \Pi_0 {\cal E} \Pi_0^\bot $}. 
By \eqref{quasi scalar}, since $(\Pi_0 h)(\ph)$ is independent of $x$, 
one has 
\[
\Pi_0^\bot {\cal E} \Pi_0 h 
= \Pi_0^\bot ({\cal E} - {\rm Id}) \Pi_0 h 
= \big( \Pi_0^\bot ({\cal E} - {\rm Id})[1] \big)(\ph,x) \cdot (\Pi_0 h)(\ph). 
\]
Also, by an explicit calculation, 
\[
\Pi_0 {\cal E} \Pi_0^\bot h 
= \Pi_0 ({\cal E} - {\rm Id}) \Pi_0^\bot h 
= \Pi_0 (G \cdot \Pi_0^\bot h)
\]
where ``$\, \cdot \,$'' is the ``row by column'' product, 
and $G = G(\ph,x)$ is the matrix with entries
\begin{equation} \label{def G}
G_{j,k}(\ph,x) 
= (\mE_{j,k}  - \mathrm{Id}_{j,k})^*[1] (\ph,x)
= (\mE^* - \mathrm{Id})_{k,j}[1] (\ph,x),
\quad \ j,k = 1,2,3.
\end{equation}
The estimate\eqref{invB}, the ansatz \eqref{ansatz}, 
and the product estimate \eqref{p1-pr} imply that 
\begin{equation}\label{st1}
\| \Pi_0^\bot {\cal E} \Pi_0 h \|_s\,,\, \| \Pi_0 {\cal E} \Pi_0^\bot h \|_s \lesssim_{s,  M} N_0^{\tau_0} \e \gamma^{- 1}\Big( \| v \|_{s_0 + \mu_M}\| h \|_s + \| v \|_{s + \mu_M} \| h \|_{s_0} \Big), \quad \forall s_0 \leq s \leq S,
\end{equation}
for some $\mu_M > 0$ large enough. 

%

\medskip

\noindent
{\sc Step 4:} {\it Invertibility of $ \Pi_0^\bot {\cal E} \Pi_0^\bot $}. 
By \eqref{def:M} we have 
$$
\begin{aligned}
& \Pi_0^\bot {\cal E} \Pi_0^\bot = M +  (\Pi_0^\bot {\cal E} \Pi_0) (\Pi_0 {\cal E} \Pi_0)^{-1}(\Pi_0 {\cal E} \Pi_0^\bot) = M \big( \Pi_0^\bot +  {\cal S} \big)\,, \\
& {\cal S} := M^{-1} (\Pi_0^\bot {\cal E} \Pi_0) (\Pi_0 {\cal E} \Pi_0)^{-1}(\Pi_0 {\cal E} \Pi_0^\bot)  \, .
\end{aligned}
$$ 
By \eqref{stima-Pi0-U-Pi 0-b}, \eqref{stM}, \eqref{st1}, we deduce that 
\begin{align*}
\| {\cal S} h \|_s 
& \lesssim_{s, M} N_0^{\tau_0} \e \gamma^{- 1}\Big( \| v \|_{s_0 + \mu_M}\| h \|_s  + \| v \|_{s + \mu_M} \| h \|_{s_0} \Big) \,, \quad \forall s_0 \leq s \leq S\,. 
\end{align*}
Therefore, by Neumann series, 
$ \Pi_0^\bot {\cal E} \Pi_0^\bot : H^s_0 \to H^s_0$, $s_0 \leq s \leq S$ is invertible and it satisfies \eqref{stima cal U bot}.

\medskip

\noindent
{\sc Proof of $(ii)$.} By Proposition \ref{proposizione trasporto}, Lemmata \ref{lemma coniugio cal L (0)}, \ref{riduzione ordine 0}, Proposition \ref{proposizione regolarizzazione ordini bassi} and recalling \eqref{operatore linearizzato}, \eqref{def cal L (0)}, using that $\Pi_0 {\cal E} \Pi_0^\bot = \Pi_0 ({\cal E} - {\rm Id}) \Pi_0^\bot$ and $\Pi_0^\bot {\cal E} \Pi_0 = \Pi_0^\bot ({\cal E} - {\rm Id}) \Pi_0$, one obtains 
\begin{equation}\label{gatto 0}
\begin{aligned}
{\cal L} {\cal E}_\bot & = \Pi_0^\bot {\cal L}^{(0)} {\cal E} \Pi_0^\bot - \Pi_0^\bot {\cal L}^{(0)} \Pi_0 ({\cal E} - {\rm Id}) \Pi_0^\bot  \\
& = \Pi_0^\bot {\cal E} {\cal L}^{(3)}\Pi_0^\bot - \Pi_0^\bot {\cal L}^{(0)} \Pi_0 ({\cal E} - {\rm Id}) \Pi_0^\bot \\
& = {\cal E}_\bot (\Pi_0^\bot {\cal L}^{(3)} \Pi_0^\bot) +  \Pi_0^\bot ({\cal E} - {\rm Id}) \Pi_0 {\cal L}^{(3)}\Pi_0^\bot - \Pi_0^\bot {\cal L}^{(0)} \Pi_0 ({\cal E} - {\rm Id}) \Pi_0^\bot\,.
\end{aligned}
\end{equation}
Hence, recalling \eqref{def cal L (3)}, one obtains the conjugation \eqref{coniugazione-pre-riducibilita} with 
\begin{equation}\label{def cal N0 cal R0 nella dim}
\begin{aligned}
{\cal Q}_0 & := \Pi_0^\bot {\cal Q} \Pi_0^\bot \,, 
\quad \ 
{\cal R}_0 := \Pi_0^\bot {\cal R}^{(3)} \Pi_0^\bot 
+ A_1 - A_2, 
\\
A_1 & := {\cal E}_\bot^{- 1}\Pi_0^\bot ({\cal E} - {\rm Id}) \Pi_0 {\cal L}^{(3)}\Pi_0^\bot, 
\quad \  
A_2 := {\cal E}_\bot^{- 1}\Pi_0^\bot {\cal L}^{(0)} \Pi_0 ({\cal E} - {\rm Id}) \Pi_0^\bot\,. 
\end{aligned}
\end{equation}
Since ${\cal Q} = {\rm Op}(Q)$ has a symbol independent of $(\vphi, x)$, 
the operator ${\cal Q}_0  = \Pi_0^\bot {\cal Q} \Pi_0^\bot 
= {\rm diag}_{j \in \Z^3 \setminus \{ 0 \}} ({\cal Q}_0)_j^j $ 
is a $3 \times 3$ block diagonal operator 
(see Definition \ref{def block-diagonal op})
with $({\cal Q}_0)_j^j = ({\cal Q})_j^j \in \Mat_{3 \times 3}(\C)$ 
for any $j \in \Z^3 \setminus \{ 0 \}$. 
The estimate \eqref{stima-blocch-i33-pre-rid} holds by \eqref{stima cal Z cal R (3)}, by Remark \ref{rem:matrix norm = pseudo-diff norm} and Lemma \ref{proprieta standard norma decay}-$(v)$.

It only remains to prove the estimate \eqref{stima cal R0 pre rid}. First, we estimate the term $\Pi_0^\bot {\cal R}^{(3)} \Pi_0^\bot$ in \eqref{def cal N0 cal R0 nella dim}. 
For any $s_0 \leq s \leq S$, 
using Lemma \ref{norma pseudo norma dec} 
and \eqref{stima cal Z cal R (3)},
one has 
\begin{equation}\label{gnocchi 0}
\begin{aligned}
|\Pi_0^\bot {\cal R}^{(3)} \Pi_0^\bot \langle D \rangle^M|_s^{k_0, \gamma} 
\lesssim |{\cal R}^{(3)}  \langle D \rangle^M|_s^{k_0, \gamma} 
& 
\lesssim | {\cal R}^{(3)} \langle D \rangle^M|_{0, s, 0}^{k_0, \gamma}  
\lesssim | {\cal R}^{(3)} |_{- M, s, 0}^{k_0, \gamma} 
\lesssim_{s, M}
\e  \| v \|_{s + \mu_M}^{k_0, \gamma}\,. 
\end{aligned}
\end{equation}
Since $\Pi_0 \ompaph \Pi_0^\bot$, 
$\Pi_0 \zeta \cdot \grad$ 
and $\Pi_0 \mQ \Pi_0^\bot$
are all zero,
recalling \eqref{def cal L (3)} one has 
$\Pi_0 \mL^{(3)} \Pi_0^\bot 
= \Pi_0 \mR^{(3)} \Pi_0^\bot$. 
By \eqref{quasi scalar}, we get
\[ 
A_1 \langle D \rangle^M h 
= {\cal E}_\bot^{- 1} \Pi_0^\bot ({\cal E} - {\rm Id}) \Pi_0 {\cal L}^{(3)} \Pi_0^\bot 
\langle D \rangle^M h 
= \big( {\cal E}_\bot^{- 1}\Pi_0^\bot {\cal E} [1] \big)(\ph,x) 
\cdot \big( \Pi_0 \mR^{(3)} \Pi_0^\bot \langle D \rangle^M h \big)(\ph).
\] 
Note that $\mR^{(3)} \langle D \rangle^M$ is estimated in \eqref{gnocchi 0}.
Also, 
\[
\Pi_0 (\mE - \Id) \Pi_0^\bot \langle D \rangle^M h 
= \Pi_0 ( G \cdot \Pi_0^\bot \langle D \rangle^M h )
= \Pi_0 \big[ (\langle D \rangle^M G) \cdot \Pi_0^\bot h \big]
\]
with $G$ defined in \eqref{def G}. 
Recalling \eqref{operatore linearizzato}, \eqref{definizione cal R}, 
since $\mU \Pi_0 = 0$, one has 
$\Pi_0^\bot \mL^{(0)} \Pi_0 = \e M_\mU(\ph,x) \Pi_0$. 
Hence
\[ 
A_2 \langle D \rangle^M h 
= {\cal E}_\bot^{- 1}\Pi_0^\bot {\cal L}^{(0)} \Pi_0 ({\cal E} - {\rm Id}) \Pi_0^\bot 
\langle D \rangle^M h 
= \e ( {\cal E}_\bot^{- 1} M_\mU) \cdot 
\Pi_0 \big[ (\langle D \rangle^M G) \cdot \Pi_0^\bot h \big].
\] 
Applying Lemma \ref{norma pseudo norma dec} 
(to bound $| \ |_s$ with $| \ |_{0,s,0}$),
Lemmata \ref{azione pseudo} and \ref{lemma stime Ck parametri}-$(i)$
(to estimate composition of operators and their action on functions), 
bounds \eqref{pseudo norm moltiplicazione}
(for the $| \ |_{0,s,0}$ norm 
of any multiplicative matrix),
\eqref{gnocchi 0} 
(for $|\mR^{(3)} \langle D \rangle^M|_{0,s,0}$),
\eqref{stima M cal U} 
(for $\| M_\mU \|_s$),
\eqref{invB} (for $\mE, \mE^*$), 
\eqref{stima cal U bot} (for $\mE_\bot^{-1}$), 
together with the assumptions 
\eqref{ansatz}, \eqref{condizione piccolezza rid trasporto}
and the trivial estimates $| \Pi_0 |_{0,s,0}$, $| \Pi_0^\bot |_{0,s,0} \leq 1$,
we obtain that 
\[
| A_1 \langle D \rangle^M |_s^{k_0, \gamma}\, , 
| A_2 \langle D \rangle^M |_s^{k_0, \gamma}
\lesssim_{s, M} \e \| v \|_{s + \mu_M}^{k_0, \gamma}\,.
\]
This proves \eqref{stima cal R0 pre rid}. 
The proof of the item $(iii)$ follows by similar arguments. 
\end{proof}

\section{Reducibility and inversion} \label{sezione riducibilita a blocchi}
In this section we perform a reducibility scheme for the operator ${\cal L}_0$ given in Lemma \ref{lemma invertibilita cal U bot}. 
We recall, as observed in Section \ref{sezione matrici norme}, 
that any linear operator can be described both by a matrix representation 
and by a pseudo-differential representation with symbol, see  
\eqref{matriciale 1}-\eqref{equivalent symbol con phi}. 
In this section we consider transformations of $H^s_0$,
the space of functions in $H^s$ having zero average in the space variable $x$
(see 
\eqref{def sobolev}).
 
We introduce some further notations. 
Given a matrix $A \in \Mat_{3 \times 3}(\C)$, we define the linear operators 
\begin{equation}\label{definizione mult left right}
\begin{aligned}
& M_L(A) : \Mat_{3 \times 3}(\C) \to \Mat_{3 \times 3}(\C), \quad B \mapsto AB, \\
& M_R(A) : \Mat_{3 \times 3}(\C) \to \Mat_{3 \times 3}(\C), \quad B \mapsto BA.
\end{aligned}
\end{equation}
Since $\| AB \|_{\HS}$, $\| BA \|_{\HS}$ $\leq \| A \|_{\HS} \| B \|_{\HS}$, 
one has 
\begin{equation}\label{MLAMRA elementry}
\| M_L(A)\|_\op \,, \, \| M_R(A)\|_\op \leq \| A \|_{\HS}
\end{equation}
%
where $\| \ \|_\op $ denotes the standard operator norm 
(
$\| \ \|_{\HS}$ is defined in 
\eqref{def Euclidean norm}).
Given $\tau, k_0, N_0 > 0$, we fix the constants 
\begin{equation}\label{definizione alpha beta}
\begin{aligned}
& \tau_0 := (k_0 + 1) \tau + k_0 ,  \quad 
M := [2 \tau_0] + 1,  \quad 
\mathtt a := \max \Big\{6 \tau_0 + 1, \frac32(\tau + 2 \tau^2) + 1\Big\} \,, \quad 
\mathtt b := \mathtt a + 1 \,,\\
& \mu(\mathtt b) := \mu_M + \mathtt b\,, \quad   
N_{- 1} := 1, \quad 
N_n := N_0^{\chi^n}, \quad 
n \geq 0, \quad 
\chi := 3/2, 
\end{aligned}
\end{equation}
where $[2\tau_0]$ is the integer part of $2\tau_0$
and $\mu_M$ is the constant appearing 
in Lemma \ref{lemma invertibilita cal U bot}. 
Note that Lemma \ref{lemma invertibilita cal U bot} holds for any $M \in \N$; 
in \eqref{definizione alpha beta} we fix its value so that $M \geq 2\tau_0$. 

\begin{remark}[\bf Choice of the constants]
The conditions $\mathtt a \geq 6 \tau_0 +1$, $\mathtt b \geq \mathtt a + 1$ in \eqref{definizione alpha beta} are used to show the convergence in the iterative estimates \eqref{stime R n + 1 R n}. We also require that $\mathtt a \geq \frac32(\tau + 2 \tau^2) + 1$, for the measure estimates of Section \ref{sezione stime di misura}, see Lemma \ref{inclusione risonanti v n n - 1} 
\end{remark}

\begin{proposition}[\bf Reducibility]\label{prop riducibilita}
Let $\gamma \in (0, 1)$, $\tau > 0$, $S > s_0$ 
and assume \eqref{ansatz} with $\mu_0 \geq s_0 +  \mu(\mathtt b)$. 
Then there exists $N_0 = N_0(S, k_0, \tau) > 0$, $\tau_2 = \tau_2(k_0, \tau, \nu) > 0$ large enough and $\delta = \delta (S, k_0, \tau) \in (0, 1)$ small enough such that if 
\begin{equation}\label{KAM smallness condition}
N_0^{\tau_2}\e \gamma^{-1} 
\leq \delta, 
\end{equation} 
then for any integer $n \geq 0$ the following statements hold.

\smallskip
\noindent 
${\bf (S1)_n}$ There exists a real and reversible operator 
$$
{\cal L}_n := \omega \cdot \partial_\vphi + {\cal N}_n +   {\cal R}_n : H^{s + 1}_0 \to H^s_0, \quad \forall s_0 \leq s \leq S
$$
defined for any $\lm = (\omega, \zeta) \in \R^{\nu + 3}$ and $k_0$ times differentiable 
in $\lm = (\omega, \zeta) \in \R^{\nu+3}$ with the following properties. 
The operator ${\cal N}_n$ is a $3 \times 3$ block-diagonal operator 
(see Definition \ref{def block-diagonal op})
defined by 
\begin{equation} \label{cal Nn rid}
\begin{aligned}
& {\cal N}_n := \zeta \cdot \nabla + {\cal Q}_n \,, \quad  
{\cal Q}_n = \diag_{j \in \Z^3 \setminus \{ 0 \}} ({\cal Q}_n)_j^j\,, \quad \ 
(\mQ_n)_j^j \in \Mat_{3 \times 3}(\C), 
\\
& \| ({\cal Q}_n)_j^j \|_{\HS}^{k_0, \gamma} \lesssim \e |j|^{- 1} \| v \|_{s_0 + \mu(\mathtt b)}^{k_0, \gamma}\,, \quad \|({\cal Q}_n)_j^j - ({\cal Q}_0)_j^j \|^{k_0, \gamma}_{\HS} \lesssim
\e  |j|^{- M} \| v \|_{s_0 + \mu(\mathtt b)}^{k_0, \gamma}, \quad \forall j \in \Z^3 \setminus \{ 0 \}\,.
\end{aligned}
\end{equation}
If $n \geq 1$, for all $j \in \Z^3 \setminus \{ 0 \}$ one has 
\begin{equation}\label{cal Nn - N n - 1}
\| ({\cal Q}_n)_j^j - ({\cal Q}_{n - 1})_j^j \|^{k_0, \gamma}_{\HS}
\lesssim_{k_0} \e N_{n - 2}^{- \mathtt a} |j|^{- M} \| v \|_{s_0 + \mu(\mathtt b)}^{k_0, \gamma} \,. 
\end{equation}
If $n \geq 0$, the operator ${\cal R}_n$ satisfies, 
for any $s_0 \leq s \leq S$, the estimates 
\begin{equation}\label{stime cal Rn rid}
\begin{aligned}
& |{\cal R}_n \langle D \rangle^M |_s^{k_0, \gamma} 
 \leq C_*(s) N_{n - 1}^{- \mathtt a}\e 
 \| v \|_{s + \mu(\mathtt b)}^{k_0, \gamma} ,  \quad 
|{\cal R}_n \langle D \rangle^M |_{s + \mathtt b}^{k_0, \gamma}  \leq C_*(s) N_{n - 1}\e  \| v \|_{s + \mu(\mathtt b)}^{k_0, \gamma} 
\end{aligned}
\end{equation}
for some constant $C_* (s) = C_*(s, k_0, \tau) > 0$ large enough. 
For any $\ell \in \Z^\nu$, $j, j' \in \Z^3 \setminus \{ 0 \}$, 
define the linear operator $L_n (\ell, j, j') \equiv L_n(\ell, j, j'; \lambda, v(\lambda)) \equiv L_n(\ell, j, j'; \lambda)$ 
$:$ $\Mat_{3 \times 3}(\C) \to \Mat_{3 \times 3}(\C)$ by
\begin{equation}\label{operatore blocchi}
L_n (\ell, j, j')  := \ii \omega \cdot \ell \, {\rm Id} + M_L\big( ({\cal N}_n)_j^j\big) - M_R\big(({\cal N}_n)_{j'}^{j'} \big).
\end{equation}
Note that, by \eqref{cal Nn rid}, 
it is $(\mN_n)_j^j = \ii (\zeta \cdot j) I + (\mQ_n)_j^j 
\in \Mat_{3 \times 3}(\C)$, where $I \in \Mat_{3 \times 3}$ is the identity matrix.


If $n=0$, we define $\Omega_0^\gamma := DC(\gamma, \tau)$; 
if $n \geq 1$, we define
\begin{equation}\label{insiemi di cantor rid}
\begin{aligned}
\Omega_n^\gamma := \Big\{ \lambda = (\omega, \zeta) \in \Omega_{n - 1}^\gamma 
& : L_{n - 1}(\ell, j, j') 
\text{ is invertible and } 
\| L_{n - 1}(\ell, j, j')^{- 1}\|_\op 
\leq \frac{\langle \ell \rangle^\tau |j|^\tau |j'|^\tau}{\gamma} \\
& \ \ 
\forall \ell \in \Z^\nu, \  
j,j' \in \Z^3 \setminus \{ 0 \}, \ \ 
(\ell, j, j') \neq (0, j, j), \ \ 
|\ell|, |j-j'| \leq N_{n-1} \Big\}.  
\end{aligned}
\end{equation}

If $n \geq 1$, there exists an invertible, real and reversibility preserving map 
$\Phi_{n -1} = {\rm Id} + \Psi_{n - 1}$, $k_0$ times differentiable in 
$\lambda = (\omega, \zeta) \in \R^{\nu + 3}$, such that, 
for any $\lambda \in \Omega_n^\gamma$, 
\begin{equation}\label{coniugazione rid}
{\cal L}_n = \Phi_{n - 1}^{- 1} {\cal L}_{n - 1} \Phi_{n - 1}\,.
\end{equation}
Moreover, for any $s_0 \leq s \leq S$, the map $\Psi_{n - 1}$ satisfies the estimates
\begin{equation}\label{stime Psi n rid}
\begin{aligned}
|\Psi_{n - 1}|_s^{k_0, \gamma}, |\langle D \rangle^{- M}\Psi_{n - 1} \langle D \rangle^{ M} |_s^{k_0, \gamma} 
& \lesssim_{s} \e \gamma^{- 1} N_{n - 1}^{\tau_0} N_{n - 2}^{- \mathtt a} \| v\|_{s + \mu(\mathtt b)}^{k_0, \gamma}\,, \\
|\langle D \rangle^{ - M}\Psi_{n - 1} \langle D \rangle^{M}|_{s + \mathtt b}^{k_0, \gamma} & \lesssim_{s} \e \gamma^{- 1} N_{n - 1}^{\tau_0}  N_{n - 2}  \| v \|_{s + \mu(\mathtt b)}^{k_0, \gamma}\,. 
\end{aligned}
\end{equation}
${\bf (S2)_n}$ 
Assume that $v_1, v_2$ satisfy \eqref{ansatz} with $\mu_0 \geq s_0 + \mu(\mathtt b)$. 
Then for any $\lambda = (\omega, \zeta) 
\in \Omega_n^{\gamma_1}(v_1) \cap \Omega_n^{\gamma_2}(v_2)$ 
with $\gamma_1, \gamma_2 \in [\gamma/2\,,\, 2 \gamma]$, 
the following estimates hold: 
\begin{equation}\label{stime delta 12 iterazione rid}
\begin{aligned}
& |\Delta_{12} {\cal R}_n \langle D \rangle^M |_{s_0} \lesssim_{} \e N_{n - 1}^{- \mathtt a} \| v_1 - v_2 \|_{s_0 + \mu(\mathtt b)}, \quad |\Delta_{12} {\cal R}_n \langle D \rangle^M|_{s_0 + \mathtt b} \lesssim_{} \e N_{n - 1} \| v_1 - v_2 \|_{s_0 + \mu(\mathtt b)}\,.
\end{aligned}
\end{equation}
Furthermore, if $n \geq 1$, for any $j \in \Z^3 \setminus \{ 0 \}$ one has
\begin{equation}\label{stime Delta 12 cal Nn}
\begin{aligned}
& \| \Delta_{12} ({\cal Q}_n - {\cal Q}_{n - 1})_j^j \|_{\HS}^{k_0, \gamma} \lesssim \e |j|^{- M} N_{n - 2}^{- \mathtt a} \| v_1 - v_2\|_{s_0 + \mu(\mathtt b)}\,, \\
&  \| \Delta_{12} ({\cal Q}_n)_j^j \|_{\HS}^{k_0, \gamma} \lesssim \e |j|^{- 1}  \| v_1 - v_2\|_{s_0 + \mu(\mathtt b)}\,.
\end{aligned}
\end{equation}
 ${\bf (S3)_n}$ Let $v_1, v_2$ as in ${\bf (S2)_n}$ and $0 < \rho \leq \gamma/2$. 
Then 
$$
N_{n - 1}^{\tau_0} \e \| v_1 - v_2 \|_{s_0 + \mu(\mathtt b)} 
\leq \rho 
\ \ \Longrightarrow \ \ 
\Omega_n^\gamma(v_1) \subseteq \Omega_n^{\gamma - \rho}(v_2). 
$$
\end{proposition}

\begin{proof}
{\sc Proof of ${\bf (S1)_0}$-${\bf (S3)_0}$.} 
The claimed properties follow directly 
from Lemma \ref{lemma invertibilita cal U bot}
and from the definition of $\Om_0^\g$. 

\smallskip

\noindent
{\sc Proof of ${\bf (S1)_{n + 1}}, {\bf (S2)_{n + 1}}$.} 
We only prove ${\bf (S1)_{n + 1}}$; the properties in ${\bf (S2)_{n + 1}}$ follow similarly.

Assume the the claimed properties ${\bf (S1)_{n}}$ holds for some $n \geq 0$ and let us prove them at the step $n + 1$. 
Let $\Phi_n = {\rm Id} + \Psi_n$ where $\Psi_n$ is an operator 
to determine. 
We compute 
\begin{equation}\label{primo coniugio Ln Psin}
\begin{aligned}
{\cal L}_n \Phi_n & = \Phi_n ( \omega \cdot \partial_\vphi + {\cal N}_n ) 
+ \omega \cdot \partial_\vphi \Psi_n + [{\cal N}_n, \Psi_n] + \Pi_{N_n} {\cal R}_n 
+ \Pi_{N_n}^\bot {\cal R}_n + {\cal R}_n \Psi_n
\end{aligned}
\end{equation}
where $\mN_n := \zeta \cdot \grad \,{\rm Id} + \mQ_n$; 
recall definition \eqref{def proiettore operatori matrici} 
of the projectors $\Pi_N$, $\Pi_N^\bot$. 
Our purpose is to find a map $\Psi_n$ solving the {\it homological equation} 
\begin{equation}\label{equazione omologica KAM}
\omega \cdot \partial_\vphi \Psi_n + [{\cal N}_n, \Psi_n] + \Pi_{N_n} {\cal R}_n = {\cal Z}_n
\end{equation}
where ${\cal Z}_n$ is the $3 \times 3$ block-diagonal operator
\begin{equation}\label{def cal Zn}
{\cal Z}_n := {\rm diag}_{j \in \Z^3 \setminus \{ 0 \}} \widehat{\cal R}_n(0)_j^j\,. 
\end{equation}

\begin{lemma}\label{Lemma eq omologica riducibilita KAM}
There exists a reversibility-preserving operator $\Psi_n$ defined for all the values of the parameters $\lambda = (\om,\zeta) \in \R^{\nu + 3}$ which satisfies, 
for any $s_0 \leq s \leq S + \mathtt b$, the estimates
\begin{equation}\label{stime eq omologica}
\begin{aligned}
|\Psi_n|_s^{k_0, \gamma}\,,\, 
|\langle D \rangle^{- M} \Psi_n \langle D \rangle^{ M}|_s^{k_0, \gamma} 
\lesssim_{} N_n^{2 \tau_0} \gamma^{- 1} |{\cal R}_n \langle D \rangle^M|_s^{k_0, \gamma}
\end{aligned}
\end{equation}
and solves, for any $\lambda = (\omega, \zeta) \in \Omega_{n + 1}^\gamma$, 
the homological equation \eqref{equazione omologica KAM}. 
\end{lemma}

\begin{proof}
To simplify notations, in this proof 
we drop the index $n$ and we write $+$ instead of $n + 1$. 
By using the matrix $3 \times 3$ block 
representation of linear operators provided in Section \ref{sezione matrici norme} 
and recalling \eqref{def proiettore operatori matrici}, 
the homological equation \eqref{equazione omologica KAM} is equivalent to solve, 
for all $\ell \in \Z^\nu$, $j, j' \in \Z^3 \setminus \{ 0 \}$, 
\begin{equation} \label{homological eq matrici}
L(\ell, j, j') \widehat \Psi_j^{j'}(\ell) 
+ \Pi_N \widehat{\cal R}_j^{j'}(\ell) 
= \widehat{\cal Z}_j^{j'}(\ell)\,, \\
\end{equation}
where we recall that $L(\ell,j,j')$ is defined by 
$$
L(\ell,j,j') : \Mat_{3 \times 3}(\C) \to \Mat_{3 \times 3}(\C), 
\quad B \mapsto \ii (\om \cdot \ell + \zeta \cdot (j-j')) B 
+ \mQ_j^j B - B \mQ_{j'}^{j'} \,.
$$
Since $\Mat_{3 \times 3}(\C)$ is a vector space over $\C$ of dimension $9$, 
$L(\ell,j,j')$ can be represented by a matrix in $\Mat_{9 \times 9}(\C)$. 
By \eqref{cal Nn rid} and \eqref{ansatz}, one has
\begin{equation} \label{est qnjj'}
\| \mQ_j^j \|^{k_0,\g}_{\HS} + \| \mQ_{j'}^{j'} \|^{k_0,\g}_{\HS} 
\lesssim_{k_0} \frac{\e}{|j|} + \frac{\e}{|j'|}\,. 
\end{equation}
For any fixed $(\ell,j,j')$, 
for all $\lm = (\om, \zeta) \in \R^{\nu+3}$, 
the scalar product $f(\lm) := \om \cdot \ell + \zeta \cdot (j-j')$
satisfies 
$\pa_{\om_k} f(\lm) = \ell_k$, 
$\pa_{\zeta_k} f(\lm) = (j-j')_k$, and  
$\pa_\lm^\b f(\lm) = 0$ for all multi-indices $\b$ of length $|\b| \geq 2$.
Thus
\begin{equation} \label{0210.4}
\begin{aligned}
\g^{|\b|} \big| \pa_\lm^\b ( \om \cdot \ell + \zeta \cdot (j-j') ) \big| 
\leq \g \langle \ell, j-j' \rangle 
\quad \ \forall \lm \in \R^{\nu+3}, \ 
|\b| \geq 1.
\end{aligned}
\end{equation}
Since $\e \leq \g$ (see \eqref{KAM smallness condition}), 
the estimates \eqref{est qnjj'}, \eqref{0210.4} imply that 
\begin{equation} \label{0210.5} 
\g^{|\b|} \| \pa_\lm^\b L(\ell,j,j';\lm) \|_\op 
\lesssim \g \langle \ell, j-j' \rangle
\quad \ \forall \lm \in \R^{\nu+3}, 
\  1 \leq |\b| \leq k_0.
\end{equation} 
For each $(\ell,j,j')$, we consider the set 
$\mA := {\cal A}(\ell, j, j')$ of the parameters $\lm = (\om,\zeta)$ 
such that $L(\ell, j, j'; \lm)$ is invertible, and 
\begin{equation}\label{0210.6}
\mK := \mK(\ell, j, j') 
:= \Big\{ \lambda \in \R^{\nu + 3} : 
L(\ell, j, j'; \lm) \text{ is invertible and } 
\| L(\ell, j, j'; \lm)^{-1}\|_\op 
\leq \frac{\langle \ell \rangle^\tau |j|^\tau |j'|^\tau}{\gamma} \Big\}.
\end{equation}
We observe that $\mK \subseteq \mA \subseteq \R^{\nu+3}$, 
$\mA$ is open,
and $\mK$ is closed.
The map $\mA \to \Mat_{9 \times 9}(\C)$, 
$\lm \mapsto L(\ell,j,j';\lm)^{-1}$  
is $k_0$ times differentiable on $\mA$. 
By induction, it is not difficult to prove that, 
for any multi-index $\b$ of positive length, 
the derivative $\pa_\lm^\b (L(\ell,j,j';\lm)^{-1})$ 
is a sum of terms of the form 
$L^{-1} (\pa_\lm^{\b_1} L) L^{-1} \cdots 
L^{-1} (\pa_\lm^{\b_m} L) L^{-1}$ 
where $1 \leq m \leq |\b|$ 
and $\b_1, \ldots, \b_m$ are nonzero multi-indices 
with $\b_1 + \ldots + \b_m = \b$
($L$ briefly denotes $L(\ell,j,j';\lm)$). 
Hence, for all $\lambda \in \mK(\ell,j,j')$, 
using the bound \eqref{0210.5}
and the definition of $\t_0$ in \eqref{definizione alpha beta},
for $1 \leq |\b| \leq k_0$ one has
\begin{equation} \label{0210.7}
\begin{aligned}
\g^{|\b|} \| \pa_\lm^\b (L(\ell,j,j';\lm)^{-1}) \|_\op  
& \lesssim 
\gamma^{-1} 
\langle \ell, j - j' \rangle^{ |\beta|} 
\big( \langle \ell \rangle |j| |j'| \big)^{(|\beta| + 1) \tau}  
\lesssim 
\gamma^{-1} \langle \ell \rangle^{\tau_0} |j|^{\tau_0} |j'|^{\tau_0}\,. 
\end{aligned}
\end{equation} 
By \eqref{0210.6}, the estimate \eqref{0210.7} also holds for $\b=0$, $\lm \in \mK$. 
Hence 
\[
\sup_{0 \leq |\b| \leq k_0} \, 
\sup_{\lm \in \mK} \, 
\g^{|\b|} \| \pa_\lm^\b (L(\ell,j,j';\lm)^{-1}) \|_\op 
\lesssim  \, \gamma^{-1} \langle \ell \rangle^{\tau_0} |j|^{\tau_0} |j'|^{\tau_0}. 
\]
By Whitney extension theorem (see, e.g., Appendix B in \cite{BBHM}
where the weights $\g^{|\b|}$ are also considered), 
the function $\mK \to \Mat_{9 \times 9}(\C)$, 
$\lm \mapsto L(\ell,j,j';\lm)^{-1}$  
admits an extension to $\R^{\nu+3}$, 
which we denote by $L(\ell,j,j'; \lm)^{-1}_{ext}$, 
satisfying 
\begin{equation} \label{0210.14}
\| L(\ell,j,j')^{-1}_{ext} \|^{k_0,\g}_\op 
:= \sup_{0 \leq |\b| \leq k_0} \, 
\sup_{\lm \in \R^{\nu+3}} \, 
\g^{|\b|} \| \pa_\lm^\b (L(\ell,j,j';\lm)^{-1}) \|_\op 
\lesssim  \, \gamma^{-1} \langle \ell \rangle^{\tau_0} |j|^{\tau_0} |j'|^{\tau_0}, 
\end{equation}
where the implicit constant in \eqref{0210.14} 
does not depend on $(\ell,j,j')$. 
The extension $L(\ell,j,j';\lm)^{-1}_{ext}$ 
coincides with $L(\ell,j,j';\lm)^{-1}$ 
for $\lm \in \mK(\ell,j,j')$. 


Consider the set $\Om_+^\g$. 
By construction, 
$L(\ell,j,j';\lm)^{-1}_{ext} = L(\ell,j,j';\lm)^{-1}$ 
for all $\lm \in \Om_{+}^\g$
because $\Om_{+}^\g \subseteq \mK(\ell,j,j')$.
For every $\lm = (\omega, \zeta) \in \R^{\nu+3}$ 
we define $\Psi(\lm)$ as
\begin{equation}\label{sol eq omologica KAM}
(\widehat \Psi(\lm))_j^{j'}(\ell) 
:= \begin{cases}
- L(\ell, j, j';\lm)^{-1}_{ext} \widehat{\cal R}(\lm)_j^{j'}(\ell) 
& \text{if $(\ell, j, j') \neq (0, j, j)$, 
\ $|\ell|, |j - j'| \leq N$}, \\
0 & \text{otherwise.}
\end{cases}
\end{equation}
Then, for all $\lm \in \Om_{+}^\g$, $\Psi$ solves the homological equation 
\eqref{homological eq matrici}. 
Furthermore, for $\beta \in \N^{\nu + 3}$, $|\beta| \leq k_0$ 
and $(\ell, j, j') \neq (0,  j, j)$, $|\ell|, |j - j'| \leq N_n$, 
using \eqref{0210.14} one computes 
\begin{equation}\label{polpettone 0}
\begin{aligned}
\g^{|\b|} \| \partial_\lambda^\beta \widehat \Psi_j^{j'}(\ell) \|_{\HS} 
& \lesssim \g^{|\b|} \sum_{\beta_1 + \beta_2 = \beta} 
\| \partial_\lambda^{\beta_1} L(\ell, j, j')^{- 1} \|_\op 
\| \partial_\lambda^{\beta_2} \widehat{\cal R}_j^{j'}(\ell)\|_{\HS}  
\\
& \lesssim  
\g^{-1} \langle \ell \rangle^{\tau_0} |j|^{\tau_0} |j'|^{\tau_0} 
\sum_{|\beta_2| \leq |\b|} \gamma^{|\beta_2|} 
\| \partial_\lambda^{\beta_2} \widehat{\cal R}_j^{j'}(\ell)\|_{\HS}\,. 
\end{aligned}
\end{equation}
To estimate the decay norm of the operator $\langle D \rangle^{- M} \Psi \langle D \rangle^{M}$ 
we need to estimate $\langle j \rangle^{- M}\| \partial_\lambda^\beta \widehat \Psi_j^{j'}(\ell) \|_{\HS} \langle j' \rangle^{M}$. 
By triangular inequality, since $|j-j'| \leq N$, 
one has 
\[
|j'|^{\t_0} 
\leq (|j'-j| + |j|)^{\t_0}
\leq (N + |j|)^{\t_0}
\lesssim (N |j|)^{\t_0}
\]
and therefore, since also $|\ell| \leq N$,
\begin{equation} \label{1903.1}
\langle j \rangle^{-M}
\langle \ell \rangle^{\tau_0} |j|^{\tau_0} |j'|^{\tau_0} 
\langle j' \rangle^{M}
\lesssim
\langle j \rangle^{-M}
N^{\tau_0} |j|^{\tau_0} 
(N |j|)^{\t_0}
\langle j' \rangle^{M}
\lesssim
N^{2\tau_0} \langle j' \rangle^{M}
\end{equation}
because, by \eqref{definizione alpha beta}, $2\t_0 \leq M$.
Hence, multiplying \eqref{polpettone 0} 
by $\langle j \rangle^{-M} \langle j' \rangle^{M}$, 
using \eqref{1903.1}, 
and recalling Definition \ref{block decay norm}, 
we obtain the estimate \eqref{stime eq omologica} for 
$\langle D \rangle^{- M} \Psi \langle D \rangle^{M}$. 
The estimate for $\Psi$ can be proved similarly. 
\end{proof}

%
%

From \eqref{stime eq omologica} and \eqref{stime cal Rn rid}
one obtains that, for any $s_0 \leq s \leq S$, 
\begin{equation}\label{stime Psin neumann}
\begin{aligned}
|\Psi_n|_{s}^{k_0, \gamma} \,,\, 
|\langle D \rangle^{- M}\Psi_n \langle D \rangle^M|_{s}^{k_0, \gamma}
& \lesssim  N_n^{2 \tau_0} \gamma^{- 1} |{\cal R}_n \langle D \rangle^M|_s^{k_0, \gamma}
\lesssim_{s } N_n^{2 \tau_0} N_{n - 1}^{- \mathtt a} 
\e \gamma^{- 1} \| v \|_{s + \mu(\mathtt b)}^{k_0, \gamma} , \\
|\Psi_n|_{s + \mathtt b}^{k_0, \gamma}\,,\, |\langle D \rangle^{- M}\Psi_n \langle D \rangle^M|_{s + \mathtt b}^{k_0, \gamma} & \lesssim N_n^{2 \tau_0} \gamma^{- 1} |{\cal R}_n \langle D \rangle^M|_{s + \mathtt b}^{k_0, \gamma} \lesssim_{s } N_n^{2 \tau_0} N_{n - 1} \e \gamma^{- 1}  \| v \|_{s + \mu(\mathtt b)}^{k_0, \gamma},
\end{aligned}
\end{equation} 
which are the estimates \eqref{stime Psi n rid} at the step $n + 1$. 
Using the ansatz \eqref{ansatz} with $\mu = \mu(\mathtt b)$, by \eqref{definizione alpha beta}, one has
\begin{equation} \label{2103.1}
|\Psi_n|_{s_0}^{k_0, \gamma} 
\lesssim  
N_n^{2 \tau_0} \gamma^{- 1} |{\cal R}_n \langle D \rangle^M|_{s_0}^{k_0, \gamma}
\lesssim 
N_n^{2 \tau_0} N_{n-1}^{- \mathtt a} \e \gamma^{-1} 
\lesssim N_0^{2 \tau_0} \e \gamma^{- 1}. 
\end{equation}
Then for $\e \gamma^{- 1}$ small enough, 
by Lemma \ref{proprieta standard norma decay}-$(iv)$, 
$\Phi_n = {\rm Id} + \Psi_n$ is invertible 
and, 
for any $s_0 \leq s \leq S$, 
\begin{equation}\label{stime Phi n inv - Id}
|\Phi_n^{- 1} - {\rm Id}|_s^{k_0, \gamma} 
\lesssim_{s}  |\Psi_n|_s^{k_0, \gamma}\,,
\qquad 
|\Phi_n^{- 1} - {\rm Id}|_{s + \mathtt b}^{k_0, \gamma} 
\lesssim_{s}  |\Psi_n|_{s + \mathtt b}^{k_0, \gamma}. 
\end{equation}
We define 
\begin{equation}\label{2003.1} 
\begin{aligned}
{\cal L}_{n + 1} & := \omega \cdot \partial_\vphi + {\cal N}_{n + 1} + {\cal R}_{n + 1}\,, 
\quad \ 
{\cal N}_{n + 1} :=  \zeta \cdot \nabla {\rm Id} + {\cal Q}_{n + 1}\,, 
\\
{\cal Q}_{n + 1} & : = {\cal Q}_n + {\cal Z}_n\,, 
\quad \ 
{\cal R}_{n + 1} := ( \Phi_n^{- 1} - {\rm Id} ) {\cal Z}_n 
+ \Phi_n^{- 1} ( \Pi_{N_n}^\bot {\cal R}_n + {\cal R}_n \Psi_n ). 
\end{aligned}
\end{equation}
All the operators in \eqref{2003.1}
are defined for any $\lm = (\om,\zeta) \in \R^{\nu+3}$. 
Since $\Psi_n, \Phi_n, \Phi_n^{- 1}$ are reversibility preserving and ${\cal N}_n, {\cal R}_n$ are reversible operators, one gets that ${\cal N}_{n + 1}$, ${\cal R}_{n + 1}$ are reversible operators. 
Moreover, by \eqref{primo coniugio Ln Psin}, \eqref{equazione omologica KAM},
for $(\om,\zeta) \in \Om^\g_{n+1}$ one has the 
identity
$\Phi_n^{-1} \mL_n \Phi_n = \mL_{n+1}$, 
which is \eqref{coniugazione rid} at the step $n+1$. 
For any $j \in \Z^3 \setminus \{ 0 \}$, 
recalling the definition \eqref{def cal Zn} of $\mZ_n$,
$$
\| ({\cal N}_{n + 1})_j^j - ({\cal N}_{n })_j^j\|_{\HS}^{k_0, \gamma} = \| ({\cal Q}_{n + 1})_j^j - ({\cal Q}_{n })_j^j\|_{\HS}^{k_0, \gamma}  \leq \| \widehat{\cal R}_n(0)_j^j \|_{\HS} \lesssim_{k_0} |{\cal R}_n \langle D \rangle^M|_{s_0}^{k_0, \gamma} \langle j \rangle^{- M}\,.
$$
Then the estimate \eqref{stime cal Rn rid}
implies the estimate \eqref{cal Nn - N n - 1} at the step $n + 1$. 
The estimate \eqref{cal Nn rid} at the step $n + 1$ follows, as usual, 
by a telescoping argument, 
using the fact that $\sum_{n \geq 0} N_{n - 1}^{- \mathtt a}$ is convergent 
since $\mathtt a > 0$ (see \eqref{definizione alpha beta}). 
Now we prove the estimates \eqref{stime cal Rn rid} at the step $n + 1$. 
By \eqref{2003.1}, 
one has
$$
{\cal R}_{n + 1} \langle D \rangle^M 
= ( \Phi_n^{- 1} - {\rm Id} ) {\cal Z}_n \langle D \rangle^M 
+ \Phi_n^{- 1} (\Pi_{N_n}^\bot{\cal R}_n \langle D \rangle^M) 
+ \Phi_n^{- 1} ({\cal R}_n \langle D \rangle^M) 
(\langle D \rangle^{- M} \Psi_n \langle D \rangle^M).
$$
By the estimates 
\eqref{stime Psin neumann}, 
\eqref{2103.1},
\eqref{stime Phi n inv - Id}, 
by applying Lemma \ref{proprieta standard norma decay}-$(ii),(v)$, 
Lemma \ref{lemma proiettori decadimento}, 
the smallness condition \eqref{KAM smallness condition} 
and the induction estimate \eqref{stime cal Rn rid}, 
for any $s_0 \leq s \leq S$ we get 
\begin{equation}\label{stime R n + 1 R n}
\begin{aligned}
& |{\cal R}_{n + 1} \langle D \rangle^M|_s^{k_0, \gamma} 
\lesssim_{s} N_n^{- \mathtt b} |{\cal R}_n \langle D \rangle^M|_{s + \mathtt b}^{k_0, \gamma} + N_n^{2 \tau_0} \gamma^{- 1} |{\cal R}_n \langle D \rangle^M|_s^{k_0, \gamma} |{\cal R}_n \langle D \rangle^M|_{s_0}^{k_0, \gamma}\,, \\
& |{\cal R}_{n + 1} \langle D \rangle^M|_{s + \mathtt b}^{k_0, \gamma} 
\lesssim_{s} |{\cal R}_n \langle D \rangle^M|_{s + \mathtt b}^{k_0, \gamma}\,. 
\end{aligned}
\end{equation}
Using the definition of the constants $\mathtt a, \mathtt b$ in \eqref{definizione alpha beta} and the smallness condition \eqref{KAM smallness condition}, 
taking $N_0 = N_0(S, k_0, \tau)> 0$ large enough, 
one gets the estimate \eqref{stime cal Rn rid} at the step $n + 1$. 

\medskip

\noindent
{\sc Proof of ${\bf (S3)_{n + 1}}$.} Assume that we have proved the claimed statement for some $n \in \N$ and let us prove it at the step $n + 1$. Let $\lambda = (\omega, \zeta) \in  \Omega_{n + 1}^\gamma (v_1)$. By the definition of the sets $\Omega_n^\gamma$ (see \eqref{insiemi di cantor rid}) and using the induction hypothesis, one has that 
\begin{equation}\label{solita prop inclusione cantor}
 \Omega_{n + 1}^\gamma (v_1) \subseteq  \Omega_{n }^\gamma (v_1) \subseteq \Omega_n^{\gamma - \rho}(v_2)\,.
\end{equation}
The property \eqref{solita prop inclusione cantor}, 
together with ${\bf (S2)_n}$, implies that 
\begin{equation}\label{Delta 12 cal Qn Lemma}
\begin{aligned}
&  \| \Delta_{12} ({\cal Q}_n)_j^j \|_{\HS} \lesssim \e |j|^{- 1}  \| v_1 - v_2\|_{s_0 + \mu(\mathtt b)}, \quad \lambda = (\omega, \zeta) \in  \Omega_{n + 1}^\gamma(v_1) \,. 
\end{aligned}
\end{equation}
Let $\lm \in \Om_{n+1}^\g(v_1)$. 
To prove the claimed inclusion, 
we have to show that for any $(\ell, j, j')$ with 
$\ell \in \Z^\nu$, 
$j, j' \in \Z^3 \setminus \{  0 \}$, 
$(\ell, j, j') \neq (0, j, j)$, 
$|\ell|, |j- j'| \leq N_n$, 
the linear operator 
$$
L_n(\ell, j, j'; v_2 ) \equiv L_n(\ell, j, j'; \lambda, v_2(\lambda)) 
: \Mat_{3 \times 3}(\C) \to \Mat_{3 \times 3}(\C)
$$
(see \eqref{operatore blocchi}) is invertible 
and $\| L_n(\ell, j, j'; v_2 )^{- 1} \|_\op 
\leq (\g-\rho)^{-1} \langle \ell \rangle^\tau |j|^\tau |j'|^\tau$. 
We distinguish two cases. 

\medskip

\noindent
{\sc Case 1: ${\rm min}\{ |j|, |j'| \} \geq N_n^{\tau}$.} By recalling \eqref{cal Nn rid} and the definitions \eqref{definizione mult left right}, we write 
\begin{equation}
\begin{aligned}
& L_n(\ell, j, j'; v_2) = \Omega(\ell, j, j') + \Delta_n(j, j'; v_2)\,, \\
& \Omega(\ell, j, j') := \ii (\omega \cdot \ell + \zeta \cdot (j - j')){\rm Id}\,, \quad  \Delta_n(j, j'; v_2) := M_L\big( ({\cal Q}_n)_j^j\big) - M_R\big(({\cal Q}_n)_{j'}^{j'} \big)\,.
\end{aligned}
\end{equation} 
Since $\lambda = (\omega, \zeta) \in DC(\gamma, \tau)$ 
(recall \eqref{def cantor set trasporto}), 
the operator $\Omega(\ell, j, j') : \Mat_{3 \times 3}(\C) \to \Mat_{3 \times 3}(\C)$ 
is invertible and satisfies the estimate 
\begin{equation}\label{stima Omega ell j j' inv}
\| \Omega(\ell, j, j')^{- 1}\|_\op 
\leq \frac{\langle \ell, j - j' \rangle^\tau }{{\frak C}_0 \g} \,. 
\end{equation}
Furthermore, by the estimates \eqref{MLAMRA elementry}, \eqref{cal Nn rid}, one has that 
\begin{equation}\label{stima Delta n j j'}
\| \Delta_n(j, j'; v_2) \|_\op 
\lesssim \frac{\e}{{\rm min}\{ |j|\,,\, |j'| \}}\,. 
\end{equation}
Since $\langle \ell, j - j' \rangle \lesssim N_n$ and
$\min \{ |j|, |j'| \} \geq N_n^\tau$, 
the estimates \eqref{stima Omega ell j j' inv}, \eqref{stima Delta n j j'} 
immediately imply that 
$$
\| \Omega(\ell, j, j')^{- 1} \Delta_n(j, j'; v_2) \|_{\rm Op} \lesssim \e \gamma^{- 1}\,. 
$$
Hence, for $\e \gamma^{- 1}$ small enough, the operator $L_n(\ell, j, j'; v_2)$ is invertible by Neumann series and 
\begin{equation} \label{inclusione non ris modi alti}
\| L_n(\ell, j, j'; v_2)^{- 1}\|_\op 
\leq 2 \| \Omega(\ell, j, j')^{- 1}\|_\op 
\leq \frac{ 2 \langle \ell, j - j' \rangle^\tau }{{\frak C}_0 (\gamma - \rho)} 
\leq \frac{ \langle \ell \rangle^\tau |j|^\tau |j'|^\tau }{ \gamma - \rho}
\end{equation}
by taking ${\frak C}_0 > 0$ large enough. 

\medskip

\noindent
{\sc Case 2: ${\rm min}\{ |j|, |j'| \} \leq N_n^\tau$.} By recalling \eqref{cal Nn rid} and the definitions \eqref{definizione mult left right}, we write 
\begin{equation}\label{splitting Ln caso modi bassi}
\begin{aligned}
& L_n(\ell, j, j'; v_2) = L_n(\ell, j, j'; v_1) + \Gamma_n(j, j')\,, \\
& \Gamma_n(j, j') := -  M_L\big( (\Delta_{12}{\cal Q}_n)_j^{j}\big) + M_R \big( (\Delta_{12}{\cal Q}_n)_{j'}^{j'}\big)\,.
\end{aligned}
\end{equation}
By \eqref{Delta 12 cal Qn Lemma}, using the property \eqref{MLAMRA elementry}, 
one gets 
\begin{equation}\label{stima Gamma n j j'}
\| \Gamma_n(j, j') \|_\op \lesssim \e  \| v_1 - v_2 \|_{s_0 + \mu(\mathtt b)}\,.
\end{equation}
Furthermore, since in this case we have $|\ell|, |j - j'| \leq N_n$ and ${\rm min}\{ |j|, |j| \} \leq N_n^\tau$ one has that 
$$
\langle \ell \rangle^\tau |j|^\tau |j'|^\tau 
\lesssim N_n^{\tau + 2 \tau^2}\,.
$$
Hence, 
for $\lambda = (\omega, \zeta) \in  \Omega_{n + 1}^\gamma (v_1)$, 
one has 
$\| L_n(\ell, j, j'; v_1)^{- 1} \|_\op 
\lesssim N_n^{\tau + 2 \tau^2} \gamma^{- 1}$,
which, together with \eqref{stima Gamma n j j'}, 
implies that 
\begin{equation}\label{orca assassina}
\| L_n(\ell, j, j'; v_1)^{- 1}  \Gamma_n(j, j') \|_\op 
\leq C N_n^{\tau + 2 \tau^2} \e \gamma^{- 1} \| v_1 - v_2 \|_{s_0 + \mu(\mathtt b)}
\end{equation}
for some $C > 0$. 
If $C N_n^{ \tau + 2 \tau^2} \e \g^{-1} \| v_1 - v_2 \|_{s_0 + \mu(\mathtt b)} 
\leq \rho \g^{-1}$, 
then, by Neumann series, 
$L_n(\ell, j, j'; v_2)$ is invertible and
\begin{equation}\label{inversione Ln v2 modi bassi}
\| L_n(\ell, j, j'; v_2)^{- 1} \|_\op 
\leq \frac{\langle \ell \rangle^\tau \langle j \rangle^\tau \langle j' \rangle^\tau}{\gamma - \rho}\,. 
\end{equation}

Thus, \eqref{inclusione non ris modi alti}, \eqref{inversione Ln v2 modi bassi} imply that $\lambda = (\omega, \zeta) \in \Omega_{n + 1}^{\gamma - \rho}(v_2)$, 
and the proof is concluded.
\end{proof}

\subsection{Convergence}
\begin{lemma}\label{lemma blocchi finali}
For any $j \in \Z^3 \setminus \{ 0 \}$, the sequence 
$({\cal N}_n)_j^j = \ii\, \zeta \cdot j \, {\rm Id} + ({\cal Q}_n)_j^j$, $n \in \N$, 
converges in the norm $\| \ \|_{\HS}^{k_0, \gamma}$ to some limit
\begin{equation}\label{def cal N infty nel lemma}
({\cal N}_\infty)_j^j = \ii \zeta \cdot j \, {\rm Id} + ({\cal Q}_\infty)_j^j
\, \in \Mat_{3 \times 3}(\C),
\end{equation}
and 
\begin{equation}\label{stime forma normale limite}
\begin{aligned}
& \| ({\cal N}_\infty)_j^j - ({\cal N}_n)_j^j \|^{k_0, \gamma}_{\HS} = \| ({\cal Q}_\infty)_j^j - ({\cal Q}_n)_j^j \|^{k_0, \gamma}_{\HS} \lesssim \e |j|^{- M} N_{n - 1}^{- \mathtt a} \| v \|_{s_0 + \mu(\mathtt b)}^{k_0, \gamma}\,, \\
& \| ({\cal Q}_\infty)_j^j\|_{\HS}^{k_0, \gamma} \lesssim  \e |j|^{- 1} \| v \|_{s_0 + \mu(\mathtt b)}^{k_0, \gamma}\,. 
\end{aligned}
\end{equation}
Moreover the $3 \times 3$ block diagonal operator ${\cal Q}_\infty := {\rm diag}_{j \in \Z^3 \setminus \{ 0 \}} ({\cal Q}_\infty)_j^j$ is real and reversible. 
\end{lemma}

\begin{proof}
The lemma follows in a standard way 
by the estimates \eqref{cal Nn rid}, \eqref{cal Nn - N n - 1} 
and a telescoping argument. 
\end{proof}

Now we define 
\begin{equation}\label{trasformazioni tilde ridu}
\widetilde \Phi_n := \Phi_0 \circ \Phi_1 \circ \ldots \circ \Phi_n, \quad n \in \N\,. 
\end{equation}


\begin{lemma}\label{lemma convergenza trasformazioni}
Let $S > s_0$ and assume \eqref{KAM smallness condition} and \eqref{ansatz} with $\mu = \mu(\mathtt b)$. 
For any $s_0 \leq s \leq S$, the sequence $\widetilde \Phi_n$, $n \in \N$, 
converges in norm $| \ |_s^{k_0, \gamma}$ 
to a real, reversibility-preserving and invertible map $\Phi_\infty$, 
with
$$
\begin{aligned}
& |\Phi_\infty^{\pm 1} - \widetilde \Phi_n^{\pm 1}|_s^{k_0, \gamma} 
\lesssim_{s} N_{n + 1}^{2 \tau_0} N_n^{- \mathtt a} \e \gamma^{- 1} 
\| v \|_{s + \mu(\mathtt b)}^{k_0, \gamma}\,, 
\quad 
|\Phi_\infty^{\pm 1} - {\rm Id}|_s^{k_0, \gamma} 
\lesssim_{s} N_0^{2 \tau_0} \e \gamma^{- 1} 
\| v \|_{s + \mu(\mathtt b)}^{k_0, \gamma}\,.
\end{aligned}
$$
\end{lemma}

\begin{proof}
The lemma follows by 
\eqref{stime Psi n rid}, 
\eqref{trasformazioni tilde ridu}, 
arguing as in Corollary 4.1 in \cite{BBM-Airy}. 
\end{proof}

For any $(\ell, j, j')$, $\ell \in \Z^\nu$, $j,j' \in \Z^3 \setminus \{ 0 \}$, 
we define 
\begin{equation}\label{L infty l j j fin}
L_\infty(\ell, j, j') := \ii \omega \cdot \ell + M_L \big( ({\cal N}_\infty)_j^j \big) 
- M_R \big( ({\cal N}_\infty)_{j'}^{j'} \big)
\end{equation}
and the set $\Omega_\infty^\gamma = \Omega_\infty^\gamma(v)$ as 
\begin{equation}\label{cantor finale ridu}
\begin{aligned}
\Omega_\infty^\gamma 
:= \Big\{ \lambda \in DC(\g,\t) 
& : L_\infty(\ell, j, j'; \lambda) \text{ is invertible and } 
\| L_\infty(\ell, j, j'; \lambda)^{- 1} \|_\op 
\leq \frac{\langle \ell \rangle^\tau |j|^\tau |j'|^\tau}{2 \gamma} \\
& \quad \forall \ell \in \Z^\nu, \ \ j, j' \in \Z^3 \setminus \{ 0 \}, 
\ \ (\ell, j, j') \neq (0, j, j)\Big\}\,. 
\end{aligned}
\end{equation}

\begin{lemma}\label{prima inclusione cantor}
One has $\Omega_\infty^\gamma \subseteq \cap_{n \geq 0} \, \Omega_n^\gamma$.
\end{lemma}

\begin{proof}
We prove by induction that $\Omega_\infty^\gamma \subseteq \Omega_n^\gamma$ 
for any integer $n \geq 0$. 
For $n = 0$ the statement is trivial 
since $\Omega_0^\gamma := DC(\g,\t)$ 
(see Proposition \ref{prop riducibilita}).
Now assume that $\Omega_\infty^\gamma \subseteq \Omega_n^\gamma$ for some $n \geq 0$ and let us show that $\Omega_\infty^\gamma \subseteq \Omega_{n + 1}^\gamma$. Let $\lambda \in \Omega_\infty^\gamma$, $\ell \in \Z^\nu$, $j, j' \in \Z^3 \setminus \{ 0 \}$, $(\ell, j, j') \neq (0, j, j)$ and $|\ell|, |j - j'| \leq N_n$. 
By using \eqref{cal Nn rid}, \eqref{operatore blocchi}, \eqref{L infty l j j fin},  
we write 
\begin{equation}\label{Ln L infty 0}
\begin{aligned}
L_n(\ell, j, j') & = L_\infty(\ell,j, j') + L_n(\ell, j, j') - L_\infty(\ell, j, j') 
= L_\infty(\ell,j, j') \big( {\rm Id} + {\cal S}_n(\ell, j, j') \big), \\
{\cal S}_n(\ell, j, j') & := L_\infty(\ell,j, j')^{- 1}\Big( M_L\big( ({\cal Q}_n - {\cal Q}_\infty)_j^j  \big) - M_R\big( ({\cal Q}_n - {\cal Q}_\infty)_{j'}^{j'} \big) \Big)\,.
\end{aligned}
\end{equation}
By \eqref{MLAMRA elementry}, \eqref{stime forma normale limite}, \eqref{cantor finale ridu}, using \eqref{ansatz}, one obtains that 
\begin{equation}\label{stima cal Sn ell j j'}
\begin{aligned}
\| {\cal S}_n(\ell, j, j')  \|_{\rm Op} \lesssim  \e \gamma^{- 1}\Big(\frac{\langle \ell \rangle^\tau |j|^\tau |j'|^\tau}{|j|^M} + \frac{\langle \ell \rangle^\tau |j|^\tau |j'|^\tau}{|j'|^M}  \Big) N_{n - 1}^{- \mathtt a}\,. 
\end{aligned}
\end{equation}
By the triangular inequality, using that $|\ell|, |j - j'| \leq N_n$ and recalling that by \eqref{definizione alpha beta}, $M > 2 \tau$, one obtains the bound
$$
\begin{aligned}
\frac{\langle \ell \rangle^\tau |j|^\tau |j'|^\tau}{|j|^M} + \frac{\langle \ell \rangle^\tau |j|^\tau |j'|^\tau}{|j'|^M}  & \lesssim  N_n^{2 \tau} 
\end{aligned}
$$
implying that (see \eqref{definizione alpha beta})
\begin{equation}\label{stima cal Sn ell j j'}
\begin{aligned}
\| {\cal S}_n(\ell, j, j') \|_\op 
\lesssim \e \gamma^{- 1} N_n^{2 \tau}N_{n - 1}^{- \mathtt a} 
\lesssim  \e \gamma^{- 1} N_0^{2\tau}\,. 
\end{aligned}
\end{equation}
Hence by \eqref{Ln L infty 0}, \eqref{stima cal Sn ell j j'}, 
taking $\e \gamma^{-1}$ small enough, 
the operator $L_n(\ell, j, j')$ is invertible by Neumann series, 
with $\| L_n(\ell, j, j'; \lambda)^{- 1}\|_\op 
\leq \g^{-1} \langle \ell \rangle^\tau |j|^\tau |j'|^\tau$. 
This shows that $\lambda \in \Omega_{n + 1}^\gamma$, 
and the proof is complete. 
\end{proof}

\begin{lemma}\label{lemma coniugio finale}
For any $\lambda \in \Omega_\infty^\gamma$, one has 
\[ 
\Phi_\infty^{- 1} {\cal L}_0 \Phi_\infty 
= {\cal L}_\infty := \omega \cdot \partial_\vphi + {\cal N}_\infty
\] 
where the operator ${\cal L}_0$ is given in \eqref{coniugazione-pre-riducibilita} 
and the $3 \times 3$ block-diagonal operator ${\cal N}_\infty$ 
in Lemma \ref{lemma blocchi finali}. 
Furthermore, the operator ${\cal L}_\infty$ is real and reversible. 
\end{lemma}

\begin{proof}
By Lemma \ref{prima inclusione cantor} 
and Proposition \ref{prop riducibilita}, one has 
$\widetilde \Phi_n^{- 1} {\cal L}_0 \widetilde \Phi_n 
= \omega \cdot \partial_\vphi + {\cal N}_n + {\cal R}_n$ for all $n \geq 0$. 
The claimed statement then follows by passing to the limit as $n \to \infty$, 
by using \eqref{stime cal Rn rid} and Lemmata \ref{lemma blocchi finali}, \ref{lemma convergenza trasformazioni}. 
\end{proof}

\subsection{Inversion of the operator ${\cal L}$.}\label{sezione inversione sommario}

By 
Lemmata \ref{lemma invertibilita cal U bot}, \ref{lemma coniugio finale}, 
on the set $\Omega_\infty^\gamma$ one has 
\begin{equation}\label{coniugio finale cal L}
{\cal L} = {\cal W}_\infty {\cal L}_\infty {\cal W}_\infty^{- 1}, \quad {\cal W}_\infty := {\cal E}_\bot \Phi_\infty\,. 
\end{equation} 


\begin{lemma}\label{tame trasformazioni finali}
Let $S > s_0$, $\tau > 0$, $\gamma \in (0, 1)$ and assume \eqref{ansatz} 
with $\mu = \mu (\mathtt b)$ (see \eqref{definizione alpha beta}). 
Then there exists $\delta = \delta(S, k_0, \tau) \in (0, 1)$ small enough 
such that, if $\e \gamma^{-1} \leq \delta$, 
then the maps ${\cal W}_\infty^{\pm 1} : H^s_0 \to H^s_0$, $s_0 \leq s \leq S$ are real and reversibility preserving and they satisfy the estimate 
$$
\| {\cal W}_\infty^{\pm 1} h\|_s^{k_0, \gamma}\lesssim_{s} \| h \|_s^{k_0, \gamma} 
+ \| v \|_{s + \mu(\mathtt b)}^{k_0, \gamma} \| h \|_{s_0}^{k_0, \gamma}
\quad \ \forall s_0 \leq s \leq S\,. 
$$
\end{lemma}

\begin{proof}
Apply Lemmata \ref{lemma invertibilita cal U bot}, \ref{lemma convergenza trasformazioni} and use Lemma \ref{proprieta standard norma decay}-$(i)$. 
\end{proof}

We define the set 
$\Lambda_\infty^\gamma \equiv \Lambda_\infty^\gamma (v)$ as  
\begin{equation}\label{prime Melnikov}
\Lambda_\infty^\gamma 
:= \Big\{ \lambda = (\omega, \zeta) \in  \R^{\nu+3} : 
\big\| \big( \ii \omega \cdot \ell \, {\rm Id} 
+ ({\cal N}_\infty)_j^j \big)^{- 1} \big\|_\op 
\leq \frac{\langle \ell \rangle^\tau |j|^\tau}{2 \gamma} 
\quad \forall (\ell, j) \in \Z^\nu \times (\Z^3 \setminus \{ 0 \})\Big\}\,. 
\end{equation}

In the next Lemma we construct a right inverse for the operator ${\cal L}_\infty$. 

\begin{lemma}[{\textbf{Inversion of ${\cal L}_\infty$}}]
\label{lemma inversione cal L infty}
For any $\lambda = (\omega , \zeta) \in \Lambda_\infty^\gamma$, 
for any $s \geq  0$ and $h(\lambda) \in H^{s + \tau_0}_0 \cap X$ 
there exists a solution 
$g(\lambda) : = {\cal L}_\infty^{- 1}(\lambda) h(\lambda) \in H^s_0 \cap Y$ 
of the equation ${\cal L}_\infty (\lambda) g(\lambda) = h(\lambda)$. 
Furthermore, the inverse operator admits an extension 
to the whole parameter space $\R^{\nu + 3}$ (which we denote in the same way) 
satisfying the tame estimate
$$
\| {\cal L}_\infty^{- 1} h \|_s^{k_0, \gamma} \lesssim \gamma^{- 1} \| h \|_{s + \tau_0}^{k_0, \gamma}\,. 
$$
\end{lemma}

\begin{proof}
Let $\lambda = (\omega, \zeta) \in \Lambda_\infty^\gamma$. Then the equation ${\cal L}_\infty g = h$ admits the solution 
\begin{equation}\label{soluzione prob diagonalizzato}
\begin{aligned}
& g (\vphi, x) = \sum_{(\ell, j) \in \Z^\nu \times (\Z^3 \setminus \{ 0 \})} 
B(\ell, j) \, \widehat h(\ell, j) e^{\ii (\ell \cdot \vphi + j \cdot x)}\,, \\
& B(\ell, j) := \big(\ii \omega \cdot \ell \, {\rm Id}+ ({\cal N}_\infty)_j^j \big)^{- 1} 
\in \Mat_{3 \times 3}(\C), 
\quad (\ell, j) \in \Z^\nu \times (\Z^3 \setminus \{ 0\})\,. 
\end{aligned}
\end{equation}
Arguing as in the proof of Lemma \ref{Lemma eq omologica riducibilita KAM} (see in particular \eqref{0210.7}-\eqref{0210.14}) one constructs an extension of $B(\ell, j; \lambda)$ that we denote by $(B(\ell, j; \lambda))_{ext}$ defined for any $\lambda \in \R^{\nu + 3}$ and satisfying the estimate
$$
\| (B(\ell, j))_{ext} \|_{\HS}^{k_0, \gamma} \lesssim \gamma^{- 1} \langle \ell \rangle^{\tau_0} |j|^{\tau_0}\,. 
$$
The claimed estimate then follows by estimating 
$| \partial_\lambda^\beta \big[ (B(\ell, j))_{ext} \widehat h(\ell, j) \big]|$ 
for any $\beta \in \N^{\nu + 3}$, $|\beta| \leq k_0$, 
arguing as in \eqref{polpettone 0}. 
\end{proof}


\begin{proposition}\label{inversione linearized}
Let $S > s_0$, $\tau > 0$, $\gamma \in (0, 1)$. 
Then there exists $\overline \mu = \overline \mu(k_0, \tau, \nu) > \mu (\mathtt b)$ 
(see \eqref{definizione alpha beta}), 
$\delta = \delta(S, k_0, \tau) \in (0, 1)$ such that 
if \eqref{ansatz} holds with $\mu_0 \geq s_0 + \overline \mu$ 
and $\e \gamma^{-1} \leq \delta$, 
for any $\lambda = (\omega, \zeta) \in \Omega_\infty^\gamma \cap \Lambda_\infty^\gamma$, 
any $s_0 \leq s \leq S$, 
any $h(\lambda ) \in H^{s + \overline \mu}_0 \cap X$ 
there exists a solution $g := {\cal L}^{- 1} h \in H^{s }_0 \cap Y$ 
of the equation ${\cal L} g = h$. 
Moreover ${\cal L}^{- 1}$ admits an extension to the whole parameter space $\R^{\nu + 3}$ 
(which we denote in the same way) that satisfies the tame estimates
$$
\| {\cal L}^{- 1} h \|_s^{k_0, \gamma} 
\lesssim_{s } \gamma^{-1} \big( \| h \|_{s + \overline \mu}^{k_0, \gamma} 
+ \| v \|_{s + \overline\mu}^{k_0, \gamma} \, \| h \|_{s_0 + \overline\mu}^{k_0, \gamma} \big), \quad \ s_0 \leq s \leq S.  
$$
\end{proposition}

\begin{proof}
The proposition is a straightforward consequence of formula \eqref{coniugio finale cal L} and Lemmata \ref{tame trasformazioni finali}, \ref{lemma inversione cal L infty}. 
\end{proof}
\section{The Nash-Moser iteration}\label{sezione Nash Moser}

In this section we construct the solution of the equation ${\cal F}(v) = 0$ 
(see \eqref{equazione cal F vorticita}) by means of a Nash Moser nonlinear iteration. 
We denote by $\Pi_n$ the orthogonal projector $\Pi_{N_n}$ (see \eqref{def:smoothings}) 
on the finite dimensional space
$$
{\mathcal H}_n := \big\{ v \in L_0^2(\T^{\nu + 3}, \R^3): \ \  
v = { \Pi}_n v  \big\},
$$
and $\Pi_n^\bot := {\rm Id} - \Pi_n$. 
The projectors $ \Pi_n $, $ \Pi_n^\bot$ satisfy the 
usual smoothing properties in Lemma \ref{lemma:smoothing}, namely 
\begin{equation}\label{smoothing-u1}
\|\Pi_{n} v \|_{s+b}^{k_0, \gamma} 
\leq N_{n}^{b} \| v \|_{s}^{k_0, \gamma}, 
\quad \ 
\|\Pi_{n}^\bot v \|_{s}^{k_0, \gamma} 
\leq N_n^{- b} \| v \|_{s + b}^{k_0, \gamma}, 
\quad \ 
s,b \geq 0.
\end{equation}
We define the constants 
\begin{equation}\label{costanti nash moser}
\kappa := 6 \overline \mu + 7, 
\quad \ 
\mathtt a_1 := \max \big\{ 6 \overline \mu + 13 \,,\, 
3 \overline \mu + 2 \tau_0 + 2, \, 
3(\overline \mu + \tau + 2 \tau^2) + 1 \big\}, 
\quad \ 
{\mathtt b}_1 := 
2 \overline \mu + \kappa + \mathtt a_1 + 3
\end{equation}
where $\overline \mu := \overline \mu(k_0, \tau, \nu) > 0$ 
is given in Proposition \ref{inversione linearized}
and $\tau_0$ is defined in \eqref{definizione alpha beta}. 
For $N_0 > 0$ we define $N_n := N_0^{\chi^n}$, 
$\chi := 3/2$, and $N_{-1} := 1$. We also fix the regularity of the forcing term $f$ in \eqref{Eulero3} as  
\begin{equation}\label{reg forzante f}
f \in {\cal C}^q(\T^\nu \times \T^3, \R^3), \quad q := s_0 + \mathtt b_1 + 1\,. 
\end{equation}
\begin{remark}[\bf 	Choice of the constants]
The conditions $\kappa \geq 6 \overline \mu + 7$, 
$\mathtt a_1 \geq 6 \overline \mu + 13$, 
$\mathtt b_1 \geq 2 \overline \mu + \kappa + \mathtt a_1 + 3$ 
in \eqref{costanti nash moser} 
allow to prove the induction estimates in $({\cal P}2)_n, ({\cal P}3)_n$ 
in Proposition \ref{iterazione-non-lineare}.
The extra conditions 
$\mathtt a_1 \geq \max \{ 3(\overline \mu + \tau + 2 \tau^2) + 1, 
3 \overline \mu + 2 \tau_0 + 2 \}$ 
are used in Section \ref{sezione stime di misura} for the measure estimates 
(Lemma \ref{inclusione risonanti v n n - 1}). 
\end{remark}

\begin{proposition}\label{iterazione-non-lineare} 
{\bf (Nash-Moser)} 
Let $\tau > 0$, and let $\Om$ be a bounded open subset of $\R^{\nu+3}$. 
There exist $ \d \in (0, 1)$, $C_* > 0$, $\overline \tau(k_0, \tau, \nu) > 0$ such that if
\begin{equation}\label{nash moser smallness condition}  
N_0^{\overline \tau} \e  \leq \d, \quad N_0 := \gamma^{- 1}
\end{equation}
then the following properties hold for all $n \geq 0$. 
\begin{itemize}
\item[$({\mathcal P}1)_{n}$] 
There exists $y_n : \R^{\nu+3} \to H^{s_0 + \mathtt b_1} \cap X$, 
with $y_0 := \mF(v_0)$, satisfying 
\begin{equation} \label{2303.1}
\| y_n \|_{s_{0}}^{k_0, \gamma} 
\leq C_* \e N_{n - 1 }^{- \mathtt a_1}\,.
\end{equation}
There exists 
$v_n : \R^{\nu + 3} \to {\cal H}_{n - 1} \cap Y$, 
with $v_0 := 0$, 
satisfying
\begin{equation} \label{ansatz induttivi nell'iterazione}
\| v_{n} \|_{s_0 + \overline \mu}^{k_0, \gamma} \leq 1.
\end{equation}
If $n \geq 1$, 
the difference $ h_n := v_n -  v_{n - 1}$ satisfies 
$\| h_1 \|_{s_0 + \overline \mu}^{k_0, \gamma} 
\lesssim \e \gamma^{-1}$ and 
\begin{equation} \label{hn}
\| h_n \|_{s_0 + \overline \mu}^{k_0, \gamma} 
\lesssim  N_{n-1}^{2 \overline \mu} N_{n-2}^{-\mathtt a_1} \e \gamma^{-1}. 
\end{equation}

\item[$({\mathcal P}2)_{n}$] 
If $n=0$ we define $\mG_0 := \Om$. 
If $n \geq 1$, we define 
\begin{equation}\label{G-n+1}
{\mathcal G}_{n + 1} := {\cal G}_n \cap 
\big( \Omega_\infty^{\gamma_n}(v_n) 
\cap \Lambda_\infty^{\gamma_n}(v_n) \big), 
\end{equation} 
where $ \gamma_{n}:=\gamma (1 + 2^{-n}) $ 
and the sets $\Omega_\infty^{\gamma_n}(v_n)$, $\Lambda_\infty^{\gamma_n}(v_n)$ 
are defined in \eqref{cantor finale ridu}, \eqref{prime Melnikov}. 
The function $\mF(v_n)$ is defined for all $\lm \in \R^{\nu+3}$
and, for every $\lm \in {\cal G}_n$, it coincides with $y_n$. 

\item[$({\mathcal P}3)_{n}$] 
One has $q_n := \| y_n \|_{s_0 + \mathtt b_1}^{k_0,\g} 
+ \e (\| v_n \|_{s_{0}+ \mathtt b_1}^{k_0, \gamma} + 1 )
\leq C_* \e N_{n}^{\kappa}$.  
\end{itemize}
\end{proposition}

\begin{proof}
To use 
the norms $\| \ \|_s^{k_0,\g}$ 
(where the sup is over $\lm = (\om,\zeta) \in \R^{\nu+3}$, 
see Definition \ref{def:Lip F uniform})
and to employ the assumption 
that the set $\Om \subset \R^{\nu+3}$ is bounded, 
we fix a $C^\infty$ function $\rho : \R^{\nu+3} \to \R$ with compact support 
such that $\rho(\lm) = 1$ for all $\lm = (\om,\zeta)$ 
in a neighborhood of the closure of $\Om$.
Thus 
\begin{equation} \label{2503.1}
\| \rho v \|_s^{k_0,\g} 
\lesssim \| v \|_s^{k_0,\g}, \quad \  
\| \rho(\lm) (\ompaph v + \zeta \cdot \grad v) \|_s^{k_0,\g} 
\lesssim \| v \|_{s+1}^{k_0,\g}
\end{equation}
for all $s \geq 0$, all $v = v(\ph,x;\lm)$. 


%
%

\smallskip

\noindent
\emph{Proof of} $({\mathcal P}1, 2, 3)_0$.
By \eqref{equazione cal F vorticita}, 
$y_0 = \mF(0) = \e F(\ph,x)$, 
$\| {\cal F} (0 ) \|_s^{k_0,\g} = \e \| F \|_s \stackrel{\eqref{equazione vorticita}}{\lesssim} \e \| f \|_{s + 1}$, 
then take $C_* \geq 1 + \| f \|_{s_0 + \mathtt b_1 + 1} $ (in \eqref{reg forzante f}, we have fixed $q = s_0 + \mathtt b_1 + 1$).

\smallskip

\noindent
\emph{Assume that $({\mathcal P}1,2,3)_n$ hold for some $n \geq 0$, 
and prove $({\mathcal P}1,2,3)_{n+1}$.}
By $({\mathcal P}1)_n$, one has $\| v_n\|_{s_0 + \bar \mu}^{k_0,\g} \leq 1$, 
and the assumption 
\eqref{nash moser smallness condition} implies 
the smallness condition 
$\e \gamma^{-1} \leq \delta$ 
of Proposition \ref{inversione linearized}
by taking $\overline \tau (k_0, \tau, \nu)$ large enough and $S = s_0 + \mathtt b_1$. 
Then Proposition \ref{inversione linearized} applies 
to the linearized operator 
\begin{equation}\label{definizione cal Ln}
{\mathcal L}_n \equiv {\mathcal L}(v_n) : = d {\mathcal F}(v_n).
\end{equation}
This implies that there exists a linear operator $\mL_n^{-1}$, 
defined for any $\lm \in \R^{\nu+3}$,  
which satisfies the tame estimate 
\begin{equation}\label{stima Tn}
\| {\mathcal L}_n^{-1} h \|_s^{k_0,\g} 
\lesssim_s \gamma^{-1} \big( \| h \|_{s + \overline \mu}^{k_0,\g} 
+ \| v_n \|_{s + \overline \mu}^{k_0,\g} \,  
\| h \|_{s_0 + \overline \mu}^{k_0, \g} \big), 
\quad \ s_0 \leq s \leq s_0 + \mathtt b_1
\end{equation}
such that for all $\lambda \in {\mathcal G}_{n + 1} 
= {\cal G}_n \cap \Omega_\infty^{\gamma_n}(v_n) \cap \Lambda_\infty^{\gamma_n}(v_n)$ 
one has $\mL_n \mL_n^{-1} = \Id$
(using the bound $\gamma_n = \gamma(1 + 2^{- n}) \in [\g, 2\g]$). 
Specializing \eqref{stima Tn} for $s= s_0$, using \eqref{ansatz induttivi nell'iterazione}, 
one has 
\begin{equation}\label{stima Tn norma bassa} 
\| {\mathcal L}_n^{- 1} h \|_{s_0}^{k_0,\g} 
\lesssim \gamma^{- 1} \| h \|_{s_0 + \overline \mu}^{k_0,\g}\,. 
\end{equation}
We define the successive approximation 
\begin{equation}\label{soluzioni approssimate}
v_{n + 1} := v_n + h_{n + 1}, \quad 
{h}_{n + 1} :=  - \Pi_n {\mathcal L}_n^{-1} \Pi_n y_n 
\in {\mathcal H}_{n},
\end{equation}
and Taylor's remainder  
$Q_n := \mF(v_{n+1}) - \mF(v_n) - \mL_n h_{n+1}$. 
By the definition of $h_{n+1}$ in \eqref{soluzioni approssimate}, 
splitting 
$\mF(v_n) = \Pi_n \mF(v_n) + \Pi_n^\bot \mF(v_n)$ 
and $\mL_n \Pi_n = \mL_n - \mL_n \Pi_n^\bot$,
and multiplying each term by the factor $\rho(\lm)$,
we calculate 
\begin{equation}\label{espansione F u n + 1}
\begin{aligned}
\rho {\cal F}(v_{n + 1}) 
& = \rho ({\cal F}(v_n) + {\cal L}_n h_{n + 1} + Q_n)
= y_{n+1} + z_{n+1}
\end{aligned}
\end{equation}
where
\begin{equation} \label{2403.1}
\begin{aligned}
y_{n+1} & := \rho \Pi_n^\bot {\cal F}(v_n) 
+ \rho {\cal L}_n \Pi_n^\bot {\cal L}_n^{-1} \Pi_n y_n
+ \rho Q_n,
\\
z_{n+1} & := \rho \Pi_n [\mF(v_n) - y_n] 
+ \rho (\Id - \mL_n \mL_n^{-1}) \Pi_n y_n.
\end{aligned}
\end{equation}
Note that $h_{n+1}, v_{n + 1}, y_{n+1}, z_{n+1}$ are defined for all $\lm \in \R^{\nu + 3}$, 
and $k_0$ times differentiable in $\lm \in \R^{\nu+3}$. 

Let us estimate $y_{n+1}$. 
Since $v_n \in \mH_{n-1}$, one has 
$\Pi_n^\bot (\om \cdot \pa_\ph v_n + \zeta \cdot \grad v_n) = 0$.  
Thus, by the definition of ${\cal F}$ in \eqref{equazione cal F vorticita}, 
the product estimate \eqref{p1-pr}, 
the smoothing property \eqref{smoothing-u1} 
(since $v_n \in {\cal H}_{n - 1}$),
the induction estimate \eqref{ansatz induttivi nell'iterazione} 
(since $\| v_n \|_{s_0 + 1}^{k_0,\g} 
\leq \| v_n \|_{s_0 + \overline \mu}^{k_0,\g} \leq 1$), 
and the assumption  
$\| F \|_{s_0 + \mathtt b_1} \lesssim 1$, 
one has 
\begin{equation} \label{2403.2}
\begin{aligned}
\| \Pi_n^\bot \mF(v_n) \|_{s_0 + \mathtt b_1}^{k_0,\g} 
& \lesssim \e (\| v_n \|_{s_0 + \mathtt b_1 + 1}^{k_0,\g} 
\| v_n \|_{s_0 + 1}^{k_0,\g} 
+ \| F \|_{s_0 + \mathtt b_1})
\lesssim \e N_{n-1} (\| v_n \|_{s_0 + \mathtt b_1}^{k_0,\g} + 1),
\\
\| \Pi_n^\bot \mF(v_n) \|_{s_0}^{k_0,\g} 
& \lesssim \e N_n^{-(\mathtt b_1 - 1)} 
\| v_n \|_{s_0 + \mathtt b_1}^{k_0,\g} \| v_n \|_{s_0+1}^{k_0,\g} 
+ \e N_n^{- \mathtt b_1} \| F \|_{s_0 + \mathtt b_1}
\lesssim \e N_n^{1 - \mathtt b_1} 
(\| v_n \|_{s_0 + \mathtt b_1}^{k_0,\g} + 1).
\end{aligned}
\end{equation}
By \eqref{operatore linearizzato}, 
\eqref{2503.1},
\eqref{stime tame operatore cal L0}, 
and using the bound $\| v_n \|_{s_0 + 1}^{k_0,\g} 
\leq \| v_n \|_{s_0 + \overline \mu}^{k_0,\g} \leq 1$, 
for any $h$ one has 
\begin{equation} \label{2503.2}
\| \rho {\cal L}_n h \|_{s_0}^{k_0,\g} 
\lesssim  \| h \|_{s_0 + 1}^{k_0,\g},
\quad \ 
\| \rho {\cal L}_n h \|_{s_0 + \mathtt b_1}^{k_0,\g} 
\lesssim  \| h \|_{s_0 + \mathtt b_1 + 1}^{k_0,\g} 
+ \e \| v_n \|_{s_0 + \mathtt b_1 + 1}^{k_0,\g} \| h \|_{s_0 + 1}^{k_0,\g}.
\end{equation}
By \eqref{2503.2}, using 
\eqref{smoothing-u1}, 
\eqref{stima Tn}, 
$\gamma^{- 1} = N_0 \leq N_n$, 
and \eqref{ansatz induttivi nell'iterazione},
we estimate
\begin{equation}\label{stima modi alti 1}
\begin{aligned}
\| \rho {\cal L}_n \Pi_n^\bot {\cal L}_n^{-1}  \Pi_n y_n \|_{s_0}^{k_0,\g}  
& \lesssim \| \Pi_n^\bot {\cal L}_n^{-1} \Pi_n y_n \|_{s_0 + 1}^{k_0,\g}  
\lesssim N_n^{-(\mathtt b_1 - 1)} 
\| {\cal L}_n^{-1} \Pi_n y_n \|_{s_0 + \mathtt b_1}^{k_0,\g}  
\\ & 
\lesssim N_n^{2 + 2 \overline \mu - \mathtt b_1} 
(\| y_n \|_{s_0 + \mathtt b_1}^{k_0,\g} 
+ \| v_n \|_{s_0 + \mathtt b_1}^{k_0,\g} \| y_n \|_{s_0}^{k_0,\g}),  
\\ 
\| \rho {\cal L}_n \Pi_n^\bot {\cal L}_n^{-1}  \Pi_n y_n \|_{s_0 + \mathtt b_1}^{k_0,\g}  
& \lesssim 
\| {\cal L}_n^{-1} \Pi_n y_n \|_{s_0 + \mathtt b_1 + 1}^{k_0,\g}  
+ \| v_n \|_{s_0 + \mathtt b_1 + 1}^{k_0,\g}  
\| {\cal L}_n^{-1} \Pi_n y_n \|_{s_0 + 1}^{k_0,\g}  
\\ & 
\lesssim N_n^{2 + 2 \overline \mu} 
(\| y_n \|_{s_0 + \mathtt b_1}^{k_0,\g}  
+ \| v_n \|_{s_0 + \mathtt b_1}^{k_0,\g} \| y_n \|_{s_0}^{k_0,\g}).  
\end{aligned}
\end{equation}
Since the non linear part of ${\cal F}$ is quadratic (see \eqref{equazione cal F vorticita}), 
$Q_n$ is a quadratic operator of $h_{n + 1}$, independent of $v_n$, with
\begin{equation} \label{2503.4}
\| Q_n \|_{s}^{k_0,\g} 
\lesssim \e \| h_{n+1} \|_{s+1}^{k_0,\g} \| h_{n+1} \|_{s_0}^{k_0,\g},
\quad \ s \geq s_0.
\end{equation}
To estimate $h_{n+1}$ we use 
\eqref{soluzioni approssimate},
\eqref{stima Tn},
\eqref{smoothing-u1},
\eqref{ansatz induttivi nell'iterazione},
and obtain
\begin{equation} \label{2503.3}
\| h_{n+1} \|_{s_0}^{k_0,\g} 
\lesssim N_n^{1 + \overline \mu} \| y_n \|_{s_0}^{k_0,\g},
\quad \ 
\| h_{n+1} \|_{s_0 + \mathtt b_1}^{k_0,\g} 
\lesssim N_n^{1 + 2 \overline \mu}
(\| y_n \|_{s_0 + \mathtt b_1}^{k_0,\g}
+ \| v_n \|_{s_0 + \mathtt b_1}^{k_0,\g}
\| y_n \|_{s_0}^{k_0,\g}).
%
\end{equation}
By \eqref{2503.4}, 
\eqref{2503.3},
\eqref{smoothing-u1}, 
we get 
\begin{equation} \label{2503.5}
\begin{aligned}
\| Q_n \|_{s_0}^{k_0,\g} 
& \lesssim \e N_n^{3+2\overline \mu} (\| y_n \|_{s_0}^{k_0,\g})^2,
\quad 
\| Q_n \|_{s_0 + \mathtt b_1}^{k_0,\g} 
\lesssim 
\e N_n^{3+3\overline \mu} 
(\| y_n \|_{s_0 + \mathtt b_1}^{k_0,\g}
+ \| v_n \|_{s_0 + \mathtt b_1}^{k_0,\g}
\| y_n \|_{s_0}^{k_0,\g}) 
\| y_n \|_{s_0}^{k_0,\g}.
\end{aligned}
\end{equation}
By the definition \eqref{2403.1} of $y_{n+1}$, 
the first property in \eqref{2503.1},
the estimates
\eqref{2403.2},
\eqref{stima modi alti 1},
\eqref{2503.5}, 
and the bound $\| y_n \|_{s_0}^{k_0,\g} \lesssim \e$ 
(which follows from \eqref{ansatz induttivi nell'iterazione}), 
we obtain
\begin{equation} \label{2503.6}
\begin{aligned}
\| y_{n+1} \|_{s_0 + \mathtt b_1}^{k_0,\g} 
\lesssim N_n^{3 + 3 \overline \mu} q_n,
\quad \ 
\| y_{n+1} \|_{s_0}^{k_0,\g} 
\lesssim N_n^{2 + 2 \overline \mu - \mathtt b_1} q_n 
+ \e N_n^{3 + 2 \overline \mu} ( \| y_n \|_{s_0}^{k_0,\g} )^2,
\end{aligned}
\end{equation}
where $q_n$ is defined in $({\mathcal P}3)_{n}$. 

\smallskip

\noindent
\textsc{Proof of $(\mP 3)_{n+1}$}.
By the inductive assumption $({\mathcal P}3)_{n}$, 
$\e \| v_n \|_{s_0 + \mathtt b_1}^{k_0,\g} \leq q_n$. 
Hence, 
by 
\eqref{soluzioni approssimate},
\eqref{2503.3}, 
and $\e \g^{-1} \leq 1$,
one has 
\begin{equation} \label{2503.7}
\e \| v_{n+1} \|_{s_0 + \mathtt b_1}^{k_0,\g} 
\lesssim \e \| v_n \|_{s_0 + \mathtt b_1}^{k_0,\g} 
+ \e \| h_{n+1} \|_{s_0 + \mathtt b_1}^{k_0,\g} 
\lesssim N_n^{1 + 2 \overline \mu} q_n.
\end{equation}
Thus \eqref{2503.6}, \eqref{2503.7} give 
\begin{equation} \label{2503.8}
q_{n+1} \leq C N_n^{3 + 3 \overline \mu} q_n
\end{equation}
for some constant $C$. 
The estimate \eqref{2503.8} 
and $({\mathcal P}3)_{n}$
imply $({\mathcal P}3)_{n+1}$ 
for $\kappa > 6 + 6 \overline \mu$ 
and $N_0$ sufficiently large 
(depending on $\kappa, \overline \mu$ and $C$ in \eqref{2503.8}). 

\smallskip

\noindent
\textsc{Proof of $(\mP 1)_{n+1}$}.
The estimate \eqref{2503.6} 
and the inductive assumption \eqref{2303.1} 
imply \eqref{2303.1} at the step $n+1$ 
for $\mathtt b_1 > 2 + 2 \overline \mu + \kappa + \mathtt a_1$, 
$\mathtt a_1 \geq 9 + 6 \overline \mu$, 
$N_0$ sufficiently large 
(depending on $\mathtt a_1, \mathtt b_1, \overline \mu, \kappa$ 
and on the implicit constant in \eqref{2503.6})  
and 
$\e$ sufficiently small (depending also on $N_0$). 
By \eqref{2503.3}, \eqref{smoothing-u1}, $\eqref{2303.1}$
we deduce \eqref{hn} at the step $n+1$  
and, by telescoping series, 
\eqref{ansatz induttivi nell'iterazione} at the step $n+1$.
The estimate for $h_1$ follows from \eqref{soluzioni approssimate}
using that $v_0 = 0$, $y_0 = \mF(v_0) = \e F(\ph,x)$, 
and $\| F \|_{s_0 + 2 \overline \mu} \leq \| F \|_{s_0 + \mathtt b_1} \lesssim 1$.

\smallskip

\noindent
\textsc{Proof of $(\mP 2)_{n+1}$}. 
For $\lm \in \mG_{n+1}$ one has $\rho(\lm) = 1$, 
$\mL_n \mL_n^{-1} = \Id$ 
and, by inductive assumption, $\mF(v_n) = y_n$. 
Hence $z_{n+1}$ in \eqref{2403.1} is zero, 
and, by \eqref{espansione F u n + 1}, 
$\mF(v_{n+1}) = y_{n+1}$. 
\end{proof}

\section{Measure estimate}\label{sezione stime di misura}
In this section we prove that the set
\begin{equation}\label{def cal G infty}
{\cal G}_\infty := \cap_{n \geq 0} {\cal G}_n
\end{equation}
has large 
Lebesgue measure. 
We estimate the measure of its complement 
$\Omega \setminus {\cal G}_\infty$. 
The main result of this section is the following proposition.

\begin{proposition} \label{prop measure estimate finale}
Let 
\begin{equation}\label{definizione finale tau}
\tau := 9\,{\rm max}\{ \nu, 3 \} + 1  \,, \quad k_0 := 11\,.
\end{equation}
Then $|\Omega \setminus {\cal G}_\infty| \lesssim \gamma$. 
\end{proposition}

The rest of this section is devoted to the proof 
of Proposition \ref{prop measure estimate finale}. 
By the definition \eqref{def cal G infty}, one has 
\begin{equation}\label{prima inclusione stime misura}
\Omega \setminus {\cal G}_\infty 
\subseteq \cup_{n \geq 0} ({\cal G}_n \setminus {\cal G}_{n + 1}),
\end{equation}
hence it is enough to estimate the measure of ${\cal G}_n \setminus {\cal G}_{n + 1}$. By \eqref{G-n+1} and using elementary properties of set theory, one has that 
\begin{equation}\label{seconda inclusione stime misura}
{\cal G}_n \setminus {\cal G}_{n + 1} \subseteq ({\cal G}_n \setminus  \Omega_\infty^{\gamma_n}(v_n)) \cup ({\cal G}_n \setminus  \Lambda_\infty^{\gamma_n}(v_n))\,.
\end{equation}
Note that $\mG_{n+1} \subseteq DC(\g_n,\tau)$ for any $n \geq 0$
because 
$\mG_{n+1} \subseteq \Omega_\infty^{\g_n}(v_n) \subseteq DC(\g_n,\tau)$
by \eqref{G-n+1}, \eqref{cantor finale ridu}.


\begin{proposition}\label{Gn Omegan Lambdan}
For any $n \geq 0$, the following estimates hold. 

\noindent
$(i)$ ${\cal G}_0 \setminus \Omega_\infty^{\gamma_0}(v_0) \lesssim \gamma$ 
and, for any $n \geq 1$, 
$|{\cal G}_n \setminus  \Omega_\infty^{\gamma_n}(v_n)| 
\lesssim \gamma N_{n}^{- \frac{1}{9}}$. 

\noindent
$(ii)$ ${\cal G}_0 \setminus \Lambda_\infty^{\gamma_0}(v_0) \lesssim \gamma$ and for any $n \geq 1$,$|{\cal G}_n \setminus  \Lambda_\infty^{\gamma_n}(v_n)| \lesssim \gamma N_{n}^{- \frac{1}{9}}$. 

\noindent
As a consequence $|{\cal G}_0 \setminus {\cal G}_1| \lesssim \gamma$ and for any $n \geq 1$, $|{\cal G}_n \setminus {\cal G}_{n + 1}| \lesssim \gamma N_{n }^{- \frac{1}{9}}$.
\end{proposition}

Propositions \ref{prop measure estimate finale} 
and \ref{Gn Omegan Lambdan}
are proved at the end of this section. 
Now 
we estimate the measure of the set ${\cal G}_n \setminus  \Omega_\infty^{\gamma_n}(v_n)$. 
The estimate of the measure of ${\cal G}_n \setminus  \Lambda_\infty^{\gamma_n}(v_n)$ can be done arguing similarly (it is actually even easier). 
%
By the definitions \eqref{G-n+1}, \eqref{cantor finale ridu}, one gets that 
\begin{equation}\label{quarta inclusione stime misura}
{\cal G}_n  \setminus \Omega_\infty^{\gamma_n}(v_n) \subseteq \bigcup_{(\ell, j, j') \in {\cal I}}{\cal R}_{\ell j j'}(v_n)
\end{equation}
where 
\begin{equation}\label{def cal I res}
{\cal I} := \big\{ (\ell, j, j') \in \Z^\nu \times (\Z^3 \setminus \{ 0 \}) \times (\Z^3 \setminus \{ 0 \}) : (\ell, j, j') \neq (0, j, j) \big\}
\end{equation}
and 
\begin{equation}\label{def cal R vn l j j'}
\begin{aligned}
{\cal R}_{\ell j j'}(v_n) 
:= \Big\{ \lambda = (\omega, \zeta) \in {\cal G}_n : 
\ & L_\infty(\ell, j, j'; \lambda, v_n(\lambda)) 
\text{ is not invertible, or it is invertible and } \\
& \| L_\infty(\ell, j, j'; \lambda, v_n(\lambda ))^{- 1} \|_{\rm op} > \frac{\langle \ell \rangle^\tau |j|^\tau |j'|^\tau}{2 \gamma_n}   \Big\}\,.
\end{aligned}
\end{equation}
In the next lemma, we estimate the measure of the resonant sets ${\cal R}_{\ell j j'}(v_n)$.

\begin{lemma}\label{stima misura risonanti sec melnikov}
For any $n \geq 0$, one has that 
$|{\cal R}_{\ell j j'}(v_n)| 
\lesssim \gamma^{\frac{1}{9}} \langle \ell, j - j' \rangle^{-\frac{1}{9}} 
\langle \ell \rangle^{-\frac{\tau}{9}} 
|j|^{-\frac{\tau}{9}} 
|j'|^{-\frac{\tau}{9}} $. 
\end{lemma}

\begin{proof}
Recalling \eqref{def cal N infty nel lemma}, \eqref{L infty l j j fin},  
for $(\ell, j, j') \in {\cal I}$ (see \eqref{def cal I res}), 
we write $L_\infty(\ell, j, j'; \lambda) \equiv L_\infty(\ell, j, j'; \lambda,  v_n(\lambda))$, $\lambda = (\omega, \zeta) \in \Omega$ as 
$$
L_\infty(\ell, j, j'; \lambda) = \lambda \cdot k \,{\rm Id} + Q(\lambda) 
: \Mat_{3 \times 3}(\C) \to \Mat_{3 \times 3}(\C)
$$
where 
\begin{equation}\label{tartufo 0}
\begin{aligned}
k := (\ell , j - j')\,, \quad    \lambda = (\omega, \zeta) \in \Omega\,, \quad Q(\lambda)  := M_L\Big( ({\cal Q}_\infty(\lambda, v_n(\lambda)))_j^j \Big) - M_R\Big( ({\cal Q}_\infty(\lambda, v_n(\lambda)))_{j'}^{j'} \Big).
\end{aligned}
\end{equation}
Since $(\ell, j, j') \in {\cal I}$, one has $k = (\ell, j - j') \neq (0,0)$. 
Recalling the definition \eqref{def cal R vn l j j'}, 
the resonant set ${\cal R}_{\ell j j'}(v_n)$ is equal 
to the set ${\cal R}_A$ defined in \eqref{risonante matrice A}, 
with $d : = 9$, 
$A := L_\infty(\ell, j, j')$, 
$\eta = \frac{\langle \ell \rangle^\tau |j|^\tau |j'|^\tau}{2 \gamma_n}$, 
and $k_0 := d + 2 = 11$. 
By the definitions \eqref{tartufo 0} and by the estimates \eqref{MLAMRA elementry}, \eqref{stime forma normale limite}, \eqref{ansatz induttivi nell'iterazione}, 
one obtains that $\| Q \|_\op^{k_0, \gamma} \lesssim \e$,
hence, by the smallness condition \eqref{nash moser smallness condition} 
(by choosing $\overline \tau$ large enough) 
one can apply Lemma \ref{lemma astratto misura risonante} 
obtaining that 
$|{\cal R}_{\ell j j'}(v_n)| 
\lesssim \gamma_n^{\frac19} \langle \ell, j - j' \rangle^{-\frac{1}{9}} 
\langle \ell \rangle^{-\frac{\tau}{9}} 
|j|^{-\frac{\tau}{9}} 
|j'|^{-\frac{\tau}{9}}$,
then $\g_n \lesssim \g$, 
and we get the claimed estimate.
\end{proof}


\begin{lemma}\label{lemma triviale modi alti}
Let $(\ell, j, j') \in {\cal I}$, $\langle \ell, j - j' \rangle \leq N_n$, ${\rm min}\{ |j|, |j'| \} \geq N_n^\tau$. Then ${\cal R}_{\ell j j'}(v_n) = \emptyset$.
\end{lemma}
\begin{proof}
Let $\lambda = (\omega, \zeta) \in {\cal G}_n $. Recalling \eqref{def cal N infty nel lemma}, \eqref{L infty l j j fin}, one writes 
$$
\begin{aligned}
L_\infty(\ell, j, j'; v_n) 
& = \Omega(\ell, j, j') \big( {\rm Id} + \Delta_n(\ell, j, j') \big), 
\quad \ 
\Omega(\ell, j, j') := \big( \omega \cdot \ell + \zeta \cdot (j - j') \big) {\rm Id} \,, 
\\
\Delta_n(\ell, j, j') & := \Omega(\ell, j, j')^{- 1} \Big( M_L\big({\cal Q}_\infty(v_n)_j^j \big) - M_R\big({\cal Q}_\infty(v_n)_{j'}^{j'} \big)\Big)\,. 
\end{aligned}
$$
Note that since $\lambda \in DC(\gamma_n, \tau)$ (see \eqref{def cantor set trasporto}) one has that $\Omega(\ell, j, j')$ is invertible and 
\begin{equation}\label{pizza rustica 0}
\| \Omega(\ell, j,j')^{- 1} \|_\op 
\leq \frac{\langle \ell, j - j' \rangle^\tau}{{\frak C}_0  \gamma_n} {\leq} \frac{N_n^\tau}{{\frak C}_0  \gamma_n}\,.
\end{equation}
The latter estimate, 
together with \eqref{MLAMRA elementry}, \eqref{stime forma normale limite}, 
\eqref{ansatz induttivi nell'iterazione}, 
since $\min \{ |j|, |j'| \} \geq N_n^\tau$ and 
$\gamma_n = \gamma(1 + 2^{- n})$, 
implies that 
\begin{equation}\label{pizza rustica 1}
\begin{aligned}
\| \Delta_n(\ell , j, j') \|_\op & \leq C \e \gamma^{- 1}
\end{aligned}
\end{equation}
for some constant $C > 0$. 
Then for $\e \gamma^{- 1}$ small enough, by Neumann series, 
$L_\infty(\ell, j, j'; v_n)$ is invertible and 
$$
\| L_\infty(\ell, j, j'; v_n)^{- 1} \|_\op 
\leq \frac{2 \langle \ell, j - j' \rangle^\tau}{{\frak C}_0  \gamma_n} \leq \frac{2 C(\tau)\langle \ell \rangle^\tau |j|^\tau |j'|^\tau }{{\frak C}_0  \gamma_n} \leq \frac{\langle \ell \rangle^\tau |j|^\tau |j'|^\tau }{2  \gamma_n}$$
by choosing $\frak C_0 = \frak C_0(\tau) > 0$ large enough. This implies that ${\cal R}_{\ell j j'}(v_n) = \emptyset$.
\end{proof}

\begin{lemma}\label{inclusione risonanti v n n - 1}
Let $(\ell, j, j') \in {\cal I}$, 
$\langle \ell, j - j' \rangle \leq N_n$, 
$\min \{ |j|, |j'| \} \leq N_n^\tau$, $n \geq 1$. 
Then 
$$
{\cal R}_{\ell j j'}(v_n) \subseteq {\cal R}_{\ell j j'}(v_{n - 1})\,.
$$
\end{lemma}

\begin{proof}
We split the proof in two steps. 

\noindent
{\sc Step 1.} We show that 
for $(\ell, j, j') \in {\cal I}$, $|\ell|, |j - j'| \leq N_n $, 
for any $\lm \in \mG_n$ one has 
\begin{equation}\label{L infty vn - v n - 1}
\| L_\infty(\ell, j, j'; v_n) - L_\infty(\ell, j, j'; v_{n - 1}) \|_\op 
\lesssim N_{n-1}^{2 \overline \mu} N_{n-2}^{- \mathtt a_1} \, \e^2 \gamma^{- 1} 
+ \e N_{n - 1}^{- \mathtt a} 
\end{equation}
(the constant $\mathtt a$ is defined in \eqref{definizione alpha beta}). 
By \eqref{def cal N infty nel lemma}, \eqref{L infty l j j fin}, 
by the property 
\eqref{MLAMRA elementry}, one computes 
\begin{equation}\label{capitone 0}
\begin{aligned}
\| L_\infty(\ell, j, j'; v_n) - L_\infty(\ell, j, j'; v_{n - 1}) \|_\op 
& \lesssim  \sup_{j \in \Z^3 \setminus \{ 0 \}} 
\big\| \big( {\cal Q}_\infty(v_n) - {\cal Q}_\infty(v_{n - 1}) \big)_j^j \big\|_{\HS} \,.
\end{aligned}
\end{equation}
By the inductive definition of the sets ${\cal G}_n$ in \eqref{G-n+1}, recalling the definitions \eqref{insiemi di cantor rid}, \eqref{cantor finale ridu} and Lemma \ref{prima inclusione cantor}, one gets the inclusion 
\begin{equation}\label{inclusione cantor 100}
\begin{aligned}
{\cal G}_n  \subseteq \Omega_n^{\gamma_{n - 1}}(v_{n - 1}).
\end{aligned}
\end{equation}
We apply Proposition \ref{prop riducibilita}-${\bf (S3)_n}$ with 
$v_1 \equiv v_{n - 1}$, 
$v_2 \equiv v_n$, 
$\rho \equiv \gamma_{n - 1} - \gamma_n = \gamma 2^{-n}$. 
By \eqref{hn}, one has 
\begin{equation}\label{verifica rid S3 n}
\begin{aligned}
 & N_{n-1}^{\tau_0} \e \| v_n - v_{n - 1} \|_{s_0 + \overline \mu} 
\leq C N_{n-1}^{\tau_0 + 2 \overline \mu} N_{n-2}^{- \mathtt a_1} \, \e^2 \gamma^{-1} 
\leq \gamma_{n - 1} - \gamma_n
 \end{aligned}
\end{equation}
for some constant $C > 0$, since
$$
C 2^n N_{n-1}^{\tau_0 + 2 \overline \mu} N_{n-2}^{- \mathtt a_1} \, \e^2 \gamma^{-2} 
\leq 1, \quad n \geq 1
$$
by recalling that $\overline \mu > \mu(\mathtt b)$, 
$\mathtt a_1 > 3 \overline \mu + \frac32 (\tau_0 + 1)$ 
(see \eqref{definizione alpha beta}, Proposition \ref{inversione linearized}, and \eqref{costanti nash moser}) and by the smallness condition \eqref{nash moser smallness condition}, by choosing $\overline \tau$ large enough. The estimate \eqref{verifica rid S3 n} implies that Proposition \ref{prop riducibilita}-${\bf (S3)_n}$ applies, implying that ${\cal G}_n  \subseteq \Omega_n^{\gamma_{n - 1}}(v_{n - 1})  \subseteq \Omega_n^{\gamma_{n}}(v_{n })$ (recall also \eqref{inclusione cantor 100}). Then, we can apply the estimate \eqref{stime Delta 12 cal Nn} of Proposition \ref{prop riducibilita}-${\bf (S2)_n}$ (with $v_1 \equiv v_{n - 1}$, $v_2 \equiv v_n$, $\gamma_1 \equiv \gamma_n$, $\gamma_2 \equiv \gamma_{n - 1}$), implying that 
\begin{equation}\label{capitone 100}
\sup_{j \in \Z^3 \setminus \{ 0 \}} \big\| \big({\cal Q}_n(v_n) - {\cal Q}_n(v_{n - 1})\big)_j^j \big\|_{\HS} 
\lesssim \e \| v_n - v_{n - 1} \|_{s_0 + \overline \mu} 
\stackrel{\eqref{hn}}{\lesssim} 
N_{n-1}^{2 \overline \mu} N_{n-2}^{-\mathtt a_1} \e^2 \gamma^{- 1}, \quad 
\lambda \in {\cal G}_n. 
\end{equation}
Hence, by triangular inequality
and by \eqref{stime forma normale limite}, \eqref{capitone 100}, 
\eqref{ansatz induttivi nell'iterazione}, one has 
\begin{equation}\label{capitone 1}
\begin{aligned}
\sup_{j \in \Z^3 \setminus \{ 0 \}} 
\big\| \big({\cal Q}_\infty(v_n) - {\cal Q}_\infty(v_{n - 1})\big)_j^j \big\|_{\HS} 
& \leq \sup_{j \in \Z^3 \setminus \{ 0 \}} 
\big\| \big({\cal Q}_\infty(v_n) - {\cal Q}_n(v_{n })\big)_j^j \big\|_{\HS} \\
& \quad \ \ 
+ \sup_{j \in \Z^3 \setminus \{ 0 \}} 
\big\| \big({\cal Q}_n(v_n) - {\cal Q}_n(v_{n - 1})\big)_j^j \big\|_{\HS}  \\
& \quad \ \ 
+ \sup_{j \in \Z^3 \setminus \{ 0 \}} 
\big\| \big({\cal Q}_n(v_{n - 1}) - {\cal Q}_\infty(v_{n - 1})\big)_j^j \big\|_{\HS} \\
& 
\lesssim N_{n-1}^{2 \overline \mu} N_{n-2}^{-\mathtt a_1} \e^2 \gamma^{-1} 
+ \e N_{n - 1}^{- \mathtt a}\,.
\end{aligned}
\end{equation}
The estimate \eqref{L infty vn - v n - 1} then follows 
by \eqref{capitone 0}, \eqref{capitone 1}.  
 
\smallskip
\noindent
{\sc Step 2.} 
Let $\lambda \in {\cal G}_n$ 
and $\langle \ell,  j - j' \rangle\leq N_n$, ${\rm min}\{ |j|, |j'| \} \leq N_n^\tau$,
$n \geq 1$. We write 
\begin{equation}\label{pizza rustica 9}
\begin{aligned}
L_\infty(\ell, j, j'; v_n) 
& =  L_\infty(\ell, j, j'; v_{n - 1}) \big( {\rm Id} + \Delta_n(\ell, j, j') \big), \\
\Delta_n(\ell, j, j') & := L_\infty(\ell, j, j'; v_{n - 1})^{- 1} 
\Big( L_\infty(\ell, j, j'; v_n) - L_\infty(\ell, j, j'; v_{n - 1}) \Big).
\end{aligned}
\end{equation}
Since $|j - j'| \leq N_n$ and $\min \{ |j|, |j'| \} \leq N_n^\tau$, 
by triangular inequality one has $\max\{ |j|, |j'| \} \lesssim N_n^\tau$. 
Therefore, using that $\lambda \in {\cal G}_n$ 
and the estimate \eqref{L infty vn - v n - 1}, one gets 
\begin{equation}\label{pizza rustica 10}
\begin{aligned}
\| \Delta_n(\ell, j, j') \|_\op 
\lesssim 
N_n^{\tau + 2 \tau^2} N_{n-1}^{2 \overline \mu} N_{n-2}^{-\mathtt a_1} \, \e^2 \gamma^{-2} 
+ N_n^{\tau + 2 \tau^2} N_{n-1}^{- \mathtt a} \, \e \gamma^{- 1}.
\end{aligned}
\end{equation}
Hence, by the smallness condition \eqref{nash moser smallness condition}, 
using that $\mathtt a_1 > 3 \overline \mu + \frac94 (\tau + 2 \tau^2)$ 
(see \eqref{costanti nash moser}), 
$\mathtt a > \frac32(\tau + 2 \tau^2)$ 
(see \eqref{definizione alpha beta}), 
by \eqref{pizza rustica 9}, \eqref{pizza rustica 10}, 
one gets that 
$\| \Delta_n(\ell, j, j') \|_\op \leq 2^{-n}$, 
so that $ L_\infty(\ell, j, j'; v_n) $ is invertible and 
$\|  L_\infty(\ell, j, j'; v_n)^{- 1} \|_\op 
\leq \langle \ell \rangle^\tau |j|^\tau |j'|^\tau \frac{1}{2 \gamma_{n}}$,
implying the claimed inclusion.  
\end{proof}

\begin{lemma}\label{inclusione per sommare serie}
For any $n \geq 1$, the following inclusion holds:
\begin{equation}\label{quinta inclusione stime misura}
{\cal G}_n \setminus \Omega_\infty^{\gamma_n}(v_n) \subseteq \bigcup_{\begin{subarray}{c}(\ell, j, j') \in {\cal I} \\
\langle \ell, j - j' \rangle \geq N_n 
\end{subarray}}{\cal R}_{\ell j j'}(v_n)\,.
\end{equation}
\end{lemma}
\begin{proof}
Let $n \geq 1$. By the definition \eqref{def cal R vn l j j'}, ${\cal R}_{\ell j j'}(v_n) \subseteq {\cal G}_n $. By Lemmata \ref{lemma triviale modi alti}, \ref{inclusione risonanti v n n - 1} for any $(\ell, j, j') \in {\cal I}$, $\langle \ell, j - j' \rangle \leq N_n$, if ${\rm min}\{ |j|, |j'| \} \geq N_n^\tau$ then ${\cal R}_{\ell, j j'}(v_n) = \emptyset$ and if ${\rm min}\{ |j|, |j'| \} \leq N_n^\tau$ then ${\cal R}_{\ell j j'}(v_n) \subseteq {\cal R}_{\ell j j'}(v_{n - 1})$. On the other hand ${\cal G}_n \cap {\cal R}_{\ell j j'}(v_{n - 1}) = \emptyset$ (see \eqref{G-n+1}), therefore \eqref{quinta inclusione stime misura} holds.
\end{proof} 
By the inclusion \eqref{quarta inclusione stime misura} and by Lemma \ref{stima misura risonanti sec melnikov}, one gets that 
\begin{equation}\label{G0 meno Omega (v0)}
\begin{aligned}
\Big| {\cal G}_0  \setminus \Omega_\infty^{\gamma_0}(v_0)\Big| \lesssim \sum_{(\ell, j, j') \in {\cal I}} |{\cal R}_{\ell j j'}(v_0)| \lesssim \gamma \sum_{(\ell, j, j' ) \in \Z^\nu \times (\Z^3 \setminus \{ 0 \}) \times (\Z^3 \setminus \{ 0 \}) }  \frac{1}{ \langle \ell \rangle^{\frac{\tau}{9}} |j|^{\frac{\tau}{9}} |j'|^{\frac{\tau}{9}}}  \stackrel{\eqref{definizione finale tau}}{\lesssim} \gamma.
\end{aligned}
\end{equation}
Now let $n \geq 1$. By the inclusion \eqref{quinta inclusione stime misura} and by Lemma \ref{stima misura risonanti sec melnikov}, one has 
\begin{equation}\label{Gn meno Omega (vn)}
\begin{aligned}
\Big| {\cal G}_n \setminus \Omega_\infty^{\gamma_n}(v_n)\Big|& \lesssim \sum_{\begin{subarray}{c}(\ell, j, j') \in {\cal I} \\
\langle \ell, j - j' \rangle \geq N_n 
\end{subarray}} |{\cal R}_{\ell j j'}(v_n)| 
  \lesssim \gamma \sum_{\begin{subarray}{c}(\ell, j, j') \in {\cal I} \\
\langle \ell, j - j' \rangle \geq N_n 
\end{subarray} }  \frac{1}{ \langle \ell, j - j' \rangle^{\frac{1}{9}} \langle \ell \rangle^{\frac{\tau}{9}} |j|^{\frac{\tau}{9}} |j'|^{\frac{\tau}{9}}}   
\stackrel{\eqref{definizione finale tau}}{\lesssim} \gamma N_n^{- \frac{1}{9}}\,.
\end{aligned}
\end{equation}

\medskip

\begin{proof}
[\textsc{Proof of Propositions \ref{prop measure estimate finale}, \ref{Gn Omegan Lambdan}}]
Proposition \ref{Gn Omegan Lambdan}-$(i)$ follows by recalling 
formula \eqref{quarta inclusione stime misura} 
and by applying the estimates \eqref{G0 meno Omega (v0)}, \eqref{Gn meno Omega (vn)}. 
Item $(ii)$ can be proved by similar arguments. 
The estimates on ${\cal G}_n \setminus {\cal G}_{n + 1}$ follows by recalling 
formula \eqref{seconda inclusione stime misura} and by applying items $(i),(ii)$.

Proposition \ref{prop measure estimate finale} follows by applying 
Proposition \ref{Gn Omegan Lambdan}, using the inclusion \eqref{prima inclusione stime misura} and the fact that the series $\sum_{n \geq 1} N_n^{-1/9}$ 
converges. 
\end{proof}

\section{Proof of Theorems \ref{main theorem 1}, \ref{main theorem 2}}
\label{sezione teoremi principali}

{\sc Proof of Theorem \ref{main theorem 1}.}
Fix $\gamma := \e^c$ with $0 < c < \frac{1}{\overline \tau}$ where $\overline \tau$ is the constant appearing in the smallness condition \eqref{nash moser smallness condition}. Note that, since $k_0$ and $\tau$ have been fixed in \eqref{definizione finale tau}, the constant $\overline \tau = \overline \tau(\nu)$ only depends on the number of frequencies $\nu$. Then $N_0^{\overline \tau}\e = \gamma^{- \overline \tau}\e = \e^{1 - c \overline \tau}$, therefore the smallness condition \eqref{nash moser smallness condition} is fullfilled for $\e$ small enough. 
By Proposition \ref{iterazione-non-lineare}, $({\cal P}1)_n$, using a telescoping argument, 
the sequence $(v_n)_{n \geq 0}$ converges to $v_\infty \in H^{s_0 + \overline \mu} \cap Y$ 
with respect to the norm $\| \  \|_{s_0 + \overline \mu}^{k_0, \gamma}$, and 
\begin{equation}\label{conv vn v infty}
\| v_\infty \|_{s_0 + \overline \mu}^{k_0, \gamma} \lesssim \e \gamma^{- 1}, \quad \| v_\infty - v_n \|_{s_0 + \overline \mu}^{k_0, \gamma} \lesssim \e \gamma^{-1} N_{n-1}^{-\mathtt a_1}, \quad \forall n \geq 1\,. 
\end{equation}
 By recalling \eqref{def cal G infty} and Proposition \ref{iterazione-non-lineare}-$({\cal P}2)_n$, for any $\lambda \in {\cal G}_\infty$, ${\cal F}(v_n) \to 0$ as $n \to \infty$, therefore, the estimate \eqref{conv vn v infty} implies that ${\cal F}(v_\infty) = 0$ for any $\lambda \in {\cal G}_\infty$. By setting $\Omega_\e := {\cal G}_\infty$, by applying Proposition \ref{prop measure estimate finale} and using that $\gamma = \e^c$, one gets that $\lim_{\e \to 0} |\Omega_\e| = |\Omega|$ and hence the proof is concluded. 
 
 \bigskip
 
 \noindent
 {\sc Proof of Theorem \ref{main theorem 2}.} 
Let $v \in H^s_0$ be the solution of the equation $\mF(v) = 0$ 
provided by Theorem \ref{main theorem 1}, 
where $\mF$ is defined in \eqref{equazione cal F vorticita}. 
Let $u := \mU(v) = \curl (\Lm^{-1} v)$. 
One has 
$\div u = 0$, $\div F = \div (\curl f) = 0$
(the divergence of any curl is zero).
Since $\div u = 0$, one directly calculates
$\div( u \cdot \grad v - v \cdot \grad u) 
= u \cdot \grad (\div v)$.
Moreover $\div \Pi_0^\bot = \div$ because $\div \Pi_0 = 0$. 
Hence, taking the divergence of the identity $\mF(v) = 0$, 
we get
\begin{equation}\label{eq div v}
(\omega \cdot \partial_\vphi +\zeta \cdot \nabla +  \e u \cdot \nabla)(\div v) 
= 0. 
\end{equation}
Since $u = \curl (\Lm^{-1} v) \in H^{s_0 + \overline \mu} \cap X$ 
and $\| u \|_{s_0 +\overline \mu} \lesssim 1$, for $\e$ small enough 
we can apply Proposition \ref{proposizione trasporto} to the transport operator 
$$
{\cal T} = \omega \cdot \partial_\vphi + \zeta \cdot \nabla + \e u(\vphi, x) \cdot \nabla
$$
and obtain that there exists a reversibility preserving, invertible map ${\cal A}$ of the form \eqref{def mappa trasporto} such that 
$$
{\cal A}^{- 1} (\omega \cdot \partial_\vphi + \zeta \cdot \nabla + \e u \cdot \nabla) {\cal A} 
= \omega \cdot \partial_\vphi + \zeta \cdot \nabla , \quad \ 
\forall (\omega, \zeta) \in \Omega_\e \subseteq DC(\gamma, \tau). 
$$
Hence, by \eqref{eq div v}, $ (\omega \cdot \partial_\vphi + \zeta \cdot \nabla) {\cal A}^{-1} (\div v) = 0$.
Since $(\om,\zeta)$ is Diophantine, 
the kernel of $\ompaph + \zeta \cdot \grad$ is given by the constants, 
so that ${\cal A}^{-1} (\div v)$ is a constant,  
say ${\cal A}^{-1} (\div v) = c$. This implies that $\div v = c$, because ${\cal A}^{-1}$ is a change of variable. 
By periodicity, the space average of any divergence is zero, 
and therefore $\div v = c = 0$.

Since $\div u = 0$, $\div v = 0$, integrating by parts one deduces that 
$u \cdot \grad v$ and $v \cdot \grad u$ have both zero average
in the space variable $x \in \T^3$. 
Also $F$ has zero average in $x$ (because $F$ is a curl). 
This implies that $\Pi_0^\bot( u \cdot \grad v - v \cdot \grad u - F) 
= u \cdot \grad v - v \cdot \grad u - F$, 
and then from the equation $\mF(v) = 0$ we deduce that
\[
\ompaph v + \zeta \cdot \grad v + \e (u \cdot \grad v - v \cdot \grad u - F) = 0.
\]
Moreover, by Theorem \ref{main theorem 1}, 
$v \in H^s_0$ has zero space average,
therefore $\Lm^{-1} v = (- \Delta)^{-1} v$ (see \eqref{def pi0 Lm Lm inv})
and $u = \curl (-\Delta)^{-1} v$. 
Restore the tilde that have been removed when passing 
from \eqref{2803.1} to \eqref{equazione cal F vorticita}, 
and consider $v = \tilde\e \tilde v$. 
Then $v$ solves \eqref{equazione vorticita} 
and $\| v \|_s = \tilde \e \| \tilde v \|_s \leq C \tilde\e^{1+a} = C \e^b$, 
$b := (1+a)/2$.

Now we show that there exists a pressure $p(\vphi, x)$ such that $(u, p)$ 
solves \eqref{Eulero3}. 
Recall the general formula $\curl(\curl w) = - \Delta w + \grad (\div w)$.
Then $\curl u = \curl (\curl (- \Delta)^{- 1} v) = v$
because $\div v = 0$ and $\Pi_0 v = 0$. 
Then a direct computation shows that 
$$
{\rm curl}(u \cdot \nabla u) = u \cdot \nabla v - v \cdot \nabla u.
$$
Since $\curl (\omega \cdot \partial_\vphi u) = \omega \cdot \partial_\vphi v$, 
$\curl(\zeta \cdot \nabla u) = \zeta \cdot \nabla v$ and $F = \curl f$, 
the equation \eqref{equazione vorticita} can be rewritten as 
$$
\curl (\Gamma) = 0 \quad \text{ where } \ \ 
\Gamma:=  \omega \cdot \partial_\vphi u + \zeta \cdot \nabla u 
+ \e u \cdot \nabla u - \e f. 
$$
Then $\Gamma$ is a smooth irrotational vector field. 
We observe that $\Gamma$ has zero average. 
Indeed $\Pi_0 u = 0$ (because $u$ is a curl), 
$\Pi_0 f = 0$ (by assumption), 
and, integrating by parts,  
$\Pi_0 (u \cdot \nabla u ) = - \Pi_0 ((\div u)\, u) = 0$
(because $\div u = 0$), 
whence we deduce that $\Pi_0 \Gamma = 0$. Then $p :=  (- \Delta)^{- 1} \div \Gamma$ satisfies $\nabla p = \Gamma$, namely
$$
\omega \cdot \partial_\vphi u + \zeta \cdot \nabla u 
+ \e u \cdot \nabla u - \e f = - \nabla p. 
$$
Hence $u, p$ solve the equation \eqref{Eulero2} and $\| u \|_s, \| p \|_s \lesssim_s \e^b$.

\appendix

\section{Appendix}
In this appendix we prove a lemma which allows to provide the measure estimate of the resonant sets defined in \eqref{def cal R vn l j j'}.
Let ${\cal H}$ be a Hilbert space of dimension $d$ with scalar product $\langle \cdot, \cdot \rangle_{{\cal H}}$ 
and let $\{ \vphi_1, \ldots, \vphi_d \}$ be an orthonormal basis of ${\cal H}$. 
We consider the set $\mB(\mH)$ of linear maps $B : \mH \to \mH$ 
with the operator norm $\| \ \|_\op$. 
Given $B \in \mB(\mH)$, we identify it with its matrix representation 
$(B_j^i)_{j,i = 1, \ldots, d}$ 
where $B_{j}^i  := \langle B \vphi_i \,,\, \vphi_j \rangle_{{\cal H}}$.
Let  
\begin{equation}\label{forma A lambda}
A(\lambda) = \lambda  \cdot k \,{\rm Id}_{\cal H} + Q(\lambda) 
\quad \text{with } \ k \in \Z^{p} \setminus \{ 0 \},
\end{equation} 
where the map  
$\R^p \to {\cal B}({\cal H})$, $\lambda \mapsto Q(\lambda)$ 
is $k_0$ times differentiable.
Let $\Omega \subseteq \R^p$ be a bounded, open set and $\eta > 0$. In the lemma below, we provide a measure estimate of the set 
\begin{equation}\label{risonante matrice A}
{\cal R}_A := \big\{ \lambda \in \Omega : 
A(\lambda) \text{ is not invertible or it is invertible and } 
\| A(\lambda)^{- 1} \|_\op 
> \eta \, \big\}\,. 
\end{equation}
A similar version of the following Lemma can be found in \cite{P1}. For sake of completeness we insert our own proof. 
\begin{lemma}\label{lemma astratto misura risonante}
Let $k_0, d \in \N$ with $k_0 \geq d + 2$, $\gamma \in (0, 1)$. 
Then there exists $\delta_0 \in (0, 1)$ such that if 
\begin{equation}\label{Montparnasse0}
\| Q \|^{k_0, \gamma}_\op 
\leq \delta, \qquad 
\delta \gamma^{ - k_0} \leq \d_0,
\end{equation} 
then the Lebesgue measure of the set ${\cal R}_A$ satisfies 
$|{\cal R}_A| \lesssim (|k| \eta)^{-\frac{1}{d}}$.
\end{lemma}

\begin{proof}
First, if $A$ is invertible, one has the elementary inequality 
$\| A^{- 1} \|_{{\cal B}({\cal H})} 
\leq C_0(d) \| A\|_{{\cal B}({\cal H})}^{d - 1} | \det A |^{-1}$, 
for some $C_0(d) > 0$. 
By \eqref{forma A lambda}, \eqref{Montparnasse0}, 
since $\d \leq 1 \leq |k|$, 
one has $\| A \|_{{\cal B}({\cal H})}^{d - 1} \leq C_1 |k|^{d - 1}$
for some $C_1$ depending on $d,\Om$. 
Then $\| A^{- 1} \|_{{\cal B}({\cal H})} 
\leq C_2 |k|^{d - 1} |\det A|^{-1} $ 
for some $C_2 > 0$. 
Hence 
\begin{equation}\label{def widetilde cal R A}
{\cal R}_{A} \subseteq \widetilde{\cal R}_{A} 
\quad \text{ where } \ 
\widetilde{\cal R}_{A} := \Big\{ \lambda \in \Omega : 
|\det A(\lambda)| < \frac{C_2 |k|^{d - 1}}{\eta}\Big\}\,. 
\end{equation}
By using \eqref{forma A lambda}, \eqref{Montparnasse0}, a direct calculation shows that 
\begin{equation}\label{determinante polinomio}
\det A(\lambda) 
= (\lambda \cdot k)^d + r_{d - 1}(\lambda) (\lambda \cdot k)^{d - 1} + \ldots + r_1(\lambda) ( \lambda \cdot k) + r_0( \lambda)
\end{equation}
where for any $n = 0, \ldots, d - 1$ the maps $\lambda  \mapsto r_n(\lambda )$ 
are $k_0$ times differentiable and 
\begin{equation}\label{stima rn stima misura}
|r_n|^{k_0, \gamma }  \lesssim_{k_0, d}  \delta\,, \quad \forall n = 0, \ldots, d - 1\,. 
\end{equation}
Now let 
\begin{equation}\label{insalata di rinforzo 0}
\begin{aligned}
& \lambda  = \frac{k}{|k|} s + v, \quad v \cdot k = 0\,, \quad  \Gamma (s) := {\rm det}\Big( A \Big( \frac{k}{|k|} s + v \Big) \Big)\,, \\
& q_n(s) := r_n\Big(\frac{k}{|k|}s + v \Big), \quad n = 0, \ldots, d - 1\,. 
\end{aligned}
\end{equation}
By \eqref{determinante polinomio}, \eqref{insalata di rinforzo 0}, using that $\lambda  \cdot k = |k| s$, one has  
$$
\Gamma (s) = |k|^d  \big( s^d + q_\Gamma(s) \big), \quad q_\Gamma(s) := \frac{q_{d - 1}(s)s^{d - 1}}{|k|} + \ldots + \frac{q_1(s) s}{|k|^{d - 1}} + \frac{q_0(s)}{|k|^d} \,.
$$
By \eqref{stima rn stima misura}, \eqref{insalata di rinforzo 0}, 
one gets $|q_\Gamma|^{k_0, \gamma} \lesssim_{k_0, d} \delta$, implying that 
$$
|\partial_s^{d} \Gamma(s)| 
\geq \Big( d! - |\partial_s^{d} q_\Gamma(s)| \Big)|k|^d  
\geq \Big( d! - C \delta \gamma^{- k_0} \Big)|k|^d 
\geq \frac{d!}{2} |k|^d
$$
by taking $\delta \gamma^{- k_0}$ small enough. The latter estimate together with Theorem 17.1 in \cite{Russman} implies that 
$$
\Big| \Big\{ s :  \lambda  = \frac{k}{|k|}s + v \in \Omega, \quad |\Gamma(s)| < \frac{C_2(d) |k|^{d - 1}}{\eta}  \Big\} \Big| \lesssim_d \Big(\frac{1}{|k| \eta} \Big)^{\frac{1}{d}}\,.
$$
By a Fubini argument one gets that 
$|\widetilde{\cal R}_{A}| \lesssim ( |k| \eta )^{-\frac{1}{d}}$ 
and the claimed statement follows by recalling \eqref{def widetilde cal R A}. 
\end{proof}

\bigskip

\begin{flushright}
\textbf{Pietro Baldi}

\smallskip

Dipartimento di Matematica e Applicazioni ``R. Caccioppoli''

Universit\`a di Napoli Federico II

Via Cintia, Monte S.\ Angelo

80126 Napoli, Italy

\smallskip

\texttt{pietro.baldi@unina.it}

\bigskip

\textbf{Riccardo Montalto}

\smallskip

Dipartimento di Matematica ``Federigo Enriques''

Universit\`a degli Studi di Milano

Via Cesare Saldini 50

20133 Milano, Italy

\smallskip

\texttt{riccardo.montalto@unimi.it}
\end{flushright}

\end{document}